\documentclass[12pt, reqno]{amsart}
\usepackage{amsfonts, amssymb}
\usepackage{amsmath, amscd}
\usepackage{color}
\usepackage[colorlinks=true, pdfstartview=FitV, linkcolor=blue, citecolor=blue, urlcolor=blue]{hyperref}
\usepackage{tikz}
\usepackage{graphicx}
\usepackage{mathtools}
\usetikzlibrary{arrows, decorations.markings}
\usetikzlibrary{decorations.markings, arrows.meta}
 \usepackage[all]{xy}
\newlength{\bibitemsep}
\newlength{\bibparskip}
\let\oldthebibliography\thebibliography
\renewcommand\thebibliography[1]{%
  \oldthebibliography{#1}%
  \setlength{\parskip}{\bibitemsep}%
  \setlength{\itemsep}{\bibparskip}%
}

\theoremstyle{plain}
\newtheorem{theorem}{Theorem}[section]
\newtheorem{lemma}[theorem]{Lemma}
\newtheorem{proposition}[theorem]{Proposition}

\theoremstyle{definition}
\newtheorem{remark}[theorem]{Remark}
\newtheorem{definition}[theorem]{Definition}

\newtheorem{example}[theorem]{Example}

\newcommand{\nc}{\newcommand}
\nc{\delete}[1]{}
\nc{\red}{\textcolor{red}}
\nc{\blue}{\textcolor{blue}}
\nc{\cal}[1]{{\mathcal #1}}
\nc{\mf}[1]{{\mathfrak #1}}
\nc{\mc}[1]{\mathcal{#1}}
\nc{\bb}[1]{\mathbb{#1}}
\nc{\f}[1]{\mathfrak{#1}}
\nc{\mbf}[1]{\mathbf{#1}}
\nc{\msf}[1]{\mathsf{#1}}
\nc{\scr}[1]{\mathscr{#1}}
\nc{\on}[1]{\operatorname{#1}}
\nc{\onsf}[1]{\operatorname{{\sf #1}}}
\nc{\onbf}[1]{\operatorname{{\bf #1}}}
\font\cyr=wncyr10
\nc{\sha}{{\mbox{\cyr X}}}
\nc{\udim}{\underline{\onsf{dim}}}
\nc{\guvw}{g_{U,V}^{W}}
\nc{\htt}{\onsf{ht}}
\nc{\gr}{\onsf{gr}}
\nc{\cf}{\onsf{cf}}
\nc{\ad}{\onsf{ad }}
\nc{\Ad}{\onsf{Ad}}
\nc{\id}{\onsf{Id}}
\nc{\Fr}{\onsf{Fr}}
\nc{\Der}{\onsf{Der}}
\nc{\End}{\onsf{End}}
\nc{\tor}{\onsf{Tor}}
\nc{\Ext}{\onsf{Ext}}
\nc{\ext}{\onsf{ext}}
\nc{\Fil}{\onsf{Fil}}
\nc{\Hom}{\onsf{Hom}}
\nc{\grhom}{\onsf{grHom}}
\nc{\ch}{\onsf{ch}}
\nc{\Ch}{\onsf{Ch}}
\nc{\ind}{\onsf{Ind}}
\nc{\coind}{\onsf{Coind}}
\nc{\Mod}{\onsf{-Mod}}
\nc{\biMod}{\onsf{-biMod}}
\nc{\Poiss}{\onsf{-Poiss}}
\nc{\grmod}{\onsf{-grMod}}
\nc{\vamod}{\onsf{-VAMod}}
\nc{\res}{\onsf{Res}}
\nc{\soc}{\onsf{Soc}}
\nc{\rad}{\onsf{Rad}}
\nc{\Aut}{\onsf{Aut}}
\nc{\Dist}{\onsf{Dist}}
\nc{\Lie}{\onsf{Lie}}
\nc{\Ker}{\onsf{Ker}}
\nc{\im}{\onsf{Im}}
\nc{\wt}{\onsf{wt}}
\nc{\st}{\onsf{St}}
\nc{\diag}{\onsf{Diag}}
\nc{\rep}{\onsf{rep}}
\nc{\Set}{\onsf{Set}}
\nc{\sSet}{\onsf{sSet}}
\nc{\smCat}{\onsf{smCat}}
\nc{\spec}{\onsf{spec}}
\nc{\Sym}{\onsf{Sym}}
\nc{\Vecs}{\onsf{Vec}^{s}}
\nc{\colim}{\operatornamewithlimits{\underset{\longrightarrow}{lim}}}
\nc{\graphdot}{\onwithlimits{\bullet}}
\nc{\lrarrows}{\onwithlimits{\leftrightarrows}}
\nc{\interval}[1]{\mathinner{#1}}
\nc{\blist}{\begin{list}{\rom{(\roman{enumi})}}{\setlength{\leftmargin}{0em}
\setlength{\itemindent}{7ex}
\setlength{\labelsep}{2ex}\setlength{\listparindent}{\parindent}
\usecounter{enumi}}}
\nc{\elist}{\end{list}}
\nc{\subsub}[1]{\noindent{\bf #1}}

\usepackage{etoolbox}
\apptocmd{\thebibliography}{\setlength{\itemsep}{5pt}}{}{}

\allowdisplaybreaks

\makeatletter \def\l@subsection{\@tocline{2}{0pt}{1pc}{5pc}{}} \def\l@subsection{\@tocline{2}{0pt}{2pc}{2.5pc}{}} \makeatother

\title[dg vertex operator algebras]{Differential graded vertex operator algebras and their Poisson algebras}
\author[A. Caradot]{Antoine Caradot$^1$}
\address{$^1$School of Mathematics and Statistics\\
Henan University \\
Kaifeng, Henan, CHINA}
\email{caradot@henu.edu.cn}
\author[C. Jiang]{Cuipo Jiang$^2$}
\address{$^2$School  of Mathematics\\
Shanghai Jiao Tong University \\
Shanghai, CHINA}
\email{cpjiang@math.sjtu.edu.cn}
\author[Z. Lin]{Zongzhu Lin$^3$}
\address{$^3$Department of Mathematics\\
Kansas State University \\
Manhattan, KS 66506, USA}
\email{zlin@math.ksu.edu}
\date{\today}

\begin{document}
\maketitle

\begin{abstract} 
 In this paper, we define differential graded vertex operator algebras and the algebraic structures on the associated Zhu algebras and $C_2$-algebras. We also introduce the corresponding notions of modules, and investigate the relations between the different module categories. 
\end{abstract}

\tableofcontents

 \section{Introduction}
 There have been many versions of vertex algebras such as vertex super algebras and graded vertex algebras (\cite{Dong-Han}). The differential graded vertex algebras are natural combinations. The motivation for studying differential graded vertex algebras is from both homological algebra, and algebraic geometry. In algebraic geometry, computing the Donaldson-Thomas invariants in a series of work by Joyce, where the vertex algebras are actually cohomologies of certain moduli stacks \cite{Joyce, Latyntsev}. There are also quantum vertex algebras and vertex algebras over fields of positive characteristics. Differential graded vertex algebras are also a testing model for defining vertex algebras over more general braided monoidal categories (\cite{Joyal-Street}) with a goal of defining derived vertex algebras over triangulated symmetric tensor categories and over algebraic stacks. 
 
In this paper, we simply replace the symmetric tensor tensor category of $\msf k $-vector spaces (over a field $\msf k$ of characteristic zero) by the category of differential complexes of $\msf k$-vector spaces. This simple change already generates unexpected structures on the Zhu algebra, which is not only an associated algebra, but also has two filtrations, serving as a deformations of differential graded Poisson algebras (\cite{Lu-Wang-Zhang}). This paper is focused over the abelian category of differential complexes, where the morphism space and internal homomorphism objects are differential in contrast to the classical category of vector spaces. We will study the relationship between the associated $C_2$-algebras and Zhu algebras in the category of differential filtered algebras, which is another model for replacement of the category of vector spaces. 
 
One of the main point of the categorical approach is to understand and interpret the loop space, or rather its functor algebra $\msf k[t, t^{-1}]$. In fact, this algebra should be understood as $\bb A^1 $ with the  divisor $\{0\}$, while the vector operators should be understood as continuous distributions. In this way, both the $C_2$-algebra and the Zhu algebra will carry geometric interpretations. This will formulated in a different work. 
 
The outline of the paper is as follows: Section~\ref{sec:2} recalls basic concepts of the symmetric tensor category of differential category as well as the loop complexes. Since the Zhu algebra is not a differential graded (dg) algebra, we introduce the concept of differential filtered algebras (algebra objects in the symmetric tensor category of differential filtered vector spaces). The Zhu algebra has another natural filtration (by the weights) making it a bifiltered algebra. Furthermore, it is almost commutative (cf. \cite{Chriss-Ginzburg}), which makes the associated graded algebra a dg poisson algebra. We will also need the concept of representations of these algebra objects in their appropriated tensor category. 
 
Section~\ref{sec:3} is devoted to defining dg vertex algebras as well as their properties. For each dg vertex algebra, its cohomology is automatically a graded vertex algebra. It is a $\bb Z$-graded lift of a super vertex algebra. We will not give many examples in this paper. In an up-coming paper (\cite{Caradot-Jiang-Lin5}), we define dg vertex Lie algebras from which one applies the standard construction to construct dg vertex operator algebras from any dg Lie algebra. 
 
Section~\ref{sec:4} defines various module categories which is reflected by which tensor category we are viewing the vertex algebras in. 
 
In Section~\ref{sec:5}, we construct the $C_2$-algebra $R^{[*]}(V^{[*]})$ of dg vertex algebra $V^{[*]}$, which is naturally a dg Poisson algebra. There is a natural functor from the category of dg vertex algebra modules to the category of Poisson modules of the $C_2$-algebra. 
 
The Zhu algebra $A(V^{[*]})$ of a dg vertex algebra is studied in Section~\ref{sec:6}. This algebra is much more complicated than expected. The dg structure was destroyed in the construction as the differential gradation it not transferred to the quotient, largely due the fact that differential degree and conformal weights are not separated as we regarded $t$ (coming from the loop algebra) as have cohomological degree 2 and conformal weight 1. In the definition, one could define $t$ as having cohomological degree $2k$. The choice of $k=1$ was from Joyce's construction of vertex algebras from cohomological Hall algebras \cite{Joyce, Latyntsev}, where the algebra $ \msf k[t]=H^*_{G_m}(\{pt\})$ with $t$ in the cohomological degree 2.  If $k=0$, then $A(V^{[*]})$ would be a filtered associated dg algebra which is almost commutative.  For $k=1$, $A(V^{[*]})$ has both a cohomological filtration and weight filtration. This algebra and its module categories play the role of the category of spectral sequences. 
 
 Section~\ref{sec:7} is devoted to the study of the left, right, and bifiltered modules of the bifiltered Zhu algebra, as well as their associated graded algebra and modules. The associated weight graded dg algebra is independent of the order in which we carry out the associated graded constructions. This associated graded construction is exactly that of taking limits in dealing with spectral sequences. 
 
In Section~\ref{sec:8}, we compare the graded dg Poisson algebras $R^{[*]}(V^{[*]})$ and the associated (graded) dg Poisson algebra $\gr(A(V^{[*]}))$ - as well as their modules.  What we get is a diagram in the 2-category of categories.  
 
As mentioned earlier, the category of differential complexes we consider is an abelian category. It has a natural structure of $2$-category with $2$-morphisms being homotopy equivalences. Thus one can naturally work in the homotopy category by defining homotopy Lie algebras and homotopy vertex algebras, as homotopy Poisson algebras and repeat most of the constructions in this paper.   We will leave this in future work deal this general case.  

In \cite{Caradot-Jiang-Lin1} and \cite{Caradot-Jiang-Lin2}, given a presentation of $R(V)$, we constructed the Tate resolution of the trivial $R(V)$-module, which is a graded commutative dg algebra. Its degree zero part is the polynomial algebra $R(V)$, and other degrees are free modules over the degree $0$ part. In the dg setting, $R^{[*]}(V^{[*]})$ is a graded commutative dg algebra, and one can construct the Tate resolution $T(R^{[*]}(V^{[*]}))$ of $R^{[*]}(V^{[*]})$. We conjecture that there should be a Tate resolution $T(V^{[*]})$ of a dg vertex algebra, whose $C_2$-algebra is a Tate resolution of a dg algebra. Now, replacing the polynomial algebra should be the universal dg vertex algebra corresponding to a dg Lie algebra. We expect that this construction will be a version of BGG resolution in the category of rational affine vertex algebras. This should provide a resolving dg vertex algebra, as comparison with a resolution of dg schemes in \cite{Behrend}.

\delete{Given a vertex operator algebra $(V, Y, \mathbf 1, \omega)$, one can attach two associative algebras. One  is the Zhu algebra $A(V)$ and the other is the Poisson algebra $R(V)$. In this paper, we compare the representation theories of the vertex operator algebra $V$ with those of the Zhu algebra $A(V)$ and the $C_2$-algebra $R(V)$.  The motivation of such comparison arises from a geometric point of view. The goal is to see what analogous roles the two associative algebras $R(V)$ and $A(V)$ play in terms of representation theory, and how much information about the representation theory of the vertex operator algebra is preserved by the representation theories of these two associative algebras. 

One prototypical example is the universal affine vertex operator algebra $ V=V_k(\mathfrak g)$ of a simple Lie algebra $\mathfrak g$ at a non-critical level $k$. In this case, the associative algebras satisfy $A(V)=U(\mathfrak g)$ and $R(V)=\operatorname{Sym}^*(\mathfrak g)$. The Zhu algebra $ A(V)$ always carries a filtered algebra structure with an increasing filtration:
 \[ \cdots \subseteq F^pA(V)\subseteq  F^{p+1}A(V)\subseteq \cdots \] 
 of vector subspaces. With respect to this filtered algebra structure, $A(V)$ is almost commutative in the sense that:
  \[ [F^pA(V), F^qA(V)]\subseteq F^{p+q-1}A(V).\] 
Thus the associated graded algebra $ \gr A(V)$ is commutative and automatically carries a Poisson algebra structure with the Poisson bracket $\{ . , . \}$, which has degree $-1$.  In general there is surjective homomorphism of Poisson algebras $ R(V)\longrightarrow  \gr A(V)$. However this homomorphism is not an isomorphism.  

From the deformation theory viewpoint, $A(V)$ is a deformation of $\gr A(V)$. In the case of $ V=V_k(\mathfrak g)$, we have the isomorphism $R(V)=\gr A(V)=\operatorname{Sym}^*(\mathfrak g)$, the algebra of regular functions on the Poisson variety $\mathfrak g^*$. This suggests that $A(V)$ plays the role of the algebra of differential operators while $R(V)$ plays the role of function algebra of the cotangent bundle.   

Comparisons of the representation theory of $V$ and that of $A(V)$ was the main motivation in Zhu's thesis to define and study the Zhu algebra $A(V)$. For example, there is a one-to-one correspondence between the irreducible admissible representations of $V$ and the irreducible representations of $A(V)$. More functorial studies of these irreducible representations as well as specific constructions of irreducible modules involving higher Zhu algebras were carried out in \cite{Dong-Jiang} and \cite{Dong-Li-Mason2}. The study of the fusion product were also carried out in \cite{Dong-Ren} and \cite{Huang-Yang}. Recently, Huang has packed all those higher Zhu algebras as well as those bimodules in \cite{Dong-Jiang} to form a large associative algebra in order to study the fusion product of $V$. The goal of this paper is, purely from the representation theory viewpoint, to compare certain representation theoretic invariants of these algebras. 

As already observed the more functorial version of the comparison between representations, the $V$-modules should corresponds to $A(V)$-bimodules. On the other hand, the $C_2$-algebra $R(V)$ is a Poisson algebra and the representation theory should reflect the Poisson structure. The category of modules we will consider is the Poisson modules. We will establish functorial relations among these these different algebras. }

\newcommand{\cc}{\ensuremath{\mathbb{C}}}

\newcommand{\lefttorightarrow}{
\hspace{.1cm}
{
\setlength{\unitlength}{.30mm}
\linethickness{.09mm}
\begin{picture}(8,8)(0,0)
\qbezier(7,6)(4.5,8.3)(2,7)
\qbezier(2,7)( - 1.5,4)(2,1)
\qbezier(2,1)(4.5, - .3)(7,2)
\qbezier(7,6)(6.1,7.5)(6.8,9)
\qbezier(7,6)(5,6.1)(4.2,4.4)
\end{picture}
\hspace{.1cm}
}}

\section{Basics of the category of differential graded complexes} \label{sec:2}

The main reference of this section is Chapter 12 of \cite{Stacks-Project}.

\subsection{The categories of super braided graded vector spaces and graded differential (cochain) complexes} 
Let $ \Vecs$ be the symmetric tensor category of graded $\bb C$-vector spaces. Set $V^{[*]}=\oplus_{p\in \bb Z} V^{[p]}$  with the Koszul braiding  $T_{V^{[*]},U^{[*]}}: V^{[*]}\otimes U^{[*]}\to U^{[*]}\otimes V^{[*]}$ defined by 
\[ 
T_{V, U}(v\otimes u)=(-1)^{|v||u|} u\otimes v
\]
for all homogeneous vectors $v\in V^{[|v|]}$ and $ u\in U^{[|u|]}$ of degrees $ |v|$ and $ |u|$ respectively. 

Given two graded vector spaces $ V^{[*]}=\oplus_{n}V^{[n]}$ and $ U^{[*]}=\oplus_n U^{[n]}$, the homomorphisms in $\Vecs$ is defined to be $ \Hom^{[0]}_{\bb C}(V^{[*]}, U^{[*]})$ consisting linear maps $f: V^{[*]}\to U^{[*]} $ such that $ f(V^{[n]})\subseteq U^{[n]}$. 

We denote the internal hom space $ \Hom_{\Vecs}^{[*]} (V^{[*]}, U^{[*]})=\oplus_p \Hom_{\bb C}^{[0]}(V^{[*]}, U[p]^{[*]})$. Here $U[p]^{[*]}$ is the same vector space $U^{[*]}$  with grading defined by $ U[p]^{[n]}=U^{[n+p]}$. We remark that if $V$ has only finitely many non-zero homogeneous components, we have $ \Hom_{\bb C}^{[*]}(V^{[*]}, U^{[*]})=\Hom_{\bb C}(V^{[*]}, U^{[*]})$ as $\bb C$-vector spaces for all graded vector spaces $U^{[*]}$. If $U$ is not graded, the equality does not hold and we will use the notation $\Hom_{\bb C}^{[*]}$ to distinguish with nongraded $\Hom_{\bb C}$. In particular $\on{End}^{[*]}_{\bb C} (V^{[*]})=\Hom_{\bb C}^{[*]}(V^{[*]}, V^{[*]})$ is a graded associative algebra.

Let $ \Ch $ be the category of differential (cochain) complexes of $\bb C$-vector spaces, with objects $(V^{[*]}, d^{[*]})$, where $ V^{[*]}$ is an object in $\Vecs$ and $ d^{[*]}\in \Hom^{[1]}(V^{[*]}, V^{[*]})$ such that $ d^{[*]} \circ d^{[*]}=0$. Here $ d^{[p]}: V^{[p]}\to V^{[p+1]}$ is a linear map. The morphisms $ f:(V^{[*]}, d_V^{[*]}) \to (W^{[*]}, d_W^{[*]})$  in $\Ch$ are chain maps, i.e., $f=(f^{[p]}: V^{[p]}\to W^{[p]})$ such that $ d_W^{[p]}\circ f^{[p]}=f^{[p+1]}\circ d_V^{[p]}$. We use $ \Ch((V^{[*]}, d_V^{[*]}) , (W^{[*]}, d_W^{[*]}))$ to denote the set of all chain maps. It is a vector space, making $ \Ch$ an abelian category with small limit and colimit.  We remark that $\Ch((V^{[*]}, d_V^{[*]}) , (W^{[*]}, d_W^{[*]}))\subseteq \Hom^{[0]}(V^{[*]}, W^{[*]})$ and the equality does not hold in general. 

We define $Z^{[p]}(V^{[*]})=\on{Ker}(d_V^{[p]}: V^{[p]}\to V^{[p+1]})$ and call its elements $p$-cocycles. Then $Z^{[*]}(V^{[*]})\subseteq V^{[*]}$ is subcomplex with zero differential. Similarly, $ B^{[p]}(V^{[*]})=d^{[p-1]}(V^{[p-1]})$ defines a subcomplex $B^{[*]}$ of $ Z^{[*]}$, and the cohomology complex denoted by $H^{[*]}(V^{[*]}, d_V^{[*]})=Z^{[*]}/B^{[*]}$ is also a complex with zero differential. We note that $Z^{[*]}$, $B^{[*]}$, and $H^{[*]}$ are additive functors from $ \Ch$ to $\Ch$ with image in the full subcategory $\Vecs$.

The degree shift functor $[n]: \Ch\to\Ch$ is defined by $(V^{[*]}, d^{[*]})\to (V[n]^{[*]}, d[n]^{[*]})$ by $ V[n]^{[p]}=V^{[p+n]}$ with the differential $ d[n]^{[p]}=(-1)^n d^{[n+p]}$. Clearly, $[n]: \Ch\to\Ch$ is an additive endofunctor on $ \Ch$ and $[n]\circ [m]=[n+m]$.  It follows that $ H^{[p]}((V[n]^{[*]}, d[n]^{[*]})=H^{[p+n]}((V^{[*]}, d^{[*]})$, but the natural isomorphism is not the identity map.

\subsection{Tensor product and internal homomorphisms}
The tensor product in $ \Ch$ is given as 
\[ (V^{[*]}, d_V^{[*]})\otimes (U^{[*]}, d_U^{[*]})=(V^{[*]}\otimes U^{[*]}, d^{[*]}_{V\otimes U} )
\]
with $ d_{V\otimes U}(v\otimes u)=d_V(v)\otimes u+(-1)^{|v|}v\otimes d_U(u)$ for homogenous $v\in V^{[*]}$ and $ u\in U^{[*]}$.

The braiding is the same as in $\Vecs$ by noticing
\[ 
d_{U\otimes V}(T_{V, U}(v\otimes u))=(-1)^{|v||u|}(d_U(u)\otimes v+(-1)^{|u|}u\otimes d_V(v))=T_{V, U}(d_{V\otimes U}(v\otimes u)).
\]
So $ T_{V, U}$ is a chain map, and thus a morphism in the category $ \Ch$. Other conditions for making $ \Ch$ into a strict symmetric tensor category over $ \bb C$ can be verified as exercises. 

Let $ \bb C$ be the complex with $\bb C$ at degree zero. Then $ \bb C[n]$ is a complex with $\bb C$ placed in degree $-n$. The functor $ [n]$ can be represented by $\bb C[-n]$ (cf. \eqref{hom-shift} below). Indeed, we have $V[n]^{[*]}=\Hom^{[*]}(\bb C[-n],V^{[*]})=(\bb C[n]) \otimes V^{[*]}$.

Similarly $\Hom^{[*]}(V^{[*]}, W^{[*]})$ with the differential defined by
\begin{align}\label{eq:diff_end}
 d(f)(v)=d_{W}(f(v))-(-1)^{|f|}f(d_V(v))
\end{align}
for all $ f\in \Hom^{[*]}(V^{[*]}, W^{[*]})$ and $ v\in V^{[*]}$ homogeneous, is also an object in $\Ch$. 

The standard adjointness linear maps induces 
\[ 
\Ch((U^{[*]}, d^{[*]}_U)\otimes (V^{[*]}, d^{[*]}_V), (W^{[*]}, d^{[*]}_W))
=\Ch((U^{[*]}, d^{[*]}_U), \Hom_{\bb C}^{[*]}(V^{[*]}, W^{[*]})).
\] 
Even more generally we have 
\begin{align}\label{hom-tensor-adjoint}
 \Hom^{[*]}(U^{[*]} \otimes V^{[*]}, W^{[*]})
=\Hom^{[*]}(U^{[*]}, \Hom^{[*]}(V^{[*]}, W^{[*]})).
\end{align}
Furthermore
\begin{align}\label{hom-shift}
 \Hom_{\bb C}^{[*]}(V^{[*]}[-n], W^{[*]})=\Hom_{\bb C}^{[*]}(V^{[*]}, W^{[*]}[n])=\Hom_{\bb C}^{[*]}(V^{[*]}, W^{[*]})[n].
 \end{align}
This can also be seen as a special case of \eqref{hom-tensor-adjoint} for the special case $U^{[*]}=\bb C[-n]$ (which is a compact object in $\Ch$, having finite dimensional total space).

\subsection{Associative dg algebras and dg Lie algebras}
For any object $(V^{[*]}, d_V^{[*]})$, $\End_{\bb C}^{[*]}(V^{[*]})$ is also an object in $\Ch$ making $\End^{[*]}(V^{[*]})$ an associative dg algebra  (cf. \cite{Lu-Wang-Zhang}).  We will use the graded commutator (called super commutator) in $\End^{[*]}(V^{[*]})$ for homogeneous operators $ f, g$ given by 
\[
[f, g]^s=f\circ g-(-1)^{|f||g|}g\circ f.
\]

Then $\End^{[*]}(V^{[*]})$ becomes a dg Lie algebra of degree $0$.  A dg Lie algebra (of degree zero) is an object $(L^{[*]}, d_L^{[*]})$ with a zero cocycle $[-, -]\in Z^{[0]}(\Hom^{[*]}(L^{[*]}\otimes L^{[*]}, L^{[*]}))$, i.e., 
\[ d_L[x,y]=[d_L(x), y]+(-1)^{|x|}[x, d_L(y)] \]
satisfying
\begin{enumerate}\setlength\itemsep{5pt}
\item[(i)] (Skew symmetry) $  [x, y]=-(-1)^{|x||y|}[y, x] $;
\item[(ii)] (Jacobi identity)
\[ [x, [y,z]]=[[x,y], z]+(-1)^{|x||y|}[y, [x, z]]\]
for all homogeneous $x, y, z\in L^{[*]}$.
\end{enumerate}
Note that $ \End^{[*]}(L^{[*]})$ is a dg Lie algebra. Set $ \ad: L^{[*]} \to \End^{[*]}(L^{[*]})$ defined by $ x\to \ad(x)=[x, -]$. The Jacobi identity is equivalent to 
\[ \ad([x,y])=\ad(x)\circ \ad(y)-(-1)^{|\ad(x)||\ad(y)|}\ad(y)\circ\ad(x).\]
Thus under the skew-symmetry, the bracket $[-,-]$ provides a dg Lie algebra structure on $L$ if and only if $\ad: L^{[*]} \to \End^{[*]}(L^{[*]})_{Lie}$ is a dg Lie algebra homomorphism (of degree 0).  

In general (cf. \cite{Lu-Wang-Zhang}), a dg Lie algebra of degree $p$ is an object $(L^{[*]}, d_L^{[*]})$ in $\Ch$ together with  $ [-, -]\in Z^{[0]}(\Hom^{[*]}(L[-p]^{[*]} \otimes L[-p]^{[*]}, L[-p]^{[*]})$ making $ (L[-p]^{[*]}, [-,-]) $ dg Lie algebra of degree zero. Noting that for $ x\in L^{[*]}$ homogeneous, we have $ |x|=|x|_{L[n]}+n$, thus $ |x|+p=|x|_{L[-p]}$. It follows that we need the following conditions: \\
(i) Skew symmetric $  [x, y]=-(-1)^{(|x|+p)(|y|+p)}[y, x] $;\\
(ii) Jacobi identity
\[ [x, [y,z]]=[[x,y], z]+(-1)^{(|x|+p)(|y|+p)}[y, [x, z]]\]
for all homogeneous $x, y, z\in L^{[*]}$. 

Using the chain complex isomorphism
\[
\Hom^{[*]}(L[-p]^{[*]}\otimes L[-p]^{[*]}, L[-p]^{[*]})\cong \Hom^{[*]}(L^{[*]}\otimes L^{[*]}, L[p]^{[*]})
\]
defined by $ \phi\mapsto \widetilde{\phi}$ with $\widetilde{\phi}(x\otimes y)=(-1)^{p|x|}\phi(x\otimes y)$, we get an isomorphism 
\[
\begin{array}{rcl}
Z^{[0]}(\Hom^{[*]}(L[-p]^{[*]} \otimes L[-p]^{[*]}, L[-p]^{[*]})) & \stackrel{\cong}{\to} & Z^{[0]}(\Hom^{[*]}(L^{[*]}\otimes L^{[*]}, L[p]^{[*]})), \\[5pt]
 &  \stackrel{\cong}{\to} & Z^{[p]}(\Hom^{[*]}(L^{[*]}\otimes L^{[*]}, L^{[*]})).
 \end{array}
\]
The first isomorphism is the map $ \phi\mapsto \widetilde\phi $ and the second one is the identity map in \eqref{hom-shift}. Thus a dg Lie algebra of degree $p$ is an object $(L^{[*]},d_L^{[*]})$ with a $p$-cocycle in $Z^{[p]}(\Hom^{[*]}(L^{[*]}\otimes L^{[*]}, L^{[*]}))$ satisfying the Lie algebra relations.

Here $[-, -]\in Z^{[p]}(\Hom^{[*]}(L^{[*]}\otimes L^{[*]}, L^{[*]}))$ is a $p$-cocycle means $d_{L}\circ [-, -]=(-1)^p[-, -]\circ d_{L\otimes L}$, i.e.,
\[ d_L[x,y]=(-1)^p[d_L(x), y]+(-1)^{|x|+p}[x, d_L(y)]\]
for all $ x,y \in L$ homogeneous.  Define $[x,y]'=(-1)^{p|x|}[x,y]$. Then 
\[ d_L[x,y]'=[d_L(x), y]'+(-1)^{|x|+p}[x, d_L(y)]'.\]
Here $ |x|+p=|x|_{L[-p]}$. Then $ [-,-]'\in Z^{[0]}(\Hom^{[*]}(L[-p]^{[*]} \otimes L[-p]^{[*]}, L[-p]^{[*]}))$.  A homomorphism between two dg Lie algebras is a chain map satisfying the expected conditions.  

A dg module for a dg Lie algebra $ (L^{[*]}, d_L^{[*]})$ is an object in $(M^{[*]}, d_M^{[*]})$ in $\Ch$ together with a homomorphism of dg Lie algebras $(L^{[*]}, d_L^{[*]})\to \End^{[*]}(M^{[*]})$ (see \cite{Lu-Wang-Zhang} for details).

If follows from the definitions that we have 
\[ \Ch((V^{[*]}, d_V^{[*]}) , (W^{[*]}, d_W^{[*]}))=Z^{[0]}(\Hom^{[*]}(V^{[*]}, W^{[*]})).\]
Thus chain maps are the $0$-cocycles in $\Hom^{[*]}(V^{[*]}, W^{[*]})$. 

Clearly we have a functor $ \Vecs\to \Ch$ by taking the zero differential. There is also a cohomological functor $H^{[*]}: \Ch\to \Vecs$. The above mentioned functor $ \Vecs\to \Ch $ is a section of the cohomology functor. In this paper, we will not talk about the homotopy for vertex operator algebras, which will be dealt in a later paper.

Given a complex $(V^{[*]}, d_V^{[*]})$ in $ \Ch$ over a field $ \mbf k$,  the associative dg algebra $A^{[*]}= \End^{[*]}_{\mbf k}(V^{[*]})$  has a natural dg Lie algebra with the super Lie bracket (of degree 0) defined by 
\[ [f, g]^s=f\circ g-(-1)^{|f||g|}g\circ f\]
for homogeneous linear maps $f$ and $g$ of degrees $|f|$ and $|g|$ respectively. In fact $[-, -]^s\in Z^{[0]}(\Hom_{\mbf k}^{[*]}(A^{[*]}\otimes A^{[*]}, A^{[*]}))$ since 
\[ d_A [f, g]^s=[d_Af, g]^s+(-1)^{|f|}[f, d_A g]^s.\]

\delete{
The graded Lie algebra satisfies  (for homogeneous elements $ f, g, h\in \End_{\mbf k}^{[*]}(V)$)
\begin{itemize}\setlength\itemsep{5pt}
\item  $[f, g]^s+(-1)^{|f||g|}[g, f]^s=0$ (skew symmetric), and 
\item $ [f, [g, h]^s]^s=[[f, g]^s, h]^s +(-1)^{|f||g|}[g, [f, h]^s]^s$ (graded derivation). 
\end{itemize} 
}

Let $R^{[*]}=\bigoplus_{p}R^{[p]}$ be a dg algebra. The multiplication $m: R^{[*]} \otimes R^{[*]} \to R^{[*]}$ is a degree $0$ chain map (a morphism in  $\Ch$). We do not assume the multiplication to be associative. We define a graded derivation $D: R^{[*]} \to R^{[*]}$ of degree $p$ as an element $D\in \End^{[p]}(R^{[*]})$ satisfying 
\[ D(m(x\otimes y))=m(D(x)\otimes y)+(-1)^{p|x|}m(x\otimes D(y))
\]
for all $ x, y \in R$ homogeneous. We will simply write $xy=m(x\otimes y)$ in the following. 

Let $ \Der^{[*]}_{\mbf k}(R^{[*]})$ be the graded subspace of $\End_{\mbf k}^{[*]}(R^{[*]})$ generated by all graded derivations $ D$, and let $d$ be the differential of the complex $\End_{\mbf k}^{[*]}(R^{[*]})$. For $ D\in \Der_{\mbf k}^{[p]}(R^{[*]})$, we have
\begin{align*} d(D)(x y)&=d_R(D(x y))-(-1)^{p}D(d_R(x y))),\\
&=d_R(D(x)y+(-1)^{p|x|}x D(y))-(-1)^{p}D(d_R(x) y+(-1)^{|x|}x d_R(y)),\\
&=d_R(D(x)) y+(-1)^{p+|x|}D(x)d_R(y)+(-1)^{p|x|}d_R(x)D(y)\\
&\quad  -(-1)^{p}D(d_R(x)) y-(-1)^{p(|x|+2)}d_R(x)D(y) +(-1)^{(p+1)|x|}xd_R(D(y))\\
&\quad   -(-1)^{|x|+p}D(x) d_R(y)-(-1)^{p|x|+|x|+p}xDd_R(y)),\\
&=(d_R\circ D-(-1)^{p}D\circ d_R)(x)y+(-1)^{(p+1)|x|}(x(d_R\circ D-(-1)^{p}D\circ d_R)(y),\\
&=d(D)(x)y+(-1)^{(p+1)|x|}xd(D)(y).
\end{align*}
Hence $d(D)\in \Der^{[p+1]}(R^{[*]})$ and $\Der_{\mbf k}^{[*]}(R^{[*]})$ is subcomplex of $\End_{\mbf k}^{[*]}(R^{[*]})$. We leave it for the reader to verify that $\Der_{\mbf k}^{[*]}(R^{[*]})$ is a dg Lie subalgebra of $ \End_{\mbf k}^{[*]}(R^{[*]})$ of degree $0$ (cf.  \cite{Behrend}).

We remark that in \cite{Lu-Wang-Zhang}, one can define associative algebras with multiplication of other degrees, just like dg Lie algebras of other degrees. We will not go in that generality.

\subsection{dg Poisson algebras and dg Poisson modules}
We define the not necessarily commutative dg Poisson algebra (cf.  \cite{Lu-Wang-Zhang}) as dg associative algebra $A^{[*]}$ with a dg Lie algebra structure $\{-,-\}: A^{[*]} \otimes A^{[*]} \to A^{[*]}$ (of degree zero) such that the map $ A^{[*]}\to \End^{[*]}(A^{[*]})$ defined by $ a\mapsto \{a, -\}$ defines a dg Lie algebra homomorphism $ (A^{[*]}, \{-,-\})\to \Der^{[*]}_{\mbf k}(A^{[*]}_{asso})$, i.e., 
\begin{align*} 
\{a, b\}&=-(-1)^{|a||b|}\{b,a\}\\
 \{\{a, b\}, x\}&=\{a, \{b, x\}\}-(-1)^{|a||b|}\{b, \{a, x\}\}\\
 \{a, x y\}&=\{a, x\}y+(-1)^{|a||x|}x\{a, y\}.
\end{align*}

In this case, $\End^{[*]}(A^{[*]})_{Lie}$ has two dg Lie subalgebras $ \Der^{[*]}(A^{[*]}_{Asso})$ and $ \Der^{[*]}(A^{[*]}_{Lie})$. The second an third conditions say that the map $ \ad: A^{[*]} \to \End^{[*]}(A^{[*]})_{Lie}$ has the image in $\Der^{[*]}(A^{[*]}_{Asso})\cap \Der^{[*]}(A^{[*]}_{Lie})$, making it a dg Lie algebra homomorphism. 

\begin{remark}
If the Lie bracket on $A^{[*]}$ is given by the super commutator, then given a derivation $d \in  \Der^{[*]}(A^{[*]}_{Asso})$ and $x, y \in A^{[*]}$ homogeneous, we have
\begin{align*}
\begin{array}{rcl}
d([x,y]^s) &= & d(xy-(-1)^{|x||y|}y x), \\[5pt]
 & = & d(x)y+(-1)^{[x||d|}xd(y)-(-1)^{|x||y|}(d(y)x+(-1)^{|y||d|}yd(x)), \\[5pt]
 & = & d(x)y-(-1)^{(|x|+|d|)|y|}yd(x)+(-1)^{[x||d|}xd(y)-(-1)^{|x||y|}d(y)x, \\[5pt]
  & = & [d(x),y]^s+(-1)^{|x||d|}[x, dy]^s.
\end{array}
\end{align*}
Thus $d \in \Der^{[*]}(A^{[*]}_{Lie})$ and we have $\Der^{[*]}(A^{[*]}_{Asso}) \subset \Der^{[*]}(A^{[*]}_{Lie})$. However the other direction is unlikely, because the information on the product in $A_{Asso}$ is finer than the one on the Lie product in $A^{[*]}_{Lie}$. Hence a derivation of $A^{[*]}_{Lie}$ might not behave well enough with regards to the product in $A^{[*]}_{Asso}$.
\end{remark}

\begin{example} Let $A^{[*]}$ be an associative dg algebra in $ \Ch$, i.e., there is an associative multiplication $ A^{[*]} \otimes A^{[*]} \to A^{[*]}$ which is a chain map. Then $(A^{[*]}, d_A^{[*]})$ has a dg Lie algebra structure with Lie bracket defined by 
\[ [a, b]^s=ab-(-1)^{|a||b|}b a\]
for all homogeneous $a, b \in A^{[*]}$. Moreover, $A^{[*]}$ is a dg Poisson algebra.  

A filtered associative dg algebra  in $ \Ch$ is a dg algebra $(A^{[*]}, d_A^{[*]})$  equipped with a filtration of subcomplexes 
\[ \cdots\subseteq F^nA^{[*]}\subseteq F^{n+1}A^{[*]}\subseteq \cdots \subseteq A^{[*]}\]
such that $F^m A^{[*]}\cdot F^nA^{[*]}\subseteq F^{n+m}A^{[*]}$. 

In this case, $\gr(F^\bullet A^{[*]})=\bigoplus_{n}F^nA^{[*]}/F^{n-1}A^{[*]}$ is a graded dg algebra, where each component is a subcomplex. The filtration also makes $A_{Lie}^{[*]}$ a filtered dg Lie algebra. The super commutator $ [-,-]^s$ on $A$ descends to $\gr (F^\bullet A^{[*]})$, denoted as $\overline{[-, -]^s}$, making $\gr (F^\bullet A^{[*]})$ a graded dg Lie algebra. 

If we set  $ A_n^{[*]}=F^{n}A^{[*]}/F^{n-1}A^{[*]}$, then $ A^{[*]}_n=\bigoplus_{p}A^{[p]}_n$ is an object in $ \Ch$ and we have 
\[ A^{[p]}_n\cdot A^{[q]}_{m}\subseteq A^{[p+q]}_{n+m}.\]
 We call a filtered dg algebra $F^\bullet A^{[*]}$ almost  commutative  in
$\Ch $ if 
\[ [ F^mA^{[*]}, F^nA^{[*]}]^s\subseteq F^{n+m-1}A^{[*]},\]
i.e., for $ a\in F^m A^{[*]}$, and $ b\in F^nA^{[*]}$ homogenous (in internal differential degree), 
\[
[a,b]^{s}=ab-(-1)^{|a||b|}ba\in F^{m+n-1}A^{[*]}.
\]
In this case, $\gr(F^\bullet A^{[*]})$ has a dg Lie algebra structure with differential complex structure induced from that  of $(A^{[*]}, d_A^{[*]})$ with Lie bracket
\[ 
\{-, -\}^s: \gr(F^\bullet A^{[*]})\otimes \gr(F^\bullet A^{[*]})\to \gr(F^\bullet A^{[*]})
\]
given as follows: if we write $ \overline{a}\in \gr (F^\bullet A^{[*]})$ for the image of $ a \in F^nA^{[*]}$, then 
\[ \{ \overline{a} , \overline{b} \}^s=\overline{[a,b]^s }\in \gr (F^\bullet A^{[*]})_{m+n-1}.\]

\delete{
It follows from the definition that
\[\{ A_{m}^{[p]}, A_{n}^{[q]}\}^s\subseteq A_{m+n-1}^{[p+q]}.
\]
}

Note that $ \gr (F^\bullet A^{[*]})$ is a $\bb Z$-graded object in $\Ch$, which we call graded complex. We  will call the degree in  the differential complex as differential degree or cohomological degree, when there is another gradation  such as in the case of $ \gr (F^\bullet A^{[*]})$. Any filtered algebra structure $F^\bullet A^{[*]}$ on $A^{[*]}$ defines a filtered dg Lie algebra structure on $(A, [-, -]^s)$. Thus $(\gr F^\bullet A^{[*]}, \overline{[-,-]^s})$ is a graded dg Lie algebra. It is isomorphic to the natural graded dg Lie algebra structure of the graded dg associative algebra $\gr F^\bullet A^{[*]}$. If $ F^\bullet A^{[*]}$ is almost commutative, then $\gr F^\bullet A^{[*]}$ is graded commutative and the dg Lie algebra  $(\gr F^\bullet A^{[*]}, [-,-]^s)$ is abelian in the dg sense. 

To distinguish with the differential degree shift functor, we denote $\gr (F^\bullet A^{[*]})_{[1]}$ for the shift of gradation such that $(\gr (F^\bullet A^{[*]})_{[1]})_n=\gr (F^\bullet A^{[*]})_{n+1}$. In this case, there is no differential to worry about. We define a Lie bracket
 \[\{-,-\}_{[1]}: \gr (F^\bullet A^{[*]})_{[1]}\otimes \gr (F^\bullet A^{[*]})_{[1]}\to \gr (F^\bullet A^{[*]})_{[1]}\]
 to be $\{-,-\}^s$ as a linear map. Then $\{-, -\}_{[1]}$ is a chain map for the differential complex and  homogeneous of degree $0$ in the gradation induced from $F^\bullet A^{[*]}$. Furthermore $\gr F^\bullet A^{[*]}$ and $\gr (F^\bullet A^{[*]})_{[1]}$ are both graded commutative with respect to the differential degree, but the Lie brackets are different. Indeed, if $F^\bullet A^{[*]}$ is almost commutative, then $\{-, -\}^s$ is of degree $-1$ in the gradation induced from $F^\bullet A^{[*]}$, while $\{-, -\}_{[1]}$ is of degree $0$.
 \end{example} 

\begin{example} Let $ \mf g^{[*]}$ be a dg Lie algebra in $ \Ch$. Let $ \Sym^{[*]}_{\Ch}(\mf g^{[*]})$ be the symmetric algebra which has a unique dg Lie algebra structure extending the dg Lie algebra structure on $\mf g^{[*]}$, making $ \Sym^{[*]}_{\Ch}(\mf g^{[*]})$ a graded dg Poisson algebra. 

Let $ U_{\Ch}(\mf g^{[*]})$ be the universal enveloping dg associative algebra of $ \mf g^{[*]}$ in the category 
$\Ch$. Then $ U_{\Ch}(\mf g^{[*]})$ is a quotient of the tensor algebra $ T_{\Ch}(\mf g^{[*]})$ of the differential complex $\mf g^{[*]}=(\mf g^{[*]}, d_{\mf g}^{[*]})$ in $\Ch$. The images $F^p U_{\Ch}(\mf{g})$ of the subcomplexes $\bigoplus_{n\leq p} \mf g^{\otimes n}$ define a filtered structure, making $U_{\Ch}(\mf g)$ into an almost commutative filtered dg associative algebra in $ \Ch$. The PBW theorem states that the associated graded dg algebra  $\gr (F^\bullet U_{\Ch}(\mf g^{[*]}))$ is isomorphic to $ \Sym^{[*]}(\mf g^{[*]})$ as a dg Poisson algebras.
\end{example}

Let $R^{[*]}=(R^{[*]}, \{-,-\})$ be a dg Poisson algebra.  Let $ M^{[*]}=(M^{[*]},d_M^{[*]})$ be a left dg module for the associative dg algebra $R$. We can also make $ M^{[*]}$ into a right $ R^{[*]}$-module using the super commutativity property of $ R^{[*]}$ by 
\[ x r=(-1)^{|x||r|}r x\]
for all $ x\in M^{[*]}$ and $ r\in R^{[*]}$ homogeneous. Thus $ M^{[*]}$ is a $ R^{[*]}$-dg bimodule.

\begin{definition}  A \textbf{dg Poisson module} of a dg Poisson algebra $R^{[*]}=(R^{[*]}, \{-,-\})$ is a dg module $M^{[*]}=(M^{[*]}, d_M^{[*]})$ for $R^{[*]}$ as associative algebra together with a Lie algebra left  dg module structure 
\[
\{-,-\}: R^{[*]} \otimes M^{[*]} \to M^{[*]},
\]
i.e., the induced map $ R^{[*]} \to \End^{[*]}(M^{[*]})$ defined by $ r\mapsto \{r, -\}$ is a dg Lie algebra homomorphism. The bracket $\{-, -\}$ is homogeneous of degree $p$. The two module structures on $M^{[*]}$ satisfy the following derivation properties: for all $ r_1, r_2\in R^{[*]}$ and $ x\in M^{[*]}$ homogeneous, 
\begin{itemize}\setlength\itemsep{5pt}
\item $\{ r_1, r_2x\}=\{r_1, r_2\}x+(-1)^{(|r_1|+p)|r_2|}r_2\{r_1, x\}$,
\item $\{r_1 r_2, x\}=r_1\{r_2, x\}+(-1)^{|r_1||r_2|}r_2\{r_1,x\} $.
\end{itemize} 
\end{definition}

\begin{example}
Let $R$ be a commutative $\mbf k$-algebra and let $ X=\spec(R)$ be the corresponding affine $\mbf k$-scheme.  $\Der_{\mbf k}(R)$ is the global section of vector fields on $ X$. A Poisson bracket $\{-, -\}$ on $ R$ is equivalent to a Lie algebroid structure $\{-, -\}$ on  the structure sheaf $ \cal O_X$, i.e., a $ \mbf k$-Lie algebra structure together with $ \mbf k$-Lie algebra homomorphism $(\cal O_X, \{-,-\})\to \Der_{\mbf k}(R)$, which is called the anchor homomorphism (cf.  \cite{Polishchuk}.)

Any $R$-module $M$ with a compatible $ \Der_{\mbf k}(R)$-module (as Lie algebra) is called $D$-module. A notion of Poisson module structure on an $ \cal O_X$-module is given by a flat Poisson connection (cf.  \cite{Polishchuk}).
\end{example}

\subsection{Loop complexes and interpretation in $\Ch$}
If we regard $\bb C [t, t^{-1}]$ as a graded vector space with $|t|=2$, then we can regard it as a differential complex (concentrated in even degrees). Therefore  $V^{[*]}\otimes \bb C[t, t^{-1}]$ is a differential complex, and the differential structure is $d(v\otimes t^n)=d_V(v)\otimes t^n$. Thus $ V^{[*]}\otimes t^n=V^{[*]}[-2n]$ and we we have 
 \[
 V^{[*]}\otimes \bb C[t, t^{-1}]=\bigoplus_{n\in \bb Z}V^{[*]}[-2n].
 \]
For any differential complex $ (W^{[*]}, d^{[*]}_W)$, we have 
\begin{align*} \Hom^{[*]}(V^{[*]}\otimes \bb C[t, t^{-1}], W^{[*]})
&=\prod_{n\in \bb Z}\Hom^{[*]}(V^{[*]}[-2n],  W^{[*]})  \\
&=\prod_{n\in \bb Z}\Hom^{[*]}(V^{[*]},  W^{[*]}[2n])\\
&=\Hom^{[*]}(V^{[*]},  W^{[*]})[[ x, x^{-1}]] .
\end{align*}
In particular, if $W^{[*]}=V^{[*]}$, using the tensor-Hom adjointness \eqref{hom-tensor-adjoint}, we get
\begin{align}\label{Laurentseries} 
 \Hom^{[*]}(\bb C[t, t^{-1}], \End^{[*]}(V^{[*]})) =\End^{[*]}(V^{[*]})[[ x, x^{-1}]].
\end{align}

Here $|x|=-2$ and $ V^{[*]}\otimes x^n=V^{[*]}[2n]$.  Indeed, each element $f \in \Hom^{[p]}(V^{[*]}\otimes \bb C[t, t^{-1}], W^{[*]})$ is of the form $f=(f_n)_{n\in \bb Z}$ with 
\begin{align*}
f_n \in \Hom^{[p]}(V^{[*]}\otimes \bb C t^n, W^{[*]})&=\Hom^{[p]}(V^{[*]}[-2n], W^{[*]}) \\
&=\Hom^{[p]}(V^{[*]}, W^{[*]}[2n])\\
&=\Hom^{[2n+p]}(V^{[*]}, W^{[*]}) \\
&=(\Hom^{[*]}(V^{[*]}, W^{[*]})[2n])^{[p]} \\
&=(\Hom^{[*]}(V^{[*]}, W^{[*]}) \otimes \bb C x^n)^{[p]}.
\end{align*}
Thus the variable $x$ represents the complex of one dimensional vector space concentrated in the degree $-2$, while the variable $t$ represents the complex of one dimensional vector concentrated in degree $2$. They will be useful in the description of vertex operator algebras. 

Since $ t$ is of even degree, we have 
\[V^{[*]}\otimes \bb C[t, t^{-1}]= \bb C[t, t^{-1}]\otimes V^{[*]}.\]
Thus we have from \eqref{hom-tensor-adjoint}
\begin{align*}
\Hom^{[*]}(V^{[*]}\otimes \bb C[t, t^{-1}], W^{[*]})&=\Hom^{[*]}(\bb C[t, t^{-1}],\Hom^{[*]}( V^{[*]}, W^{[*]}))\\
&=\prod_{n\in \bb Z}\Hom^{[*]}(\bb C[-2n],\Hom^{[*]}( V^{[*]}, W^{[*]}))
\end{align*}

\subsection{Differential filtered algebras} \label{sec:df-algebra} A differential filtered vector space $(V, F^\bullet V, d_V)$ is a vector space $ V$ together with an increasing filtration of subspaces $ \cdots \subseteq F^iV\subseteq F^{i+1}V\subseteq \cdots \subseteq V$ such that $ \cap_{i} F^iV=\{0\}$ and $ \cup_i F^iV=V$ together with a linear map $ d_V: V\to V$ such that $d_V F^iV\subseteq F^{i+1}V$ and $ d_V^2F^iV\subseteq F^{i+1}V$. 

A homomorphism between two differential filtered vector spaces $f: (V, F^\bullet V, d_V)\to
(V', F^\bullet V', d_{V'})$ is a linear map $f: V\to V'$ such that $ f(F^iV)\subseteq F^iV'$ and $f\circ d_V=d_{V'}\circ f$. 
The category of differential filtered vector spaces is an additive category. 
\delete{There is also a a tensor product structure 
\[ F^i(V\otimes V')=\sum_{j\in \bb Z}(F^{i-j}V\otimes F^{j}V')\subseteq V\otimes V'
\]
and 
\[ d_{V\otimes V'}(v\otimes v')= ??? \]
\red{this may not be well-defined. It can be defined over the direct sum $ \oplus_{j\in \bb Z}F^{i-j}V\otimes F^{j}V' \to V\otimes V'$ by $ d(v\otimes v')=d_V v\otimes v'+(-1)^{|v|}v\otimes d_{V'}v'$.}}

Let $ \gr^{[*]}(V)=\gr(F^\bullet V)=\bigoplus_i F^{i}V/F^{i-1}V$. Then $d_V$ induces a differential complex structure on $ \gr^{[*]}(V)$. 

For each $ 0\neq v\in V$, 
$|v|=\min\{i\; |\; v\in F^iV\}$ is called the differential degree of $v$. 
For each $0\neq v\in V$, its image is written as $0 \neq \overline{v} \in \gr^{[|v|]}(V)=F^{|v|}V/F^{|v|-1}V$.  
Hence in the differential complex $(\gr^{[*]}(V), d_V)$, the degree is $|\overline{v}|=|v|$. We remark that for differential graded spaces, the degrees are only defined for homogeneous elements, while for differential filtered vector spaces, the degrees are defined for every non-zero element. We note that $ |v+u|\leq \max\{ |u|, |v|\}$ and $ |v+u|= \max\{ |u|, |v|\}$  if $ |v|\neq |u|$. The image $ \overline{v}$ can be thought as the symbol of $v$ as in the case of differential operators. 
 
 \begin{definition}
 A \textbf{differential filtered algebra} $ (A, F^\bullet A, d_A, \ast)$ is a differential filtered vector space $(A, F^\bullet A, d_A)$ together with an associative multiplication 
\[ \ast: A\otimes A\to A\] making $ (A, F^\bullet A, \ast)$ a filtered associative algebra such that
\[d_A(a\ast b)-(d_Aa\ast b+(-1)^{|a|}a\ast d_Ab)\in F^{|a|+|b|}A.\]
 \end{definition}
 
 We remark that if $ |a'|<|a|$, then $ |a+a'|=|a|$ and 
 \begin{align*}
 \begin{array}{rl}
 d((a+a')\ast b)-(d(a+a')&\ast b+(-1)^{|a|}(a+a')\ast db) \\[5pt] 
 =&d_A(a\ast b)-(d_Aa\ast b+(-1)^{|a|}a\ast d_Ab) \\[5pt]
 & +d_A(a'\ast b)-(d_Aa'\ast b+(-1)^{|a|}a'\ast d_Ab)\in F^{|a|+|b|}A.
 \end{array}
 \end{align*}
Thus for any differential filtered algebra $ (A, F^\bullet A, \ast, d)$, the differential defines a differential graded algebra $(\gr^{[*]}_{F^\bullet}(A), \ast, d)$.  Two elements $ a, b\in A$ are called differential filtered commutative if 
\[ a\ast b-(-1)^{|a||b|}b\ast a\in F^{|a|+|b|-1}A.\] 
This is equivalent to say that the associated dg algebra is commutative. 

Similarly, a left differential filtered $A$-module 
is a differential filtered vector space $ (V, F^\bullet V, d_V)$ such that  $ V$ is a left $ A$-module and $ (F^iA)\ast F^jV\subseteq F^{i+j}V$ for all $ i, j \in \bb Z$ and 
\[ d_V(a\ast v)-(d_Aa\ast v+(-1)^{|a|}a\ast d_Vv) \in F^{|a|+|v|}V.\]

A morphism $f:M \longrightarrow N$ of left differential filtered $A$-modules is linear map such that $f(a \ast m)=a \ast f(m)$, $f(F^iM)\subseteq F^{i}N$ and $(f\circ d_M-d_N \circ f)(F^iM) \subset F^{i}N$. Note that such a morphism induces a morphism $\gr^{[*]}_{F^\bullet} M \longrightarrow \gr^{[*]}_{F^\bullet} N$ of differential graded modules.

We can define right differential filtered $A$-module $(V, F^\bullet V, d_V, \ast)$ by giving a linear map $ \ast: V\otimes A\to V$ satisfying 
$ F^jV\ast F^iA \subseteq F^{i+j}V$ for all $ i, j \in \bb Z$ and 
\[ d_V( v\ast a)-(d_Vv\ast a+(-1)^{|v|} v\ast d_Aa )\in F^{|a|+|v|}V.\]
Here we are using the left derivation property. Using the compatibility condition, it can be changed to a right derivation property by defining $ v\ast_r a=(-1)^{|v||a|}v\ast a$. In fact we have
\begin{align*}&d_V( v\ast_r a)-(v\ast_r d_A(a)+ (-1)^{|a|}(d_V(v)\ast_r a) ) 
\\
&\quad= d_V( (-1)^{|a||v|}v\ast a)-((-1)^{|d_A(a)||v|}v\ast d_A(a)+(-1)^{|a|}(-1)^{|a||d_V(v)|} d_V(v)\ast a ),\\ 
&\quad = (-1)^{|a||v|}\big (d_v( v\ast a)-((-1)^{|v|}v\ast d_A(a)+ d_V(v)\ast a ) \big )\in F^{|a|+|v|}V.
\end{align*}

We remark that each differential graded algebra $A^{[*]}=(\bigoplus_i A^{[i]}, d_A, \ast)$ is automatically a differential filtered algebra with $F^pA=\oplus_{i\leq p}A^i$ and two products given by $ a\ast_l b=a\ast b$ and $ a\ast_r b=(-1)^{|a||b|}a\ast b$ for homogenous elements $a, b$ and then extend linearly to all elements $a, b\in A^{[*]}$. However, for a differential filtered algebra $ (A, F^\bullet A, d_A, \ast)$, the right multiplication $ \ast_r$ need not be defined. Hence in application, one may require that $ \ast_r$ is defined to satisfy associative and compatibility conditions.

\begin{definition} Let $ (A, F^\bullet A, d_A, \ast)$ be a differential filtered algebra. A \textbf{bimodule} for $ (A, F^\bullet A, d_A, \ast)$ is a differential filtered vector space $(V, F^\bullet V, d_V)$ with  left and right $ (A, F^\bullet A, d_A, \ast)$-module structures $ \ast_l$ and $\ast_r$ such that 
\[ (a\ast_l v)\ast_r b=a\ast_l(v\ast_r b).
\]
\end{definition}

\subsection{Weight filtrations on differential filtered algebras and Poisson differential filtered algebras}\label{sec:weight-filtration}
Let $ (A, F^\bullet A, d_A, \ast)$ be differential filtered algebra. A weight filtered differential filtered structure is an increasing filtration
\[
\cdots \subseteq W_nA\subseteq W_{n+1}A\subseteq \cdots
\]
of subspaces such that 
\[
d_A(W_nA)\subseteq W_nA
\]
and  $F^i (W_nA)=W_nA\cap F^iA$ defines a differential filtered vector space structure on $W_nA$ and $W_nA\ast W_mA\subseteq W_{n+m}A$. We will call $ (A, F^\bullet A, W_\bullet A, d_A, \ast)$ a weight filtered differential filtered algebra.

In this case, $\gr_*( A)=\bigoplus_n\gr_n(A)$ with $\gr_n(A)=W_nA/W_{n-1}A$ is a graded differential filtered algebra $(\gr_*(A), F^\bullet, d, \ast)$ in the sense that each homogeneous component is a filtered differential vector space with differential induced from $d_A$ on each $\gr_n(A)$. It can also be interpreted in the sense that the differential $d$ is homogeneous of weighted degree zero and $ F^i\gr _n( A) $ is the image of $ F^iA\cap W_nA$ in the homogeneous subspace $ \gr_n( A)$.  

Similarly, we define the weight degree of  $0\neq a \in A$ by $ \on{wt}(a)=\min\{ n \;|\; a\in W_nA\}$.  We will denote the image of $a$ by $ \pi(a)\in \gr_{\on{wt}(a)}(A)$  and $ \pi(a)\neq 0$ if $a\neq 0$. Thus $ \on{wt}(\pi(a))=\on{wt}(a)=n$ is well-defined for $ \pi(a) \in \gr_nA$ and is independent of the choice of $ a\in W_n(A)$. But $ |\pi(a)|$ with respect to the filtration $F^i\gr_*(A)$ is not well-defined in general. We will need some compatibility conditions between the two filtrations. 

For $0\neq \pi(a) \in  \gr_nA$, we define $ |\pi(a)|=\min\{ p \;|\; \pi(a) \in  F^p\gr_nA\}$. Note that for $a\in A$, we have in general $ |a|\geq |\pi(a)|=p.$  It follows from the definition that  there exists $a'\in F^pA\cap W_nA$ such that $a-a'\in W_{n-1}A$.  

\delete{
Under this assumption, the differential condition 
\begin{align*}d(\pi(a)\ast \pi(b))&=d(\pi(a))\ast \pi(b)+(-1)^{|a||b|}\pi(a)\ast d(\pi(b)),\\
&=d(\pi(a))\ast \pi(b)+(-1)^{|\pi(a)||\pi(b)|}\pi(a)\ast d(\pi(b))
\end{align*}
is independent of the choice of $ a$ and 
$ b$. Thus the graded 
}

Similarly we will use $\gr^{[*]}$ (resp. $\gr^{[*]}\gr_*$) to define the associated graded vector space with respect to the filtration $ F^\bullet A$ (resp. $ F^\bullet (W_nA/W_{n-1}A)$). Then the weight graded dg algebra $(\gr^{[*]}\gr_*(A), d, \ast)$ can also be defined.  We will use the convention that subscripts $ \gr_*$ for weight gradation and superscript $ \gr^{[*]}$ for differential (cohomological) gradation.

On the other hand, the image of $ F^iW_nA $ in $ \gr^{[i]}(A)=F^iA/F^{i-1}A$ defines a filtration of subcomplexes  $W_\bullet \gr^{[*]}A$ making $ (\gr^{[*]}A, d, \ast)$ a weight filtered dg algebra.  Then one can construct the associated weight graded dg algebra $(\gr_*\gr^{[*]}(A), d, \ast)$. It is straightforward to verify the following:

\begin{lemma}\label{lem:2gradations} $(\gr^{[*]}\gr_*(A), d, \ast)$ and 
$(\gr_*\gr^{[*]}(A), d, \ast)$ are isomorphic weight graded dg algebras.  
\end{lemma}
In fact 
\begin{align*} \gr_n\gr^{[i]}(A)&=W_n(F^iA/F^{i-1}A)/W_{n-1}(F^iA/F^{i-1}A)\\
&=W_nA\cap F^iA/(W_{n-1}A\cap F^iA+W_nA\cap F^{i-1}A)\\
&=F^i(W_nA/W_{n-1}A)/F^{i-1}(W_nA/W_{n-1}A)=\gr^{[i]}\gr_n(A).
\end{align*}
The differentials and multiplications can be verified from this identification routinely.  

Similar to the definition of the degree function $ |\cdot |: A\to \bb Z$ with respect to the differential filtration $ F^\bullet V$ for a differential filtered vector space $V$, we define the weight function $ \on{wt}: V\to \bb Z $ of weighted filtered differential filtered vector space $V$ with respect to the weighted filtration $ W_\bullet V$ by
\[
\on{wt}(v)=\min\{ n \;|\; v\in W_nV\}. 
\]
We can similarly define a weight  filtered left (resp. right)
module $(V, W_\bullet, F^\bullet, d_V, \ast)$ 
for  a weight filtered differential filtered 
algebra $ (A, F^\bullet, W_\bullet, d, \ast)$ as a left (resp. right) module by additionally 
equipping a weight filtration $ W_\bullet V$ 
such that  $ d_V W_nV\subseteq W_n V$ 
and $ F^i(W_nV)=F^iV\cap W_nV$ defines 
a differential filtered vector space structure 
 on $W_nV$ and 
\[ W_nA \ast W_mV\subseteq W_{n+m}V
\quad (\text{resp. } \; W_mV\ast W_nA \subseteq W_{n+m}V \;).
\]
In particular we have $ F^iW_nA\ast F^jW_mV\subseteq F^{i+j}W_{m+n}V$.


\section{dg vertex (operator) algebras}\label{sec:3}
\subsection{Vertex algebras in the symmetric tensor category $\Ch$.}
In what follows we give the definitions and basic properties of dg vertex (operator) algebras. The proofs and details can be found in \cite{Lepowsky-Li}.

\begin{definition} 
A \textbf{differential graded (dg) vertex algebra} in $\Ch$ is a cochain complex $(V^{[*]}, d_V^{[*]})$ over $ \cc$ equipped with a chain map (vertex operator map) in $ \Ch$ 
\begin{align*}
\begin{array}{cccc}
Y(\cdot, x): & V^{[*]} & \longrightarrow & \on{End}^{[*]}(V^{[*]})[[x,x^{-1}]] \\
           & v & \longmapsto      & Y(v, x)=\displaystyle \sum_{n \in \mathbb{Z}}v_n x^{-n-1}
\end{array}
\end{align*}
with $x$ of degree $-2$, and a particular vector $\mbf{1} \in V^{[0]}$ with $d_V(\mbf 1)=0$, the vacuum vector, satisfying the following conditions:   
\begin{itemize}\setlength\itemsep{5pt}
\item For any $u, v \in V^{[*]}$, $u_nv=0$ for $n$ sufficiently large, i.e., $Y(u, x)v \in V^{[*]}((x))$ (Truncation property),
\item $Y(\mathbf{1},x)=\on{id}_V$ (Vacuum property),
\item $Y(v, x)\mathbf{1} \in V[[x]]$ and $\displaystyle \lim_{x \to 0} Y(v, x)\mathbf{1}=v$ (Creation property),
\item The Jacobi identity 
\begin{align}
\begin{array}{l@{\hspace{0cm}}l}
 x_2^{-1}& \delta(\frac{x_1-x_0}{x_2})Y(Y(u, x_0)v, x_2)= \\[5pt]
& x_0^{-1}\delta(\frac{x_1-x_2}{x_0})Y(u, x_1)Y(v, x_2)- (-1)^{|u||v|}x_0^{-1}\delta(\frac{x_2-x_1}{-x_0})Y(v, x_2)Y(u, x_1),
\end{array}
\end{align}
where $\delta(x+y)=\displaystyle \sum_{n \in \mathbb{Z}}\sum_{m=0}^{\infty}\binom{n}{m}x^{n-m}y^m$ and $v,u\in V$ are homogeneous.
\end{itemize}
A dg vertex algebra will be denoted as $ (V^{[*]}, d_V^{[*]}, Y(\cdot, x), \mathbf{1})$.
\end{definition}

\begin{remark}\label{Remark_2} \delete{
(0) $G_m$ action on the nilpotent variety $\cal N$ by $ x\mapsto t\cdot x=t^{-2}x$. I will check how this is related. 
}
(1) We first note that because $Y(\cdot, x)$ is a chain map, if $v\in V^{[p]}$ has degree $p$, then $ v_n\in \End^{[p-2n-2]}(V^{[*]})=(\End^{[*]}(V^{[*]})[-2n-2])^{[p]}$.  We will denote the degree of $v_n$ by $ |v_n|=|v|-2n-2$. 

(2) The differential structure $d$ on $\End^{[*]}(V^{[*]})$ is given by $ d(f)=d_V\circ f-(-1)^{|f|}f\circ d_V$ for $ f\in \End^{[m]}(V^{[*]})$ (cf.  \eqref{eq:diff_end}). Let $d'$ be the differential on $ \End^{[*]}(V^{[*]})[[x, x^{-1}]]$. Since $d'(x^n)=0$ for all $n$, we have 
\[ d'(v_n x^{-n-1})=d(v_n)x^{-n-1}=(d_V\circ v_n-(-1)^{|v|-2n-2}v_n\circ d_V)x^{-n-1}.\]
Since $Y(\cdot, x)$ is a chain map, i.e.,  for $v\in V^{[p]}$, $ Y(d^{[p]}_V(v), x)=d'^{[p]}(Y(v, x))$, we have 
\[
(d_V(v))_n=d(v_n)=(d_V\circ v_n-(-1)^{|v|-2n-2}v_n\circ d_V);
\]
Hence
\begin{align}\label{eq:d(v_nu)}
d_V(v_n(u))&=d(v_n)(u)+(-1)^{|v|-2n-2}v_n(d_V(u)) \nonumber\\
&=d_V(v)_n(u)+(-1)^{|v|-2n-2}v_n(d_V(u)).
\end{align}

For each $ n\in \bb Z$, the multiplication $ V^{[*]}\otimes V^{[*]}\to V^{[*]}[-2n-2]$ defined by $ v\otimes u\mapsto v_n(u)$ makes $V^{[*]}$ into a graded algebra with a product of degree $0$ and $ d_V$ is a derivation of this algebra.

\begin{lemma}\label{lem:chain_map_1}
 For any homogenous $u\in V^{[*]}$ with $d_V(u)=0$, the linear map $\cal D^u_n: V^{[*]}\to V^{[*]}[|u|-2n-2]$ defined by $\cal D^u_n(v)=v_n(u)$ is chain map. 
\end{lemma}
In fact, more general is true:

\begin{lemma} \label{lem:chain_map_2} The map $u\mapsto \cal D_n^u$ defines a chain map
$ V^{[*]}\to \End^{[*]}(V^{[*]})[-2n-2]$.
\end{lemma}
Taking $ u=\mbf 1$ and using $d_V(\mbf 1) =0$ we have 
\begin{align*}
 d_V(v_n\mbf 1)=d_V(v)_n\mbf 1.
\end{align*}
In particular, the map $\cal D_n:  V^{[*]}\to V^{[*]}[-2n-2]$ defined by $ v\mapsto v_n\mbf 1$ is a  chain map, i.e., $\cal D_n\in \End^{[-2n-2]}(V^{[*]})$ such that $ d(\cal D_n)=0$.

For $v \in V^{[p]}$, we have
\begin{align*} Y(d^{[p]}_V(v), x)=d'^{[p]}(Y(v, x))=\sum_{n}d'^{[p]}(v_n x^{-n-1})
=\sum_{n}d^{[p-2n-2]}(v_n)x^{-n-1}.
\end{align*}
In particular, if $u\in V^{[m]}$, we have 
\begin{align*} Y(d^{[p]}_V(v), x)u
&=\sum_{n}d^{[p-2n-2]}(v_n)u x^{-n-1}\\
&=\sum_{n}\big (d_V^{[p-2n-2+m]}(v_n(u))-(-1)^{p-2n-2}v_{n}(d_V^{[m]}(u)) \big )x^{-n-1}\\
&=d'^{[p+m]}Y(v, x)u-(-1)^{p}Y(v, x)d_V^{[m]}(u).
\end{align*}
\end{remark}

The weak associativity coming from the Jacobi identity is the same as in the classical case, i.e., for $u, w \in V^{[*]}$, there exists $k \in \mathbb{N}$ such that for any $v \in V^{[*]}$, we have
\begin{align}\label{eq:weak_associativity}
(x_0+x_2)^kY(Y(u, x_0)v, x_2)w=(x_0+x_2)^kY(u, x_0+x_2)Y(v, x_2)w.
\end{align}
The weak commutativity becomes: for any homogeneous elements $u, v \in V^{[*]}$, there exists $k \in \mathbb{N}$ such that
\begin{align}\label{eq:weak_commutativity}
(x_1-x_2)^k \left( Y(u, x_1)Y(v, x_2)-(-1)^{|u||v|}Y(v, x_2)Y(u, x_1) \right)=0.
\end{align}

\begin{remark} We note that the Jacobi identity is equivalent to 
\begin{align} \label{eq:jacobi_2}\sum_{i\geq 0}(-1)^i\binom{l}{i}(u_{m+l-i}v_{n+i}-(-1)^l(-1)^{|v||u|}v_{n+l-i}u_{m+i}) =\sum_{i\geq 0}\binom{m}{i}(u_{l+i}v)_{m+n-i}
\end{align}
for all $ l, m, n\in \bb Z$ and for all $v, u \in V^{[*]}$ are homogeneous of
 degree $|u|$ and $ |v|$ respectively. Taking $ l=0$ we get
\begin{align} \label{eq:jacobi_3}
[u_m, v_n]^{s}= u_m v_n-(-1)^{|u_m||v_n|}v_nu_m=\sum_{i\geq 0}\binom{m}{i}(u_m(v))_{m+n-i}
\end{align}
 for all  
 $ m, n \in \bb Z$.  Taking $ m=0$ we get 
\[ (u_l(v))_n=\sum_{i
\geq 0}(-1)^i\binom{l}{i}(u_{l-i}v_{n+i}-(-1)^l(-1)^{|v||u|}v_{l+n-i}u_{i}).\]

If we write $\on{LHS}((u, x_1), (v_, x_2), x_0)$ for the left hand side of the Jacobi identity,  
then we have 
\[ \on{LHS}((u, x_1), (v, x_2), x_0)=(-1)^{|u||v|}\on{LHS}((v, x_2), (u, x_1), -x_0). 
\]
This implies the right hand side of the Jacobi identity should also satisfy
 \begin{align} 
 x_2^{-1}\delta(\frac{x_1-x_0}{x_2})Y(Y(u, x_0)v, x_2)=(-1)^{|u||v|} x_1^{-1}\delta(\frac{x_2+x_0}{x_1})Y(Y(v, -x_0)u, x_1).
\end{align}
and the graded version of the skew symmetry
\[ Y(u, x)v=(-1)^{|u||v|}e^{x\cal D}Y((v, -x)u
\]
for all $ u, v \in V$ homogeneous and $\cal D(v)=v_{-2}\mathbf{1}$ for all $v \in V^{[*]}$.
\end{remark}

\begin{example}
Let $V^{[*]}$ be a dg vertex algebra concentrated in degree $0$, i.e., $V^{[*]}=V^{[0]}$. Based on Remark~\ref{Remark_2}, we have $a_n(V) \subseteq V^{[-2n-2]}$ for any $a \in V$. It follows that $a_nV=0$ for any $n \neq -1$. Furthermore, using Formula~\eqref{eq:jacobi_3}, we have $a_{-1}b-b_{-1}a=[a_{-1},b_{-1}]^s \mbf 1=0$. Thus $V$ has a commutative algebra structure with the product $a \cdot b=a_{-1}b$. We see that any commutative algebra $A$ can be seen as a dg vertex algebra concentrated in degree $0$ by setting, for all $a, b \in A$,
\begin{equation*}
a_n b=\begin{cases}
0 & \text{ if } n \neq -1, \\[5pt]
a \cdot b  & \text{ if } n=-1.
\end{cases}       
\end{equation*}
\end{example}

\begin{example}
Let $(A^{[*]},d_A^{[*]}, D)$ be a graded commutative dg algebra with unit $1$ equipped with a derivation of degree $2$, i.e., $D:A^{[*]} \longrightarrow A[2]^{[*]}$ is a chain map of degree $0$. Using the classical construction of Borcherds, we define the following map:
\begin{align*}
\begin{array}{cccc}
Y(\cdot, x): & A^{[*]} & \longrightarrow & \on{End}^{[*]}(A^{[*]})[[x,x^{-1}]] \\[5pt]
                   & a & \longmapsto    & e^{xD}a=\displaystyle \sum_{n \in \mathbb{N}}\frac{1}{n!}D^n(a)x^n
\end{array}
\end{align*}
where we set $|x|=-2$. Then for any homogeneous $a \in A^{[|a|]}$ and $n \in \mathbb{N}$, we have $|D^n(a)x^n|=|a|+2n-2n=|a|$. Thus $Y(\cdot, x)$ is a map of degree $0$.

Set $a \in A^{[|a|]}$. Then
\begin{align*}
\begin{array}{rcl}
Y(d_Aa, x) & = & \displaystyle \sum_{n \in \mathbb{N}}\frac{1}{n!}D^n(d_Aa)x^n \\[5pt]
                         & = & \displaystyle \sum_{n \in \mathbb{N}}\frac{1}{n!}d_A(D^n(a))x^n
\end{array}
\end{align*}
because $D$ and $d^{[*]}$ commute. Similarly, we have
\begin{align*}
\begin{array}{rcl}
d'(Y(a, x)) & = & \displaystyle \sum_{n \in \mathbb{N}}\frac{1}{n!}d'(D^n(a)x^n), \\[5pt]
           & = & \displaystyle \sum_{n \in \mathbb{N}}\frac{1}{n!}d(D^n(a))x^n
\end{array}
\end{align*}
because $d'(x^n)=0$. Furthermore, as $D^n(a)$ can be seen as an operator in $\on{End}^{[*]}(A^{[*]})$, for any $b \in A^{[|b|]}$ we have
\begin{align*}
\begin{array}{rcl}
d(D^n(a))(b) & = & d_A \circ D^n(a)(b)-(-1)^{|D^n(a)|}D^n(a) \circ d_A(b), \\[5pt]
                    & = &d_A(D^n(a)(b))-(-1)^{|a|+2n}D^n(a)d_A(b), \\[5pt]
                    & = &d_A(D^n(a))b+(-1)^{|D^n(a)|}D^n(a)d_A(b)-(-1)^{|a|+2n}D^n(a)d_A(b), \\[5pt]
                    & = &d_A(D^n(a))b.
\end{array}
\end{align*}
It follows that $d'(Y(a, x)) = \sum_{n \in \mathbb{N}}\frac{1}{n!}d_A(D^n(a))x^n=Y(d_Aa, x)$, and that $Y(\cdot, x)$ is a chain map. In this context, $a_n=0$ when $n \geq 0$, so the first three axioms of a dg vertex algebra are easily verified. As $A^{[*]}$ is graded commutative and $D$ is of even degree, for any $a, b$ homogeneous, we have
\[
Y(a, x_1)Y(b, x_2)=(-1)^{|a||b|}Y(b, x_2)Y(a, x_1),
\]
so the weak commutativity \eqref{eq:weak_commutativity} is verified. We then apply the same reasoning as in \cite[Example 3.4.6]{Lepowsky-Li} to show that the weak associativity \eqref{eq:weak_associativity}  is also verified. Proposition 3.4.3 of \cite{Lepowsky-Li} still applies in the dg setting, and the Jacobi identity is satisfied. Therefore $(A^{[*]},Y(\cdot, x),1)$ is a dg vertex algebra.
\end{example}

\begin{lemma}\label{lem:comm_D} The map given by $\cal D(v)=v_{-2}\mathbf{1}$ is 2-cocycle in $ Z^{[2]}\End^{[*]}(V^{[*]})$, i.e., it satisfies $d_V\circ \cal D=\cal D\circ d_V$. Furthermore,
\[ [ \mathcal D, Y(v, x)]^s=Y(\mathcal D (v), x)=\frac{d}{dx}Y(v, x).\]
\end{lemma}

\begin{remark} A dg vertex algebra in the category $ \Ch$ can also be defined using Borcherds relations described in \cite[ Def. 3.6.5]{Lepowsky-Li} with all  vectors being homogeneous. The only modifications are  (iv)
\[ u_n(v)=(-1)^{|u||v|}\sum_{i\geq 0}(-1)^{i+n+1}\cal D^{(i)}(v_{n+i}(u))
\] 
and (v) which we have stated above. 
\end{remark}
A  dg vertex algebra homomorphism $f: (V^{[*]}, d_V, Y(\cdot, x), \mbf 1)\to (V'^{[*]}, d_{V'},Y'(\cdot, x), \mbf 1')$  is a chain map $ f\in \Ch ((V^{[*]}, d_V),( V'^{[*]}, d_{V'}))$ such that 
\[ f( Y(v, x)u)=Y'(f(v), x)(f(u))\quad \text{ and } \quad f(\mbf 1)=\mbf 1'.\]
Given any differential complex $(V^{[*]}, d_V^{[*]})$, let $ H^{[*]}(V^{[*]})$ be the cohomology, which can be regarded as differential complex with zero differential. 
\begin{theorem} If $ (V^{[*]}, d_V^{[*]}, Y(\cdot, x), \mbf 1)$ is a dg vertex algebra , then there is an induced map
$ H(Y)(\cdot, x): H^{[*]}(V^{[*]})\to \End^{[*]}(H^{[*]}(V^{[*]}))[[x,x^{-1}]]$ defining a dg vertex algebra structure on $H^{[*]}(V^{[*]})$, with $\mbf 1$ being the image of $\mbf 1$ in $H^{[0]}(V^{[*]})$ since $ d_V\mbf 1=0$.
\end{theorem} 
\begin{proof} One only needs to verify the Jacobi identity, which can be directly verified by using Lemmas \ref{lem:chain_map_1} and \ref{lem:chain_map_2}, as well as the component version \eqref{eq:jacobi_2} of the Jacobi identity.
\end{proof}
If $f: (V^{[*]}, d_V, Y(\cdot, x), \mbf 1)\to (V'^{[*]}, d_{V'},Y'(\cdot, x), \mbf 1')$ is homomorphism of dg vertex algebras, then $H(f): H^{[*]}(V^{[*]})\to H^{[*]}(V'^{[*]})$ is a homomorphism of dg vertex algebras. We say that $ f$ is a quasi-isomorphism if $H(f)$ is an isomorphism of dg vertex algebra.

\begin{remark}
A dg vertex algebra with zero differential is a vertex super algebra by taking $V^{even}=\sum_p V^{[2p]}$ and $V^{odd}=\sum_p V^{[2p-1]}$.
\end{remark}
 
We will refrain from getting into the homotopy theory aspect of dg vertex algebras and focus on the vertex algebra properties in the dg context.

\subsection{Conformal structures} 
A {\bf  dg vertex operator algebra } is a vertex algebra $ (V^{[*]}, Y(\cdot, x), \mathbf{1})$ together with an element $ \omega \in V^{[4]}$ satisfying $d_V^{[4]}(\omega)=0$,  called conformal vector (or Virasoro element) such that $ Y(\omega, x)=\sum_{n\in \mathbb Z}L(n)x^{-n-2}$  (thus $ |L(n)|=-2n$) and: 
\begin{itemize}\setlength\itemsep{5pt}
\item $[L(m),L(n)]=(m-n)L(m+n)+\frac{1}{12}(m^3-m)\delta_{m+n,0}c_V$ for $m, n \in \mathbb{Z}$ where $c_V \in \mathbb{C}$ is the central charge of $V^{[*]}$. 
\item The linear operator $ L(0)\in \End_{\bb  C}^{[0]}( V^{[*]})$ is semisimple and $ V^{[*]}=\bigoplus_{n\in \frac{1}{2}\mathbb Z} V^{[*]}_n$ with $L(0)v=nv=\on{wt}(v)v$ for $ n\in \frac{1}{2}\mathbb{Z}$, $v \in V^{[*]}_{n}$. Furthermore $\omega \in V^{[*]}_2$, $\dim V^{[*]}_n<\infty $, and $V^{[*]}_n=0$ for $ n \ll 0$.  
\item $L(-1)=\mathcal D$.
\end{itemize}

We remark that a dg vertex operator algebra $V^{[*]}=\bigoplus_{n\in \frac{1}{2}\mathbb Z} V^{[*]}_n$ is automatically  a $\frac{1}{2} \bb Z$-graded vector space by  weights. This weight grading together with the grading by the differential degree gives $V^{[*]}$ a $ \mathbb{Z} \times \frac{1}{2}\mathbb{Z}$-grading.  If $v\in V^{[p]}_n$ we will write $ wt(v)=n$ and $|v|=p$. Thus the weights and degrees in this setting do not interchange. 

\begin{lemma}
In a dg vertex operator algebra $V^{[*]}$, for any $v \in V^{[*]}$, we have
\begin{align}\label{lem:comm_L(-1)}
[L(-1), Y(v, x)]^{s}= Y(L(-1)v, x)=\frac{d}{dx}Y(v, x),
\end{align}
\begin{align}\label{lem:comm_L(0)}
[L(0), Y(v, x)]^{s}= x\frac{d}{dx}Y(v, x)+Y(L(0)v, x).
\end{align}
\end{lemma}

\begin{proof}
Formula~\eqref{lem:comm_L(-1)} is a direct consequence of the definition of $\omega$ and of Lemma~\ref{lem:comm_D}. Using Formula~\eqref{eq:jacobi_3}, we have $[L(0), v_n]^s=(L(-1)v)_{n+1} + (L(0)v)_n$. By rewriting the previously equality using Formula~\eqref{lem:comm_L(-1)} and the summation $Y(\cdot, x)$, we obtain Formula~\eqref{lem:comm_L(0)}.
\end{proof}

A dg vertex operator algebra homomorphism $f: (V^{[*]}, Y, \mbf 1, \omega)\to (V'^{[*]}, Y', \mbf 1', \omega')$ in $ \Ch$ is a homomorphism of vertex algebra $f$ such that $f(\omega)=\omega'$. Thus $f$  preserves both the degree graded structure and the weight graded structure.  This structure has the property that, for any $ u \in V_n^{[|u|]}$,
\[
u_m(V_r^{[p]})\subseteq V^{[p+|u|-2m-2]}_{r+n-m-1}.
\]

\begin{remark} Since $ d_V(\omega)=0$, using \eqref{eq:diff_end} and \eqref{eq:d(v_nu)}, it follows that each $ L(n)\in Z^{[-2n]}(\End^{[*]}(V^{[*]}))$. Hence in this context, the Lie algebra $ \on{\bf Vir}$ generated by the $L(n)$ is a graded Lie algebra (or differential graded Lie algebra concentrated in even degrees). Thus $(V^{[*]}, d_V^{[*]})$ is a dg module of the dg Lie algebra.  

There is also a dg version of Neveu-Schwarz algebra discussed in \cite{Kac-Wang}. The Neveu-Schwarz  algebra $\on{\bf NS}$ is a Lie algebra in $ \Vecs$ with basis $\{L_n, G_n\:|\; n\in \bb Z\}\cup \{ C\}$ with relations describe in \cite{Kac-Wang}.  We now assume there is an element $ \tau\in V^{[3]}$  with $d_V(\tau)=0$ and 
\[
Y(\tau, x)=\sum_{n\in \bb Z}G_n x^{-n-2}
\]
such that $ \omega=\frac{1}{2}G_{-1}(\tau)$ and 
\begin{align*}
\begin{array}{rcl}
  \left[ G_{m}, L_n \right]^s & = & (m+\frac{1}{2}+\frac{n}{2})G_{m+n}; \\[5pt]
  \left[ G_m, G_n \right]^s & = & 2L_{m+n-1}-\frac{1}{3}m(m+1)\delta_{m+n+1, 0}C.
\end{array}
\end{align*}
The $G_n$ is the $G_{n+\frac{1}{2}}$ in \cite{Kac-Wang}. We note that degree $|G_n|=3-2n-4=-2n-1.$ In particular $ G_{-1}$ has degree $1$. Thus $\omega\in V^{[4]}$. 
\end{remark} 
A dg vertex algebra with the structure $ \tau\in V^{[3]}$ satisfying the above conditions is called $ N=1$ dg vertex algebra following the definition of \cite{Kac-Wang}. In this case, the linear span of the linear chain maps in $Z^{[*]}(V^{[*]})$ is dg Lie algebra with zero differentials and any dg module $(M^{[*]}, d_M^{[*]})$ is automatically dg module for the dg Lie algebra.

\section{Differential graded modules for dg vertex (operator) algebras}\label{sec:4}

\subsection{The case of dg vertex algebras}
\begin{definition}
Let $(V^{[*]},Y(\cdot, x),\mathbf{1})$ be a dg vertex algebra in $ \Ch$. A  dg module over $V^{[*]}$ is an object  $(M^{[*]}, d_M^{[*]})$ in $\Ch$  equipped with a chain map
\begin{align*}
\begin{array}{cccc}
Y_M(\cdot, x): & V^{[*]} & \longrightarrow &\mathrm{End}^{[*]} (M^{[*]})[[x,x^{-1}]] \\
           & v & \longmapsto      & Y_M(v, x)=\displaystyle \sum_{n \in \mathbb{Z}}v_n x^{-n-1}
\end{array}
\end{align*}
such that for any $u,v \in V^{[*]}$, the following properties are verified:
\begin{itemize}\setlength\itemsep{5pt}
\item For any $u \in V^{[*]}$, $w \in M^{[*]}$ , $u_n w=0$ for $n$ sufficiently large, i.e., $Y_M(u, x)w \in M^{[*]}((x))$. (Truncation property)
\item $Y_M(\mathbf{1},x)=\mathrm{id}_{M}$. (Vacuum property)
\item The Jacobi identity 
 \begin{align*} 
 x_2^{-1}&\delta(\frac{x_1-x_0}{x_2})Y_M(Y(u, x_0)v, x_2)=\\
 &x_0^{-1}\delta(\frac{x_1-x_2}{x_0})Y_M(u, x_1)Y_M(v, x_2)-(-1)^{|v||u|}x_0^{-1}\delta(\frac{x_2-x_1}{-x_0})Y_M(v, x_2)Y_M(u, x_1). 
\end{align*}
\end{itemize}
\end{definition}

For a dg homogeneous element $v \in V^{[*]}$ and $m \in M^{[*]}$, one can verify that
\begin{align}\label{eq:d_M(v_nu)}
d_M(v_n(u))=d_V(v)_n(m)+(-1)^{|v|-2n-2}v_n(d_M(m)).
\end{align}

A homomorphism $f : M^{[*]} \longrightarrow M'^{[*]}$ between dg modules is a chain map $f\in \Ch(M^{[*]}, M'^{[*]})$ such that $f(v_nm) = v_nf(m)$ for $v \in V^{[*]}$, $m \in M^{[*]}$. One notes that 
\begin{align*}
 f(d_{M}(v_nm))&=f(d_V(v)_nm+(-1)^{|v|-2n-2}v_n(d_M(m)), \\
&=d_V(v)_nf(m)+(-1)^{|v|-2n-2}v_nf(d_M(m)), \\
&=d_V(v)_nf(m)+(-1)^{|v|-2n-2}v_nd_{M'}f(m), \\
&=d_{M'}(v_nf(m)).
\end{align*}
The chain map and the action of $v_n$ action are compatible. 

Note that $ \on{Ker}(f)$ and the cokernel $ \on{Coker} (f)$ are all dg modules. The category of all vertex dg modules is an abelian category which is closed under coproduct, and will be denoted by $V^{[*]}\Mod^w$.

Note that if $ v $ is homogeneous of degree $|v|$ in $V^{[*]}$, then
\[
v_n(M^{[p]})\subseteq M^{[p+|v|-2n-2]}.
\]
We remark that for each $n\in  \bb Z $, the map  $V^{[*]}\to \End^{[*]}(M^{[*]})[-2n-2]$ sending $v$ to $v_n$ is a chain map. 

\delete{
 If $V^{[*]}$ is a dg vertex operator algebra, then $d_V(\omega)=0$, and so $L_n\in Z^{[-2n]}(\on{End}^{[*]}(M^{[*]}))$ (here $Y_M(\omega, x)=  \sum_{n \in \mathbb{Z}}L_n x^{-n-2}$). In particular, $L_n: M^{[*]}\to M^{[*]}[-2n]$ is a chain map which induces a homomorphism $H(L_n): H^{[*]}(M^{[*]})\to H^{[*]} (M^{[*]})[-2n]$. 
}

\begin{theorem} \label{thm:H(M)dgModule}
The functor $H^{[*]}(-):V^{[*]}\Mod \to H^{[*]}(V^{[*]})\Mod$ is additive.
\delete{If $(V^{[*]}, Y(\cdot, x), \mbf 1)$ is a dg vertex algebra and $ (M^{[*]}, Y_M)$ is a dg $V^{[*]}$-module then, the cohomology space $ H^{[*]}(M^{[*]})$ is dg module for the dg vertex algebra $H^{[*]}(V^{[*]}, Y(\cdot, x), \mbf 1)$. }
\end{theorem}

\subsection{The case of dg vertex operator algebras}

If $(V^{[*]}, Y(\cdot, x), \mathbf{1}, \omega)$ is a dg vertex operator algebra, then a vertex algebra dg module $M^{[*]}$ is called a {\bf weak} dg module for the dg vertex operator algebra. 
Similar to Lemmas \ref{lem:chain_map_1}, \ref{lem:chain_map_2}. 
As consequences of the definition above we have:
\begin{itemize}\setlength\itemsep{5pt}
\item $Y_M(L(-1)v, x)=\frac{d}{dx}Y_M(v, x)$,
\item  $[L_m,L_n]=(m-n)L_{m+n}+\frac{1}{12}(m^3-m)\delta_{m+n,0}c_V,$ 
\end{itemize}
where $Y_M(\omega, x)=\displaystyle \sum_{n \in \mathbb{Z}}L_n x^{-n-2}$. \\

We remark that although the operator $L_0$ is acting on $M^{[*]}$, there is no requirement for $L_0$ to be diagonalizable or even locally finite. All we know is that $ L_0$ is chain map on $ M^{[*]}$.  Thus the generalized eigenspaces of $L_0$ in $M^{[*]}$ are all subcomplexes. 
 
 We can define a \textbf{graded} dg module  as a  weak dg module $M^{[*]}$ together with a decomposition $M^{[*]}=\bigoplus_{n \in \frac{1}{2}\mathbb Z}M^{[*]}(n)$ in $\Ch$ such that, for all $n \in \frac{1}{2}\mathbb{Z}$, $u \in V_n$, 
\[
u_m(M^{[*]}(r))\subseteq M^{[*]}(r+n-m-1)
\] 
for all $ m \in \mathbb Z$, $r \in  \frac{1}{2}\mathbb{Z}$.  Note that  each $ M^{[*]}(n)=\bigoplus_pM^{[p]}(n)$ is a differential (cochain) complex, i.e., an object in $ \Ch$. Thus $ M^{[*]}$ is $\bb Z \times  \frac{1}{2} \bb Z$-graded. We will use $M^{[*]}(\bullet)$ to denote the $ \frac{1}{2}\mathbb Z$-graded  on $ M^{[*]}$ as direct sums of differential complexes. The same weak $V^{[*]}$-dg module $M^{[*]}$ can have many different $ \frac{1}{2}\mathbb Z$-graded decompositions of complexes. 

 Let $V^{[*]}\Mod^{gr}$ be the category of all graded dg modules over $V^{[*]}$. We will use $(M^{[*]}, M^{[*]}(\bullet))$ to denote the objects of $V^{[*]}\Mod^{gr}$ to specify the graded structure. A morphism $ f: (M^{[*]}, M^{[*]}(\bullet))\longrightarrow  (N^{[*]}, N^{[*]}(\bullet))$ in $V^{[*]}\Mod^{gr}$ is a dg module  homomorphism  $ f: M^{[*]} \longrightarrow N^{[*]}$ such that $ f(M^{[*]}(n))\subseteq N^{[*]}(n)$ for all $ n\in \frac{1}{2}\mathbb Z$. It is straightforward to verify that $V^{[*]}\Mod^{gr}$ is an abelian category.

 More generally one can define a filtered module structure on a dg module $M^{[*]}$ by an increasing filtration of  subcomplexes 
\[
\cdots \subseteq F^pM^{[*]}\subseteq F^{p+1}M^{[*]}\subseteq \cdots
\]
satisfying 
\[
u_m(F^pM^{[*]})\subseteq F^{p+n-m-1}M^{[*]}
\] 
for all $ u\in V_n^{[*]} $ and $ m\in \bb Z$. Let $ V^{[*]}\Mod^{fil}$ be the category of all filtered dg modules whose objects are the $ F^\bullet M^{[*]}$ and the morphisms $f: F^\bullet M^{[*]}\to F^\bullet N^{[*]}$ should be morphisms in $ V^{[*]}\Mod^w$ and satisfy $f(F^pM^{[*]})\subseteq F^pN^{[*]}$ for all $ p$. The category $V^{[*]}\Mod^{fil}$ is additive, but in general not abelian.

Both categories $ V^{[*]}\Mod^{gr}$ and $ V^{[*]}\Mod^{fil}$ have an auto-equivalence functor $ T$ which is the identity on the $V^{[*]}$-module $M^{[*]}$ but with shifted graded (filtered) structures  defined by $T(M^{[*]})(n)=M^{[*]}(n+1)$ and $F^p(T(M^{[*]}))=F^{p+1}M^{[*]}$. There is a natural functor  $\gr : V^{[*]}\Mod^{fil}\to V^{[*]}\Mod^{gr}$ defined by $ \gr (F^\bullet M^{[*]})(p)=F^pM^{[*]}/F^{p-1}M^{[*]}$. This functor commutes with the auto-equivalence functor $T$, and it has a section $ (M^{[*]}, M^{[*]}(\bullet))\mapsto F^\bullet M^{[*]}$ with 
\[
F^p M^{[*]}=\oplus_{n\leq p}M^{[*]}(n).
\]
 
We remark that $V^{[*]}\Mod^{gr}$ is not a full subcategory of $ V^{[*]}\Mod^w$.   However there is a natural additive forgetful functor $V^{[*]}\Mod^{gr} \longrightarrow V^{[*]}\Mod^w$, which sends $(M^{[*]}, M^{[*]}(\bullet))$ to $M^{[*]}$ forgetting the graded structure. The category $V^{[*]}\Mod^{gr}$ has a full subcategory $V^{[*]}\Mod^{gr}_+$ consisting of all modules $M^{[*]}(\bullet)$ such that $M^{[*]}(n)=0$ for $n \ll  0$. 

An \textbf{admissible} dg module is an object  $M^{[*]}$ in $V^{[*]}\Mod^w$ such there is a graded structure $(M^{[*]}, M^{[*]}(\bullet)) $ in $V^{[*]}\Mod^{gr}$  satisfying $M^{[*]}(n)=0$ for all $ n<0$. Let $V^{[*]}\Mod^{adm}$ be the full subcategory of $ V^{[*]}\Mod^{w}$ consisting of all admissible modules $M^{[*]}$.  Thus every object in $ V^{[*]}\Mod^{adm}$ is of the form $M^{[*]}$ for some $ (M^{[*]}, M^{[*]}(\bullet))$ in $ V^{[*]}\Mod^{gr}_+$.  Being a full subcategory of $V^{[*]}\Mod^w$,  homomorphisms between admissible modules do not preserve the graded structure in general.

 We want to mention that for a given admissible module $M^{[*]}$, there could be many different graded structures $ M^{[*]}(\bullet)$ on $M^{[*]}$ to make $(M^{[*]}, M^{[*]}(\bullet))$ an object in $V^{[*]}\Mod^{gr}_+$. Therefore the graded structure on $M^{[*]}$ is not preserved under homomorphisms in $ V^{[*]}\Mod^{adm}$.

Finally, an \textbf{ordinary} dg module is a dg module $M^{[*]}$ such that $ L_0$ is semisimple with eigenspace decomposition $M^{[*]}=\bigoplus_{\lambda \in \mathbb{C}} M_\lambda^{[*]}$  \begin{itemize}\setlength\itemsep{5pt}
\item[(i)] $\on{dim} M_\lambda^{[*]} < \infty$ for all $\lambda \in \mathbb{C}$,
\item[(ii)] $M_{\lambda+n}^{[*]}=0$ for a fixed $\lambda$ and $n$ small enough,
\item[(iii)] $M_\lambda^{[*]}$ is an eigenspace of $L_0$ of eigenvalue $\lambda$.
\end{itemize}

By definition $ M_\lambda^{[*]}=\on{Ker}(L_0-\lambda: M^{[*]}\to M^{[*]})$ is subcomplex of $M^{[*]}$ since $ L_0$ commutes with the differential $ d_M$. 

We will use $V^{[*]}\Mod^{ord}$ to denote the full subcategory of $V^{[*]}\Mod^w$ of all ordinary modules. The condition (i) means the complex $M_\lambda^{[*]}$ has finite dimensional total space.

In general, given two modules $M^{[*]}$ and $N^{[*]}$ in $V^{[*]}\Mod^{ord}$, the modules $P^{[*]}$ appearing in the short exact sequence
\[ 0\longrightarrow N^{[*]} \longrightarrow P^{[*]} \longrightarrow M^{[*]} \longrightarrow 0\]
in the category $V^{[*]}\Mod^{w}$ need not be in $V^{[*]}\Mod^{ord}$, i.e., $V^{[*]}\Mod^{ord}$ need not be closed under extensions in $V^{[*]}\Mod^{w}$. Requiring the operator $L_0$ to act semisimply on $M^{[*]}$ seems to be too strong a condition.  A weak module $M^{[*]}$ is called \textbf{logarithmic} if it satisfies all conditions of ordinary modules except  (iii)  is relaxed by assuming that $ M_\lambda^{[*]}$ is generalized eigenspace for $L_0$, i.e., $M_\lambda^{[*]} =\bigcup _{m>0}\on{Ker}((L_0-\lambda)^m)$. Once again $M_\lambda^{[*]}$ is also a subcomplex of $M^{[*]}$. We will use $V^{[*]}\Mod^{log}$ to denote the full subcategory of $V^{[*]}\Mod^{w}$ consisting all such weak modules $M^{[*]}$.  However, $V^{[*]}\Mod^{log}$ is closed under extensions. 

By using the identity
\[ [L_0, v_n]=(\operatorname{wt}(v)-n-1)v_n\] 
for homogeneous $v$ in $V^{[*]}$, we can see that for any logarithmic module $M^{[*]}$, the subcomplex
\[ \bigoplus_{\lambda}\on{Ker}((L_0-\lambda): M_\lambda^{[*]} \to M_\lambda^{[*]})\]
of $M^{[*]}$ is a dg submodule of $M^{[*]}$. Hence we have that every irreducible logarithmic module is ordinary. 

\delete{
The operators $ L_{-1}, \; L_0,\;  L_1$ acting on all ordinary modules form a Lie algebra isomorphic to $ \mathfrak s\mathfrak l_2$. The category $\mathcal O^-$ for $ \mathfrak{sl}_2$ corresponding to lowest weight modules is defined as the Lie algebra modules $M$ such that $L(0)$ acts semisimply with finite dimensional weight spaces and $L(1)$ (corresponding to the negative roots) acts locally nilpotently.  It was shown in \cite{Dong-Lin-Mason} that every module in $ \mathcal O^-$ can be decomposed uniquely as direct sum of (possibly infinitely many) indecomposable modules in $\mathcal O^-$. In particular every indecomposable $N\in V\Mod^{ord}$ must have the form $N=\bigoplus_{n\geq 0}N_{\lambda+n}$ for a unique $\lambda\in \mathbb C$ such that $N_\lambda\neq 0$. 

 it is a corollary because $\mathcal O^-$ has the Krull-Remak-Schmidt property and $V-Mod^{ord}$ is a subcategory. If we take an ordinary $V$-module $M$, we can decompose it into a direct sum of $\mathfrak{sl}_2$-modules, and then regroup the summands into $V$-modules. By repeating this process on the new $V$-modules, we obtain a decomposition of $M$ into a direct sum of indecomposable $V$-modules, and this process will end because each summand $M_\lambda$ is finite dimensional. Thus $V\Mod^{ord}$ verifies the Krull-Remak-Schmidt property.
 }

\begin{proposition}\label{prop:log}
Every module $M^{[*]}$ in $V^{[*]}\Mod^{log}$ can decomposed uniquely into a direct sum of (possibly infinitely many) indecomposable dg modules such that each indecomposable component appears only finitely many times.
\end{proposition}

The proof is similar to the proof for ordinary modules. The uniqueness follows from the fact that  each generalized weight subcomplex is of finite total dimension. 

\begin{remark}
The functor $H^{[*]}(-)$ sends a weak (resp. graded, admissible, ordinary, logarithmic) $V^{[*]}$-module to a weak (resp. graded, admissible, ordinary, logarithmic) $H^{[*]}(V^{[*]})$-module \end{remark}

\begin{remark} One could also define a dg module $M^{[*]}$ for dg vertex operator algebra as cohomologically ordinary if its cohomology dg module $H^{[*]}(M^{[*]})$ is an ordinary $H^{[*]}(V^{[*]})$-module. 
\end{remark}

The result of Proposition~\ref{prop:log} defines a canonical graded structure $M^{[*]}(\bullet)$ on every logarithmic module $M^{[*]}$ so that $(M^{[*]},M^{[*]}(\bullet))$ is in $V^{[*]}\Mod^{gr}_+$. Indeed, we have $M^{[*]}=\bigoplus_{n \in \mathbb{N}}M^{[*]}(n)$ where $M^{[*]}(n)=\bigoplus_{\lambda \in \Lambda}M_{\lambda+n}^{[*]}$ and $\Lambda=\{\lambda \in \cc \ | \ M_{\lambda -m}^{[*]}=0 \text{ if } m \in \mathbb{N}^*\}$. It follows that every logarithmic dg module is also an admissible module with a well-constructed gradation. We can summarize the relations between the different types of modules we defined as follows:
\begin{align*}
V^{[*]}\Mod^{ord} \subseteq V^{[*]}\Mod^{log} \subseteq V^{[*]}\Mod^{adm} \subseteq V^{[*]}\Mod^w.
 \end{align*}

Before closing this section, we introduce the following definition:
\begin{definition}
A dg vertex operator algebra $V^{[*]}$ is \textbf{rational} if every admissible $V^{[*]}$-module is a direct sum of irreducible admissible $V^{[*]}$-modules. In particular, a rational dg vertex operator algebra is semisimple, i.e., it is a direct sum of irreducible modules for itself.  We call a dg vertex operator algebra \textbf{cohomologically rational} if its cohomology vertex operator algebra is rational in the category of $\Vecs$. 
\end{definition}

\section{$C_2$-algebras and their Poisson module categories}\label{sec:5}

\subsection{$C_2$-algebras}
Let $V$ be a  vertex  algebra. The \textbf{$\mathbf{C_2}$-algebra} $R(V)$ was introduced by Y. Zhu in \cite{Zhu} and is defined as $R(V)=V/C_2(V)$ where $C_2(V)$ is linearly spanned by the set $\{a_{-2}b \ | \ a,b \in V\}$. This quotient has a commutative Poisson algebra structure given by 
\begin{align*}
 \overline{a} \cdot \overline{b}= \overline{a_{-1}b}  \quad  \text{ and } \quad \{\overline{a}, \overline{b}\} =\overline{a_0b} \quad \text{ for } a,b \in V,
\end{align*}
where $\overline{x}$ is the image of $x \in V$ in $R(V)$. If $R(V)$ is finite dimensional then $V$ is said to be \textbf{$\mathbf{C_2}$-cofinite}.

We now assume $V^{[*]}=(V^{[*]}, Y(\cdot, x), \mbf 1)$ is a dg vertex algebra. Set $C_2^{[*]}(V^{[*]})=\on{Span}\{u_{-2}v \ | \ u, v \in V^{[*]}\}$ and define the \textbf{$\mathbf{C_2}$-dg algebra} as the quotient $R^{[*]}(V^{[*]})=V^{[*]}/C_2^{[*]}(V^{[*]})$. As above, we will write $\overline{u}$ for the class of $u \in V^{[*]}$ in $R^{[*]}(V^{[*]})$. We have the following result:

\begin{theorem}\label{thm:R(V)dgPoissonAlgebra}
The space $R^{[*]}(V^{[*]})$ is a dg Poisson algebra with a graded commutative product $\overline{u} \cdot \overline{v}= \overline{u_{-1}v}$ and a Poisson bracket $\{\overline{u}, \overline{v}\}_{R(V)} =\overline{u_0v}$. If $V^{[*]}$ is a dg vertex operator algebra, then $R^{[*]}(V^{[*]})$ is a weight graded dg Poisson algebra
\[
R^{[*]}(V^{[*]})=\bigoplus_{(p, n) \in \mathbb{Z} \times \frac{1}{2}\mathbb{Z}}R^{[p]}(V^{[*]})_n.
\]
with $R^{[p]}(V^{[*]})_n$ the image of $V^{[p]}_n$ in $R^{[*]}(V^{[*]})$.
\end{theorem}

\begin{proof}
Let $V^{[*]}$ be a dg vertex algebra and let $u, v$ be homogeneous elements of $V^{[*]}$. Then
\[
d_V(u_{-2}v)=d_V(u)_{-2}(v)+(-1)^{|u|-2(-2)+2}u_{-2}(d_V(v)) \in C_2^{[*]}(V^{[*]}).
\]
Furthermore $u_{-2}v \in V^{[|u|+|v|+2]}$ so $d_V(u_{-2}v) \in V^{[|u|+|v|+3]}$. Thus $d_V: C_2^{[k]}(V^{[*]}) \longrightarrow C_2^{[k+1]}(V^{[*]})$, and $(C_2^{[*]}(V^{[*]}),d_V^{[*]})$ is a subcomplex of $(V^{[*]},d_V^{[*]})$. The quotient space $(R^{[*]}(V^{[*]}),d_V^{[*]})$ is then a well-defined cochain complex.

We have a translation operator $\mathcal{D}$ satisfying $Y(\mathcal{D}v, x)=\frac{d}{dx}Y(v, x)$. It follows that for any $v \in V^{[*]}$ and $n \geq 2$, we have
\[
v_{-n}=\frac{1}{(n-1)!}(\mathcal{D}^{n-2}v)_{-2}.
\]
Therefore one sees that $C_2^{[*]}(V^{[*]})=\on{Span}\{u_{-k}v \ | \ u, v \in V^{[*]}, k \geq 2\}$. It can be verified that the operations $\overline{u} \cdot \overline{v}$ and $\{\overline{u}, \overline{v}\}_{R(V)}$ given in the statement of the theorem are well-defined, and that they satisfy the following identities:
\begin{itemize}\setlength\itemsep{5pt}
\item $\{u, C_2^{[*]}(V^{[*]}) \}_{R(V)} \subset C_2^{[*]}(V^{[*]})$,
\item $u. C_2^{[*]}(V^{[*]}) \subset C_2^{[*]}(V^{[*]})$,
\item $\{u,v\}_{R(V)}+(-1)^{|u||v|}\{v,u\}_{R(V)}  \in C_2^{[*]}(V^{[*]})$,
\item $\{\{u,v\}_{R(V)}, w\}_{R(V)}+(-1)^{|u||v|}\{v,\{u, w\}_{R(V)}\}_{R(V)}-\{u, \{v, w\}_{R(V)}\}_{R(V)}  \in C_2^{[*]}(V^{[*]})$,
\item $u.v-(-1)^{|u||v|}v.u \in C_2^{[*]}(V^{[*]})$,
\item $u.(v.w)-(u.v).w \in C_2^{[*]}(V^{[*]})$,
\item $\{u, v.w\}_{R(V)}-\{u,v\}_{R(V)}.w-(-1)^{|u||v|}v.\{u, w\}_{R(V)} \in C_2^{[*]}(V^{[*]})$,
\item $d_V(u.v)-(d_V(u).v+(-1)^{|u|}u.d_V(v)) \in C_2^{[*]}(V^{[*]})$,
\item $d_V(\{u,v\}_{R(V)})-(\{d_V(u),v\}_{R(V)}+(-1)^{|u|}\{u, d_V(v)\}_{R(V)}) \in C_2^{[*]}(V^{[*]})$.
\end{itemize}

We now assume that $V^{[*]}=\bigoplus_{(p, n) \in \mathbb{Z} \times \frac{1}{2}\mathbb{Z}}V^{[p]}_n$ is a dg vertex operator algebra. So if $u \in V^{[*]}$, then $u=\sum_{(p, n)}u^p_n$, with $u^p_n \in V^{[p]}_n$. Thus any element of $C_2^{[*]}(V^{[*]})$ can be written $u_{-2}v=\sum_{p, p', n, n'}(u^p_n)_{-2}(v^{p'}_{n'})$, and $(u^p_n)_{-2}(v^{p'}_{n'}) \in V^{[p+p'+2]}_{n+n'+1} \cap C_2^{[*]}(V^{[*]})$. Therefore the weight homogeneous components of the elements of $C_2^{[*]}(V^{[*]})$ are still in $C_2^{[*]}(V^{[*]})$. It then follows that $C_2^{[*]}(V^{[*]})=\bigoplus_{(p, n)} C_2^{[p]}(V^{[*]})_n$, and thus $R^{[*]}(V^{[*]})=\bigoplus_{(p, n)  \in \mathbb{Z} \times \frac{1}{2}\mathbb{Z}}V^{[p]}_n/C_2^{[p]}(V^{[*]})_n$. Furthermore, we know that $u_{-1}v \in V^{[|u|+|v|]}_{\on{wt}(u)+\on{wt}(v)}$, and so $R^{[p]}(V^{[*]})_n.R^{[p']}(V^{[*]})_{n'} \subset R^{[p+p']}(V^{[*]})_{n+n'}$. Hence the $\mathbb{Z} \times \frac{1}{2}\mathbb{Z}$-gradation is respected by the product.
\end{proof}

\begin{remark}
Although $R^{[*]}(V^{[*]})$ can be defined for any dg vertex algebra and its Poisson structure is independent of the choice of the conformal vector $\omega$, the conformal structure determines the $\mathbb{Z} \times \frac{1}{2}\mathbb{Z}$-gradation $R^{[*]}(V^{[*]})$. 
\end{remark}

\delete{
Corresponding to the algebra $R(V)$ are two module categories: $R(V)\Mod$ of all $R(V)$-modules, and $R(V)\Mod^{gr}$ of all graded $R(V)$-modules. There is a forgetful functor $R(V)\Mod^{gr} \longrightarrow R(V)\Mod$. }

\subsection{Associated dg Poisson $R^{[*]}(V^{[*]})$-modules}
Let $M^{[*]}$ be a dg module for a dg vertex algebra $V^{[*]}$. Define $C^{[*]}_2(M^{[*]})=\on{Span}\{v_{-2}m \ | \ v \in V^{[*]}, m \in M^{[*]}\}$ and set $R^{[*]}(M^{[*]})=M^{[*]}/C^{[*]}_2(M^{[*]})$. As $Y_M(\cdot, x)$ is a chain map, we have $Y_M(\cdot, x) \circ d_V=d_{End(M)[[x,x^{-1}]]} \circ Y_M(\cdot, x)$. Therefore $d_V(v)_n=d_{End(M)[[x,x^{-1}]]}(v_n)=d_M \circ v_n-(-1)^{|v_n|}v_n \circ d_M$. It follows that $d_M(v_{-2}m) \in C^{[*]}_2(M^{[*]})$, and $(C^{[*]}_2(M^{[*]}), d_M^{[*]})$ is a subcomplex of $(M^{[*]}, d_M^{[*]})$. Thus the quotient $R^{[*]}(M^{[*]})$ is a well-defined complex. The action of  $R^{[*]}(V^{[*]})$ on $R^{[*]}(M^{[*]})$ is given by
\begin{align*}
\overline{v}.\overline{m}=\overline{v_{-1}m}, 
\end{align*}
with $v \in V^{[*]}, \ m \in M^{[*]}, \ \overline{v} \in R^{[*]}(V^{[*]}), \ \overline{m} \in R^{[*]}(M^{[*]})$.
There is also another map $\{-, -\}_{R(M)}: R^{[*]}(V^{[*]})\otimes R^{[*]}(M^{[*]})\to R^{[*]}(M^{[*]})$ defined by
\[ \{ \overline{v}, \overline{m}\}_{R(M)}=\overline{v_0(m)}.
\]
\begin{theorem}
 If $M^{[*]}$ is dg module for a dg vertex algebra, then $R^{[*]}(M^{[*]})$ is a dg Poisson module for $R^{[*]}(V^{[*]})$. If $M^{[*]} \in V^{[*]}\Mod^{gr}$ for a dg vertex operator algebra $V^{[*]}$, then $R^{[*]}(M^{[*]})$ is also a $\mathbb{Z} \times \frac{1}{2}\mathbb{Z}$-graded module for $R^{[*]}(V^{[*]})$.
\end{theorem}

\begin{proof}
Let $M^{[*]}$ is dg module for a dg vertex algebra $V^{[*]}$. We verify the following points by explicit computations:
\begin{itemize}\setlength\itemsep{5pt}
\item $(R^{[*]}(M^{[*]}),d_M^{[*]})$ is a dg module (with action of degree $0$) for $R^{[*]}(V^{[*]})$ as a dg algebra,
\item $(R^{[*]}(M^{[*]}),\{-,-\}_{R(M)}, d_M^{[*]})$ is a dg Lie algebra module for $R^{[*]}(V^{[*]})$ as a dg Lie algebra, with an action $\{-,-\}$ of degree $0$.
\item $\{\overline{u}, \overline{v}.\overline{m}\}_{R(M)}=\{\overline{u},\overline{v}\}_{R(V)}.\overline{m}+(-1)^{|\overline{u}||\overline{v}|}\overline{v}.\{\overline{u},\overline{m}\}_{R(M)}$,
\item $\{\overline{u}.\overline{v},\overline{m}\}_{R(M)}=\overline{u}.\{\overline{v},\overline{m}\}_{R(M)}+(-1)^{|\overline{u}||\overline{v}|}\overline{v}.\{\overline{u},\overline{m}\}_{R(M)}$.
\end{itemize}
These properties make $R^{[*]}(M^{[*]})$ a dg Poisson module for $R^{[*]}(V^{[*]})$.

Now let $V^{[*]}$ be a dg vertex operator algebra and $M^{[*]}=\bigoplus_{n \in \frac{1}{2}\mathbb{Z}} M^{[*]}(n)$ a graded dg $V^{[*]}$-module. We apply the same reasoning to $R^{[*]}(M^{[*]})$ as we previously did to $R^{[*]}(V^{[*]})$, and we see that
\[
R^{[*]}(M^{[*]})=\bigoplus_{(p,n) \in \mathbb{Z} \times \frac{1}{2}\mathbb{Z}} R^{[p]}(M^{[*]})(n),
\]
with $R_V^{[p]}(M^{[*]})(n)=M^{[p]}(n)/ \left(M^{[p]}(n) \cap C_2^{[*]}(M^{[*]})\right)$. Likewise, we can check that
\[
R^{[p]}(V^{[*]})_n.R^{[p']}(M^{[*]})(n') \subset R^{[p+p']}(M^{[*]})(n+n').
\]
\end{proof}

If $f: M^{[*]} \longrightarrow N^{[*]}$ is a homomorphism of weak modules, then by definition we have $f(C_2^{[*]}(M^{[*]}))\subseteq C_2^{[*]}(N^{[*]})$ and so it induces a linear map $R(f): R^{[*]}(M^{[*]}) \longrightarrow R^{[*]}(N^{[*]})$, which is clearly a homomorphism of dg $R^{[*]}(V^{[*]})$-modules. Thus we have an additive functor $R: V^{[*]}\Mod^w \longrightarrow R^{[*]}(V^{[*]})\Mod$, and it is right exact. \\
If $M^{[*]}$ is in $V^{[*]}\Mod^{gr}$, we have seen that $R^{[*]}(M^{[*]})$ is also naturally graded. So the restriction of $R$ gives a functor $R^{gr}: V^{[*]}\Mod^{gr} \longrightarrow R^{[*]}(V^{[*]})\Mod^{gr}$ with the following commutative diagram:
 \begin{center}
 \begin{tikzpicture}[scale=.9, transform shape]
\tikzset{>=stealth}
\node (1) at (-2,3) []{$V^{[*]}\Mod^{gr}$};
\node (2) at (-2,0) []{$V^{[*]}\Mod^{w}$};
\node (3) at (2,3) []{$R^{[*]}(V^{[*]})\Mod^{gr}$};
\node (4) at (2,0) []{$R^{[*]}(V^{[*]})\Mod$};
\draw[->]  (1) -- node[] {} (2);
\draw[->]  (1) -- node[above] {$R^{gr}$} (3);
\draw[->]  (3) -- node[] {} (4);
\draw[->]  (2) -- node[below] {$R$} (4);
\end{tikzpicture}
\end{center}

\section{Zhu algebras and associated graded algebras}\label{sec:6} 

\subsection{The Zhu algebra}

Let $V^{[*]}$ be a dg vertex operator algebra. We consider the analogue of the Zhu algebra (cf. \cite{Zhu}) in the dg setting. There are two filtrations on $V^{[*]}$: one is 
\[ \cdots \subseteq F^p V^{[*]}\subseteq F^{p+1} V^{[*]}\subseteq \cdots \subseteq 
V^{[*]}
\]
with $ F^p V^{[*]}=\oplus_{i\leq p}V^{[i]}$ making $(V^{[*]}, F^\bullet, d_V)$ a differential  filtered vector space (cf.  Section~\ref{sec:df-algebra}), and the other is a weight filtration
  \[ \cdots \subseteq W_n V^{[*]}\subseteq W_{n+1} V^{[*]}\subseteq \cdots \subseteq 
V^{[*]}
\]
 with $ W_n V^{[*]}=\oplus_{m\leq n}V^{[*]}_m$ making $(V^{[*]}, W_\bullet, F^\bullet, d_V)$ a weight filtered differential  filtered vector space (cf.  Section~\ref{sec:weight-filtration}). It is clear to note that 
 \[ V^{[*]}=\bigcup_{p\in \bb Z}F^pV^{[*]}=\bigcup_{n\in \frac{1}{2}\bb Z}W_nV^{[*]}\; \text{ and }\;
 \bigcap_{p\in \bb Z}F^pV^{[*]}=\bigcap_{n\in \frac{1}{2}\bb Z}W_nV^{[*]}=\{0\}.
 \]
 For any $ v\in V^{[*]}$ we define $ |v|$ and $ \on{wt}(v)$ as in Section~\ref{sec:df-algebra}. In particular, if $v= \sum_i v_i$ is a finite sum  with each $0\neq v_i$ being weight homogeneous (in some $V^{[*]}_{n_i}$), then $\on{wt}(v)=\max\{ n_i\}$. Similarly,  if $v= \sum_i v^i$ is a finite sum with each $0\neq v^i$ being differential homogeneous (in some $V^{[p_i]}$), then $|v|=\max\{ p_i\}$. \begin{definition}
We define the following bilinear operation on $V^{[*]}$. For $u \in V^{[*]}$ homogeneous, $v \in V^{[*]}$,
\begin{align*}
u \ast v =\displaystyle \on{Res}_x \left( Y(u, x)\frac{(1+x)^{\on{wt}(u)}}{x}v \right)=\sum_{n \geq 0}\binom{\on{wt}(u)}{n}u_{n-1}v.
\end{align*}
Let $O(V^{[*]})$ be the subspace of $V^{[*]}$ spanned by the elements of the form
\begin{align*}
u \circ v= \on{Res}_x \left( Y(u, x)\frac{(1+x)^{\on{wt}(u)}}{x^2}v \right)=\sum_{n \geq 0}\binom{\on{wt}(u)}{n}u_{n-2}v.
\end{align*}
Define $A(V^{[*]})=V^{[*]}/O(V^{[*]})$, where we write $[v]$ for the image of $v \in V^{[*]}$ in $A(V^{[*]})$. The space $A(V^{[*]})$ is called the \textbf{Zhu algebra} of $V^{[*]}$. 
\end{definition}

The following lemma will be very useful. It is again an adaptation of a result of Zhu (cf. \cite[Lemma 2.1.3]{Zhu}) in the dg setting.

\begin{lemma}\label{lem:uv-vu}
For any $u \in V^{[|u|]}_{\on{wt}(u)}$ and $v \in V^{[|v|]}_{\on{wt}(v)}$, we have:
\begin{align*}
\resizebox{\textwidth}{!}{
$u \ast v -(-1)^{|u||v|}v \ast u \equiv \on{Res}_x \big (Y(u, x)(1+x)^{\on{wt}(u)-1}v \big) =\displaystyle \sum_{n \geq 0}\binom{\on{wt}(u)-1}{n}u_{n}v \mod O(V^{[*]}).$
}
\end{align*}
In particular, 
\begin{align}
u \ast v -(-1)^{|u||v|}v \ast u &\in F^{|u|+|v|-2}V^{[*]}  \mod  O(V^{[*]}), \label{lem:uv-vu(1)} \\
u \ast v -(-1)^{|u||v|}v \ast u &\in W_{\on{wt}(u)+\on{wt}(v)-1}V^{[*]}  \mod O(V^{[*]}), \label{lem:uv-vu(2)} \\
u \ast v -(-1)^{|u||v|}v \ast u -u_0(v)&\in F^{|u|+|v|-4}V^{[*]}  \mod  O(V^{[*]}), \label{lem:uv-vu(3)} \\
u \ast v -(-1)^{|u||v|}v \ast u -u_0(v)&\in W_{\on{wt}(u)+\on{wt}(v)-2}V^{[*]}  \mod O(V^{[*]}), \label{lem:uv-vu(4)}
\end{align}
Taking $u=\omega$ we get 
\[ \omega\ast v-v\ast \omega =(L_{-1}+L_0)(v) \in O(V^{[*]}).
\]
\end{lemma}

\begin{lemma} If $ v, v'\in V^{[*]}$ are dg homogeneous such that $ v\equiv v'\mod O(V^{[*]})$ and $v \not\in O(V^{[*]})$, then $ |v|\equiv|v'| \mod 2$. 
\end{lemma}
\begin{proof} Set $V^{[*]}=V^{[odd]}\oplus V^{[ev]}$ with $ V^{[odd]}=\bigoplus_{i\in \bb Z}V^{[2i+1]}$ and  $V^{[ev]}=\bigoplus_{i\in \bb Z}V^{[2i]}$. We write
\[ v=v'+\sum_{(a, b)\in I}a\circ b\]
with $I$ being a finite subset of pairs $ (a, b)$ of dg homogeneous non-zero elements $a$ and $ b$ in $V^{[*]}$. Let $I^{odd}$ (resp. $I^{ev}$) be the subset of the pairs $ (a, b)\in I$ such that $ |a|+|b|$ is odd  (resp. even).  Then $ \sum_{(a, b)\in I^{odd}}a\circ b\in V^{[odd]}$ and  $ \sum_{(a, b)\in I^{ev}}a\circ b\in V^{[ev]}$. If $ |v|\not\equiv |v'| \mod 2$, say, $|v|$ is  odd, then $ v\in V^{[odd]}$ and $ v'\in V^{[ev]}$. Thus $ v=\sum_{(a,b)\in I^{odd}}a\circ b \in O(V^{[*]})$, which contradicts the assumption that $v \not\in O(V^{[*]})$.
\end{proof}
The next theorem is the adaptation of the proof of Zhu in the dg setting:

\begin{theorem}\emph{\textbf{(\cite{Zhu}).}}\label{thm:Zhu}
$O(V^{[*]})$ is a two sided ideal for the multiplication $\ast$, and so $\ast$ is defined on $A(V^{[*]})$. Moreover:
\begin{enumerate}\setlength\itemsep{5pt}
\item $A(V^{[*]})$ is an associative algebra under multiplication $\ast$.
\item The image of the vacuum vector $\mathbf{1}$ is the unit of the algebra $A(V^{[*]})$. 
\item $A(V^{[*]})$ is a differential filtered algebra with the ascending filtration
$(F^pA(V^{[*]}))_{p \in \bb{Z}}$ where $F^pA(V^{[*]})$ is the image of 
$\bigoplus_{q  \leq p}V^{[q]}_*$ in $A(V^{[*]}).$
\item The image of the conformal vector $\omega$ is in the center of $A(V^{[*]})$ for the super commutative bracket $[-,-]^s$, i.e., $ [\omega] \ast [v] =[v]\ast [\omega] $
for all $ [v]\in A( V^{[*]})$.
\item $A(V^{[*]})$ has a weight filtration 
$(W_nA(V^{[*]}))_{n \in \frac{1}{2} \mathbb{Z}}$ where $W_nA(V^{[*]})$ is the image of $\bigoplus_{i  \leq n}V^{[*]}_i$ in the $A(V^{[*]}).$

\end{enumerate}
\end{theorem}

\begin{remark}
Notice that while $d_V(O(V^{[*]})) \subseteq O(V^{[*]})$, the space $O(V^{[*]})$ is not a subcomplex of $V^{[*]}$. Indeed, it does not decompose as a direct sum of homogeneous components. Therefore $A(V^{[*]})$ is not a dg algebra. We will see later how we will recover the dg setting from a subsequent construction.
\end{remark}

The associative algebra $A(V^{[*]})$ has two filtered structures, so we can consider the categories $A(V^{[*]})\Mod$ of $A(V^{[*]})$-modules, and $A(V^{[*]})\Mod^{F^\bullet}_+$ (resp. $A(V^{[*]})\Mod^{W_\bullet}_+$) of filtered $A(V^{[*]})$-modules $M^{[*]}$ with $F^pM^{[*]}=0$ for $p \ll 0$ (resp. $W_nM^{[*]}=0$ for $n \ll 0$), together with a natural forgetful functors $ A(V^{[*]})\Mod^{F^\bullet}_+ \longrightarrow A(V^{[*]})\Mod$ (resp. $ A(V^{[*]})\Mod^{W_\bullet}_+ \longrightarrow A(V^{[*]})\Mod$).

\begin{theorem}\label{thm:A(V)_dfcommutative}
$(A(V^{[*]}), F^\bullet A(V^{[*]}), \ast, d_V)$ is a differential filtered commutative algebra.
\end{theorem}

\begin{proof}
The differential filtered structure (cf. Section~\ref{sec:df-algebra}) is already known from Theorem~\ref{thm:Zhu}. It remains to prove the commutativity statement. Let $x \in F^{n}A(V^{[*]})$ and $y \in F^{m}A(V^{[*]})$ with $n=|x|$ and $m=|y|$. There exist $u \in V^{[n]}$ and $v \in V^{[m]}$ such that $[u]-x \in F^{n-1}A(V^{[*]})$ and $[v]-y \in F^{m-1}A(V^{[*]})$. We write $u=\sum_i u^i$ and $v=\sum_j v^j$ where $u^i \in V^{[n]}_i$ and $v^j \in V^{[m]}_j$. Then 
\begin{align*}
\begin{array}{rcl}
x \ast y-(-1)^{|x||y|}y \ast x &= & [u] \ast [v]-(-1)^{nm}[v] \ast [u] + F^{n+m-1}A(V^{[*]}), \\[5pt]
&= &\displaystyle \sum_{i, j} \left([u^i] \ast [v^j]-(-1)^{nm}[v^j] \ast [u^i]\right) + F^{n+m-1}A(V^{[*]}), \\[5pt]
& = &\displaystyle \sum_{i, j} \sum_{k \geq 0}\binom{i-1}{k}(u^i)_{k}v^j + O(V^{[*]}) +F^{n+m-1}V^{[*]}, \\[5pt]
& \in &\displaystyle  O(V^{[*]}) +F^{n+m-1}V^{[*]},
\end{array}
\end{align*}
where we have used Lemma~\ref{lem:uv-vu} and the fact that $|(u^i)_k v^j|=n+m-2k-2$. It follows that $x \ast y-(-1)^{|x||y|}y \ast x  \in F^{|x|+|y|-1}A(V^{[*]})$, and so $A(V^{[*]})$ is differential filtered commutative.
\end{proof}

\subsection{The filtration by degree}

The associated graded algebra with respect to the degree filtration is $\gr^{[*]}A(V^{[*]})=\bigoplus_{p \in \mathbb{Z}}F^{p+1}A(V^{[*]})/F^{p}A(V^{[*]})$. As $(A(V^{[*]}), F^\bullet, d_V)$ 
is differential filtered commutative, 
$\gr^{[*]}A(V^{[*]})$ 
is a differential graded commutative 
algebra for the product induced by 
$\ast$. If $x$ is an homogeneous element of $F^{|x|}A(V^{[*]})$, 
we will write $\widetilde{x}=x+ F^{|x|-1} A(V^{[*]})\in \gr^{[|x|]}A(V^{[*]}) $
for the class of $x$ in $\gr^{[*]}A(V^{[*]})$.

\begin{lemma}\label{lem:degree_mod_2}
Set $u \in V^{[p]}$ such that $[u] \in F^{p-1}A(V^{[*]})$. Then there exist finitely many dg homogeneous $f_i \in V^{[*]}$ with $|f_i| \leq p-2$ such that
\[ u+\sum_i f_i  \in O(V^{[*]}),\]
and
\[  |f_i| \equiv p \ \on{mod} \  2, \text{ for all } i. \]
\end{lemma}

\begin{proof}
Set $u$ as in the statement. Then $u \in O(V^{[*]})+F^{p-1}V^{[*]}$, and thus $u$ is the highest degree component of an element in $O(V^{[*]})$. We decompose this element as $\sum_{l \in L} a_l \circ b_l=\sum_{l \in L} \sum_{k \geq 0} \binom{\on{wt}(a_l)}{k}(a_l)_{k-2}b_l$ with $a_l, b_l$ dg homogeneous and weight homogeneous. As $|(a_l)_{k-2}b_l|=|a_l|+|b_l|-2k+2$, we see that all the degrees of the dg homogeneous components of $a_l \circ b_l$ are all congruent to $|a_l|+|b_l|$ modulo $2$. We can thus split our previous sum into
\[ \sum_{l \in L'} a_l \circ b_l+\sum_{l \in L''} a_l \circ b_l, \]
where $L'$ (resp. $L''$) is the set of indices such that $|a_l|+|b_l| - p \equiv 0 \ \on{mod} \ 2$ (resp. $|a_l|+|b_l| -p \equiv 1 \ \on{mod} \ 2$). It follows that all the components of the first sum have a degree with the same parity as $p$, while the other sum has the opposite parity. Hence $u$ is a summand in the sum over $L'$. Furthermore, as $u$ is the highest component of the sum over $L$, and that the sums over $L'$ and $L''$ have different parities and thus cannot cancel each other, it follows that $u$ is the highest degree term of $\sum_{l \in L'} a_l \circ b_l$. 

Set $X=\sum_{l \in L'} a_l \circ b_l$. The dg homogeneous components of $X-u$ will be written as $f_i$. Based on our construction, for each $f_i$, there exist $l \in L'$ and $k \geq 0$ such that $|f_i| = |(a_l)_{k-2}b_l|$, and so $|f_i| \equiv |a_l|+|b_l| \equiv p \ \on{mod} \  2$. In particular, as $|f_i|<p$, we have $|f_i| \leq p-2$. 
\end{proof}

We know that for $x \in F^{|x|}A(V^{[*]})$, $y \in F^{|y|}A(V^{[*]})$, there exist $u \in V^{[|x|]}$, $v \in V^{[|y|]}$ such that $[u]-x \in F^{|x|-1}A(V^{[*]})$ and $[v]-y \in F^{|y|-1}A(V^{[*]})$. We define
\begin{align*}
\{\widetilde{x}, \widetilde{y} \}_{F^\bullet}  =  [u] \ast [v]-(-1)^{|u||v|}[v] \ast [u] \ \on{mod} \  F^{|x|+|y|-3}A(V^{[*]}).
\end{align*}

\begin{lemma}\label{lem:Poisson_defined}
The map
\[ \{. ,  . \}_{F^\bullet}: \gr^{[p]}A(V^{[*]}) \otimes \gr^{[q]}A(V^{[*]}) \longrightarrow \gr^{[p+q-2]}A(V^{[*]})\]
is well-defined.
\end{lemma}

\begin{proof}
Fix $x \in F^pA(V^{[*]})$, $y \in F^qA(V^{[*]})$, and let $u' \in V^{[p]}$, $v' \in V^{[q]}$ be other elements such that $[u']-x \in F^{p-1}A(V^{[*]})$ and $[v'] - y \in F^{q-1}A(V^{[*]})$. Then $[u]-[u'] \in F^{p-1}A(V^{[*]})$. Based on Lemma~\ref{lem:degree_mod_2}, we can write $u-u'=X-\sum_i f_i$ where $X \in O(V^{[*]})$ and the $f_i$ are dg homogeneous in $V^{[*]}$ with $|f_i| \leq p-2$ and $|f_i| \equiv p \ \on{mod} \  2$. Likewise, we decompose $v-v'=Y-\sum_j g_j$ where $Y \in O(V^{[*]})$, $|g_j| \leq q-2$ and $|g_j| \equiv q \ \on{mod} \  2$. Therefore we can write
\begin{align*}
\begin{array}{rcl}
[u] \ast [v]-(-1)^{pq}[v] \ast [u] & = &\displaystyle [u'+X-\sum_i f_i] \ast [v'+Y-\sum_j g_j] \\[5pt]
 & &\displaystyle  -(-1)^{nm}[v'+Y-\sum_j g_j] \ast [u'+X-\sum_i f_i], \\[5pt]
 & = & u' \ast v'-(-1)^{pq}v' \ast u' \\[5pt]
 & &\displaystyle - \sum_j(u' \ast g_j -(-1)^{pq} g_j \ast u') \\[5pt]
 & &\displaystyle - \sum_i(f_i \ast v'-(-1)^{pq} v' \ast f_i) \\[5pt]
 & &\displaystyle + \sum_{i, j} (f_i \ast g_j -(-1)^{pq} g_j \ast f_i) \ \on{mod} \  O(V^{[*]}).
\end{array}
\end{align*}
Using the parities of the degrees of $f_i$ and $g_j$, we have
\[pq \equiv |f_i|q \equiv p|g_j| \equiv |f_i||g_j|  \ \on{mod} \ 2,\]
and so 
\[(-1)^{pq} = (-1)^{|f_i|q} = (-1)^{p|g_j|}= (-1)^{|f_i||g_j|}.\]
We can thus apply Lemma~\ref{lem:uv-vu} to the expression previously obtained and see that
\begin{align*}
[u] \ast [v]-(-1)^{pq}[v] \ast [u] \equiv  u' \ast v'-(-1)^{pq}v' \ast u' \ \on{mod} \  \left( O(V^{[*]})+F^{p+q-3}V^{[*]}\right),
\end{align*}
and so the map $\{.,.\}_{F^\bullet}$ is well-defined.
\end{proof}

\delete{
\begin{lemma}
The map
\[ \{. ,  . \}_V: \gr^{[*]}A(V^{[*]}) \otimes \gr^{[*]}A(V^{[*]}) \longrightarrow \gr^{[*]}A(V^{[*]})\]
is well-defined
\end{lemma}

\begin{proof}
Fix $x \in F^nA(V^{[*]}), y \in F^mA(V^{[*]})$ and let $u' \in V^{[n]}$, $v' \in V^{[m]}$ be other elements such that $[u']-x \in F^{n-1}A(V^{[*]})$ and $[v']-y \in F^{m-1}A(V^{[*]})$. Then $u-u' \in O(V^{[*]})+F^{n-1}V^{[*]}$. We write $u-u'=X+\sum_i f_i$ where $X \in O(V^{[*]})$ and $f_i$ is dg homogenenous in $V^{[*]}$ with $|f_i| < n$. Likewise, we decompose $v-v'=Y+\sum_j g_j$ where $Y \in O(V^{[*]})$ and $|g_j| < m$. Therefore we can write
\begin{align*}
\begin{array}{rcl}
[u] \ast [v]-(-1)^{nm}[v] \ast [u] & = &\displaystyle [u'+X+\sum_i f_i] \ast [v'+Y+\sum_j g_j] \\[5pt]
 & &\displaystyle  -(-1)^{nm}[v'+Y+\sum_j g_j] \ast [u'+X+\sum_i f_i], \\[5pt]
 & = & u' \ast v'-(-1)^{nm}v' \ast u' \\[5pt]
 & &\displaystyle + \sum_j(u' \ast g_j -(-1)^{nm} g_j \ast u') \\[5pt]
 & &\displaystyle + \sum_i(f_i \ast v'-(-1)^{nm} v' \ast f_i) \\[5pt]
 & &\displaystyle + \sum_{i, j} (f_i \ast g_j -(-1)^{nm} g_j \ast f_i) +O(V^{[*]}).
\end{array}
\end{align*}
As $u-u'$ is dg homogeneous, it means that $u-u'$ is the highest degree term of an element in $O(V^{[*]})$. We decompose this element as $\sum_l a_l \circ b_l=\sum_l \sum_{k \geq 0} \binom{\on{wt}(a_l)}{k}(a_l)_{k-2}b_l$ with $a_l, b_l$ dg homogeneous and weight homogeneous. Hence for each $l$, there exists $k_l \geq 0$ such that $\sum_l  \binom{\on{wt}(a_l)}{k_l}(a_l)_{k_l-2}b_l=u-u'$, and the terms $(a_l)_{k-2}b_l$ for $k < k_l$ are either zero or cancelled by other terms. It follows that for all $l$ we have $|(a_l)_{k_l-2}b_l|=|u|=|u'|$, and so $|a_l|+|b_l| \equiv n \ \on{mod} \  2$. But we know that for each $f_i$, there exist $l$ and $k$ such that $|f_i| = |(a_l)_{k-2}b_l|$, and so $|f_i| \equiv n \ \on{mod} \  2$. In particular, $|f_i| \leq n-2$. Similarly, we have $|g_j| \equiv m \ \on{mod} \  2$ for all $j$. This means that we have 
\[nm \equiv |f_i|m \equiv n|g_j| \equiv |f_i||g_j|  \ \on{mod} \ 2,\]
and so 
\[(-1)^{nm} = (-1)^{|f_i|m} = (-1)^{n|g_j|}= (-1)^{|f_i||g_j|}.\]
We can thus apply Lemma~\ref{lem:uv-vu} to the expression previously obtained and see that
\begin{align*}
[u] \ast [v]-(-1)^{nm}[v] \ast [u] \equiv  u' \ast v'-(-1)^{nm}v' \ast u' \ \on{mod} \  \left( O(V^{[*]})+F^{n+m-3}V^{[*]}\right),
\end{align*}
and so the map $\{.,.\}_V$ is well-defined.
\end{proof}}

\begin{proposition}\label{prop:grA(V)isDGPoisson}
 $\gr^{[*]} A(V^{[*]})$ is a graded commutative dg Poisson algebra in the category $ \Vecs$. The product is induced by $\ast$, the Poisson bracket is $\{.,.\}_{F^\bullet}$ and of degree $-2$, and the differential is induced by $d_V$. In particular, for $\widetilde{x}, \widetilde{y} \in \gr^{[*]} A(V^{[*]})$ dg homogeneous with respective preimages $u \in V^{[|\widetilde{x}|]}$, $v \in V^{[|\widetilde{y}|]}$, we have
 \[ \widetilde{x} \ast \widetilde{y}=\widetilde{[u_{-1}v]} \quad   \text{ and }  \quad     \{ \widetilde{x}, \widetilde{y} \}_{F^\bullet}=\widetilde{[u_{0}v]}.
  \]
\end{proposition}

\begin{proof}
For $x, y, z \in A(V^{[*]})$, we can verify by direct computation that:
\begin{itemize}\setlength\itemsep{5pt}
\item $\{\widetilde{x}, \widetilde{y} \}_{F^\bullet}+(-1)^{|\widetilde{x}||\widetilde{y}|}\{\widetilde{y}, \widetilde{x}\}_{F^\bullet}=0$,
\item $\{\widetilde{x}, \{ \widetilde{y}, \widetilde{z}\}_{F^\bullet} \}_{F^\bullet}=\{\{\widetilde{x}, \widetilde{y}\}_{F^\bullet}, \widetilde{z}\}_{F^\bullet}+(-1)^{|\widetilde{x}||\widetilde{y}|}\{\widetilde{y}, \{ \widetilde{x}, \widetilde{z}\}_{F^\bullet} \}_{F^\bullet}$,
\item $\{\widetilde{x},  \widetilde{y} \ast \widetilde{z}\}_{F^\bullet} =\{\widetilde{x}, \widetilde{y}\}_{F^\bullet} \ast \widetilde{z}+(-1)^{|\widetilde{x}||\widetilde{y}|}\widetilde{y} \ast \{ \widetilde{x}, \widetilde{z}\}_{F^\bullet} $,

\item $\widetilde{x} \ast \widetilde{y}-(-1)^{|\widetilde{x}||\widetilde{y}|} \widetilde{y} \ast \widetilde{x}=0$.
\end{itemize}
We define $d_{F^\bullet}: \gr^{[*]} A(V^{[*]}) \longrightarrow \gr^{[*]} A(V^{[*]})$ by $d_{F^\bullet}(\widetilde{x})=\widetilde{d_Vx}$ on homogeneous elements and then expand linearly. We can then verify explicitly that:
\begin{itemize}\setlength\itemsep{5pt}
\item $d_{F^\bullet} \circ d_{F^\bullet}=0$,
\item $d_{F^\bullet}(\widetilde{x} \ast \widetilde{y})=d_{F^\bullet}(\widetilde{x}) \ast \widetilde{y}+(-1)^{|\widetilde{x}|} \widetilde{x} \ast d_{F^\bullet}(\widetilde{y})$,
\item $d_{F^\bullet}(\{\widetilde{x}, \widetilde{y} \}_{F^\bullet})=\{d_{F^\bullet}(\widetilde{x}), \widetilde{y} \}_{F^\bullet}+(-1)^{|\widetilde{x}|}  \{\widetilde{x},  d_{F^\bullet}(\widetilde{y})\}_{F^\bullet}$.
\end{itemize}
\end{proof}

As in the classical case (cf. \cite{Arakawa-Lam-Yamada}), there is a relation between $R^{[*]}(V^{[*]})$ and $\gr^{[*]} A(V^{[*]})$. We define the following map:
\begin{align*}
\begin{array}{rccc}
\eta_{F^\bullet}: & R^{[*]}(V^{[*]}) & \longrightarrow & \gr^{[*]} A(V^{[*]}) \\[5pt]
            & u+C_2^{[|u|]}(V^{[*]}) & \longmapsto & \displaystyle u+O(V^{[*]})+\bigoplus_{p<|u|}V^{[p]}.
\end{array}
\end{align*}
Because $C_2^{[|u|]}(V^{[*]}) \subseteq O(V^{[*]})+\bigoplus_{p<|u|}V^{[p]}$, the map $\eta_{F^\bullet}$ is well-defined.

\begin{proposition}\label{prop:eta_VsurjectiveDG}
The map $\eta_{F^\bullet}$ is a surjective morphism of dg Poisson algebras.
\end{proposition}

\begin{proof}
The surjectivity is clear by the construction of $\gr^{[*]} A(V^{[*]})$. We can verify by direct computations that $\eta_{F^\bullet}$ preserves the algebra structures as well as the Poisson brackets. Finally it is straightforward to see that $\eta_{F^\bullet}$ is a chain map.
\end{proof}

\subsection{The filtration by weight}

As we proved in Theorem~\ref{thm:A(V)_dfcommutative}, the algebra $(A(V^{[*]}),$ $F^\bullet A(V^{[*]}), d_V, \ast)$ is a differential filtered algebra. It also has a weight filtered differential filtered structure (cf. Section~\ref{sec:weight-filtration}) where $W_nA(V^{[*]})$ is the image of $\sum_{m \leq n}V_m^{[*]}$ in $A(V^{[*]})$. Indeed, it is clearly an increasing filtration, and because $L_0(d_V(u))=|u|d_V(u)$ for any $u \in V^{[|u|]}$, the space $W_nA(V^{[*]})$ is stable by $d_V$. We can thus look at the associated graded algebra $\gr_{*}A(V^{[*]})=\bigoplus_{n \in \frac{1}{2}\mathbb{Z}}W_nA(V^{[*]})/W_{n-1}A(V^{[*]})$. We will first prove the following result:

\begin{proposition}\label{prop:gr_*(A(V))}
The algebra $\gr_{*}A(V^{[*]})$ is a weight graded, differential filtered algebra. Furthermore, the product is differential filtered commutative.
\end{proposition}

\begin{proof}
The algebra $\gr_{*}A(V^{[*]})$ is equipped with an ascending filtration $F^p \gr_{n}A(V^{[*]})= F^pA(V^{[*]}) \cap W_nA(V^{[*]}) \on{mod} W_{n-1}A(V^{[*]})$. The differential $d_V$ on $V^{[*]}$ induces a map on $\gr_{*}A(V^{[*]})$ that we denote $d$, and it is straightforward to verify that $(\gr_{*}A(V^{[*]}), F^\bullet, d)$ is a differential filtered vector space (cf. Section~\ref{sec:df-algebra}). Furthermore, the weight graded statement is clear by construction.

Set $\widetilde{x} \in F^p \gr_{n}A(V^{[*]})$ with $p=|\widetilde{x}|$. We can write $\widetilde{x} = [u] \on{mod} W_{n-1}A(V^{[*]})$, with $u=\sum_{l \leq p}u^l$ such that $|u^l|=l$. Likewise, for $\widetilde y \in F^q \gr_{m}A(V^{[*]})$, we can write $\widetilde y = [v] \on{mod} W_{m-1}A(V^{[*]})$, with $v=\sum_{k \leq q}v^k$ such that $|v^k|=k$. Using these decompositions, we show by direct computation that
\[
d(\widetilde{x} \ast \widetilde y)-\left(d(\widetilde{x}) \ast \widetilde y + (-1)^p \widetilde{x} \ast d(\widetilde y)\right) \in F^{|\widetilde{x}|+|\widetilde y|} \gr_{m+n}A(V^{[*]}).
\]
Hence $(\gr_{*}A(V^{[*]}), F^\bullet, d, \ast)$ is a differential filtered algebra. The differential filtered commutativity is verified using the same decomposition for $\widetilde{x}$ and $\widetilde y$, as well as Formulas~\eqref{lem:uv-vu(1)} and \eqref{lem:uv-vu(2)}.
\end{proof}

\begin{lemma}\label{lem:degree_mod_2_variant}
Set $u \in V_n^{[*]}$ such that $[u] \in W_{n-1}A(V^{[*]})$. We write $u=\sum_{q \leq p} u^q$ with $|u^q|=q$. Then there exist finitely many dg homogeneous and weight homogeneous $f_i \in V^{[*]}$ with $\on{wt}(f_i) \leq n$ such that
\[ u^p+\sum_i f_i  \in O(V^{[*]}),\]
and
\[  |f_i| \equiv p \ \on{mod} \  2, \text{ for all } i. \]
\end{lemma}

\begin{proof}
Set $u$ as in the statement. Then $u \in O(V^{[*]})+W_{n-1}V^{[*]}$, and thus $u$ is the highest weight component of an element in $O(V^{[*]})$. We decompose this element as $\sum_{l \in L} a_l \circ b_l=\sum_{l \in L} \sum_{k \geq 0} \binom{\on{wt}(a_l)}{k}(a_l)_{k-2}b_l$ with $a_l, b_l$ dg homogeneous and weight homogeneous. As $|(a_l)_{k-2}b_l|=|a_l|+|b_l|-2k+2$, we see that the degrees of the dg homogeneous components of $a_l \circ b_l$ are all congruent to $|a_l|+|b_l|$ mod $2$. We can thus split our previous sum into:
\[
\sum_{l \in L'} a_l \circ b_l+\sum_{l \in L''} a_l \circ b_l,
\]
where $L'$ (resp. $L''$) is the set of indices such that $|a_l|+|b_l| - p \equiv 0 \ \on{mod} \ 2$ (resp. $|a_l|+|b_l| -p \equiv 1 \ \on{mod} \ 2$). It follows that all the components of the first sum have a degree with the same parity as $p$, while the other sum has the opposite parity. Hence $u^p$ is a summand in the sum over $L'$. As the sums over $L'$ and $L''$ have different degree parities, they cannot cancel each other. It follows that $u^p$, being of highest weight in the sum over $L$, is a summand of the highest weight term of $\sum_{l \in L'} a_l \circ b_l$. 

Set $X=\sum_{l \in L'} a_l \circ b_l$. The dg homogeneous and weight homogeneous components of $X-u^p$ will be written as $f_i$. Based on our construction, for each $f_i$, there exist $l \in L'$ and $k \geq 0$ such that $|f_i| = |(a_l)_{k-2}b_l|$, and so $|f_i| \equiv |a_l|+|b_l| \equiv p \ \on{mod} \  2$. As $n$ is the highest weight of the sum over $L'$, it follows that $| f_i | \leq n$.
\end{proof}

\begin{remark}
We cannot define a Poisson structure on $\gr_{*}A(V^{[*]})$ in the same manner as we did in Lemma~\ref{lem:Poisson_defined}. Indeed, we have Lemma~\ref{lem:degree_mod_2_variant}, which is a variant of Lemma~\ref{lem:degree_mod_2}. However in Lemma~\ref{lem:degree_mod_2_variant}, we do not have the condition $\on{wt}(f_i) < n$. Because of this, the natural map given by the super commutator is not well-defined on $\gr_{*}A(V^{[*]})$. Furthermore, the algebra $\gr_{*}A(V^{[*]})$ is not a differential complex because $d(\gr_{n}A(V^{[*]})) \subseteq \gr_{n}A(V^{[*]})$. Therefore there is no dg Poisson structure on this algebra.
\end{remark}

As done using the filtration by the cohomological degree, we define the following map:
\begin{align*}
\begin{array}{rccc}
\eta_{W_\bullet}: & R^{[*]}(V^{[*]}) & \longrightarrow & \gr_{*} A(V^{[*]}) \\[5pt]
            & u+C_2^{[*]}(V^{[*]})_{\on{wt}(u)} & \longmapsto & \displaystyle u+O(V^{[*]})+\bigoplus_{n< \on{wt}(u)}V_n^{[*]}.
\end{array}
\end{align*}
Because $C_2(V^{[*]})_{\on{wt}(u)} \subseteq O(V^{[*]})+\bigoplus_{m<\on{wt}(u)}V^{[*]}_{m}$, the map $\eta_{W_\bullet}$ is well-defined.

\begin{proposition}\label{prop:eta_VsurjectiveDF}
The map $\eta_{W_\bullet}$ is a surjective morphism of weight graded, differential filtered, differential filtered commutative algebras.
\end{proposition}

\begin{proof}
The proof is similar to Proposition~\ref{prop:eta_VsurjectiveDG}.
\end{proof}

\begin{remark}
The morphism $\eta_{W_\bullet}$ has no reason to be an isomorphism. For example, in the classical case of the simple affine vertex operator algebra $L_{\hat{\mathfrak{g}}}(k, 0)$ where $\mathfrak{g}$ is of type $E_8$ with the adjoint representation, the morphism $\eta_{W_\bullet}$ is not injective, because $\on{dim} R(V)> \on{dim} A(V)$ (cf. \cite[Section IV.3]{Gaberdiel-Gannon}).
\end{remark}

\subsection{Using both filtrations successively}

We have looked at both gradations on $A(V^{[*]})$ separately. By looking at them successively and using Lemma~\ref{lem:2gradations}, we see that $ \gr_{*}\gr^{[*]}A(V^{[*]})$ and $ \gr^{[*]} \gr_{*}A(V^{[*]})$ are isomorphic weight graded dg algebras.

\begin{proposition}\label{prop:double_graded_A(V)}
The algebra $ \gr_{*}\gr^{[*]}A(V^{[*]}) \cong \gr^{[*]} \gr_{*}A(V^{[*]})$ is a weight graded dg Poisson algebra. The Poisson bracket is of degree $-2$ for the cohomological degree, and of degree $-1$ for the weight. Furthermore, the dg grading makes it a graded commutative algebra.
\end{proposition}

\begin{proof}
The associated graded algebra is given by:
\begin{align*}
 \gr_{*}\gr^{[*]}A(V^{[*]})=\bigoplus_{p \in \mathbb{Z}}\bigoplus_{n \in \frac{1}{2}\mathbb{Z}} \left(V_n^{[p]} \ \on{mod} \  I^{[p]}_n\right),
\end{align*}
where $I^{[p]}_n=O(V^{[*]})+\bigoplus_{\substack{q \leq p \\ m<n}}V^{[q]}_m+\bigoplus_{\substack{q<p \\ m \leq n}}V_m^{[q]}$. The product is induced by the $\ast$-product in $V^{[*]}$. It is clear that, based on the properties of the product $\ast$, this algebra is graded by both weight and cohomological degree. Set $\widetilde{x} \in \gr_{n_1} \gr^{[p_1]}A(V^{[*]})$ and $\widetilde{y} \in \gr_{n_2} \gr^{[p_2]}A(V^{[*]})$. Then there exist $u \in V_{n_1}^{[p_1]}$, $v \in V_{n_2}^{[p_2]}$ such that their images in $ \gr_{*}\gr^{[*]}A(V^{[*]})$ are $\widetilde{x}$ and $\widetilde{y}$. The bracket in $ \gr_{*}\gr^{[*]}A(V^{[*]})$ is given by:
\begin{align*}
\{ \widetilde{x}, \widetilde{y}\}_{F^\bullet W_\bullet}=u \ast v- (-1)^{p_1p_2} v \ast u \ \on{mod} \  I^{[p_1+p_2-2]}_{n_1+n_2-1}.
\end{align*}
Using Lemma~\ref{lem:uv-vu}, we see that $u \ast v- (-1)^{p_1p_2} v \ast u \in I^{[p_1+p_2-2]}_{n_1+n_2-1} + V_{n_1+n_2-1}^{[p_1+p_2-2]}$, so its image in $ \gr_{*}\gr^{[*]}A(V^{[*]})$ is in the summand $\gr_{n_1+n_2-1} \gr^{[p_1+p_2-2]}A(V^{[*]})$.

We need to verify that $\{ \widetilde{x}, \widetilde{y}\}_{F^\bullet W_\bullet}$ is well-defined. Let $u', v'$ be other preimages of $ \widetilde{x}$ and $\widetilde{y}$ respectively. Then
\[u-u' \in I^{[p_1]}_{n_1}.
\]
Hence we can write $X=u-u'+z+\sum_{i}f_i$ where $X \in O(V^{[*]})$, $z \in \bigoplus_{ m < n_1}V_m^{[p_1]}$, and the $f_i$'s are homogeneous in both cohomological degree and weight, and satisfy $|f_i| < p_1$, $\on{wt}(f_i) \leq n_1$. It follows that $u-u'+z$ satisfies the assumption of Lemma~\ref{lem:degree_mod_2}, and so we can assume that $|f_i| \equiv p_1 \ \on{mod} \ 2$. Similarly, we can write $Y=v-v'+w+\sum_{j}g_j$, where $Y \in O(V^{[*]})$, $w \in \bigoplus_{ m < n_2}V_m^{[p_2]}$, and the $g_j$'s satisfy $|g_j| \equiv p_2 \ \on{mod} \ 2$. We then obtain:
\begin{align*}
\begin{array}{rcl}
u \ast v-(-1)^{p_1p_2} v \ast u & = &\displaystyle (u'+X-z-\sum_i f_i)\ast (v'+Y-w-\sum_j g_j) \\[5pt]
 & &\displaystyle  -(-1)^{p_1p_2}(v'+Y-w-\sum_j g_j) \ast (u'+X-z-\sum_i f_i), \\[5pt]
 & = & u' \ast v'-(-1)^{p_1p_2}v' \ast u' \\[5pt]
 & &\displaystyle - \sum_j(u' \ast g_j -(-1)^{p_1p_2} g_j \ast u') \\[5pt]
 & &\displaystyle - \sum_i(f_i \ast v'-(-1)^{p_1p_2} v' \ast f_i) \\[5pt]
 & &\displaystyle + \sum_{i, j} (f_i \ast g_j -(-1)^{p_1p_2} g_j \ast f_i) \\[5pt]
  & &\displaystyle - (z \ast v' -(-1)^{p_1p_2} v' \ast z) \\[5pt]
 & &\displaystyle + z \ast w -(-1)^{p_1p_2} w \ast z \\[5pt]
 & &\displaystyle + \sum_{j} (z \ast g_j -(-1)^{p_1p_2} g_j \ast z)  \\[5pt]
 & &\displaystyle -  (u' \ast w -(-1)^{p_1p_2} w \ast u') \\[5pt]
  & &\displaystyle + \sum_{i} (f_i \ast w -(-1)^{p_1p_2} w \ast f_i) \ \on{mod} \  O(V^{[*]}).
\end{array}
\end{align*}
Using the fact that $|f_i| \equiv p_1 \ \on{mod} \  2$ and $|g_j| \equiv p_2 \ \on{mod} \  2$, the formulas~\eqref{lem:uv-vu(1)} and \eqref{lem:uv-vu(2)} in Lemma~\ref{lem:uv-vu} imply that all expressions in the right hand side except the first one first one are in $I^{[p_1+p_2-2]}_{n_1+n_2-1}$. We have thus found that
\begin{align*}
u \ast v-(-1)^{p_1p_2} v \ast u  \equiv  u' \ast v'-(-1)^{p_1p_2} v' \ast u' \ \on{mod} \  I^{[p_1+p_2-2]}_{n_1+n_2-1},
\end{align*}
and so the map $\{.,.\}_{F^\bullet W_\bullet}$ is well-defined. It is clear that this map is of degree $-2$ for the cohomological degree, and of degree $-1$ for the weight. The rest of the proof is then similar to that of Proposition~\ref{prop:grA(V)isDGPoisson}.
\end{proof}

We have seen in Theorem~\ref{thm:R(V)dgPoissonAlgebra} the $R^{[*]}(V^{[*]})$ is a weight graded dg Poisson algebra, as is the previous associated graded algebra. We define the following map:
\begin{align*}
\begin{array}{rccc}
\eta_{F^\bullet W_\bullet}: & R^{[*]}(V^{[*]}) & \longrightarrow &  \gr_{*} \gr^{[*]} A(V^{[*]}) \\[5pt]
            & u+C_2^{[p]}(V^{[*]})_n & \longmapsto & \displaystyle u+O(V^{[*]})+\bigoplus_{\substack{q \leq p \\ m<n}}V^{[q]}_m+\bigoplus_{\substack{q<p \\ m \leq n}}V_m^{[q]}
\end{array}
\end{align*}
for $u \in V_n^{[p]}$. Using a reasoning similar to that of Proposition~\ref{prop:eta_VsurjectiveDG}, it is straightforward to show the following:

\begin{proposition}\label{prop:eta_double_grading}
The map $\eta_{F^\bullet W_\bullet}$ is a surjective morphism of weight graded dg Poisson algebras.
\end{proposition}

\section{Associated $A(V^{[*]})$-bimodules and associated graded modules}\label{sec:7}

We have seen that one can associate an $R^{[*]}(V^{[*]})$-module to a $V^{[*]}$-module. We can do something similar for the Zhu algebra. In fact, as we will describe in this section, we have the following diagram for a dg vertex operator algebra $V^{[*]}$:

\begin{equation}\label{fig:losange}
\begin{tikzpicture}[baseline=(current  bounding  box.center)]
\tikzset{>=stealth}
\node (1) at (-2,0) []{$V^{[*]}\Mod^{gr}_+$};
\node (2) at (3,2) []{$A(V^{[*]})\biMod^{fil}_+$};
\node (3) at (3,-2) []{$R^{[*]}(V^{[*]})\Poiss_+$};
\node (4) at (8,0) []{$\gr  A(V^{[*]})\Mod^{gr}_+$};

\draw[->]  (1) -- (2);
\draw[->]  (1) -- (3);
\draw[->]  (2) --  (4);
\draw[->]  (3) -- (4) ;
\end{tikzpicture}
\end{equation}
where $\on{Poiss}$ is the category of Poisson modules, $\on{Mod}^{gr}$ is the category of $\frac{1}{2}\mathbb{Z}$-graded modules, $\on{biMod}^{fil}$ is the category of filtered bimodules. Finally, the index $+$ means that the $n$-th summand of the gradation (or filtration) is zero for $n \ll 0$.

There are in fact three of these diagrams, depending on which filtered structure is considered on the $A(V^{[*]})$-module, i.e., the degree filtration, the weight filtration, or both of them. Hence $\gr A(V^{[*]})$ can be $\gr^{[*]}A(V^{[*]})$, $\gr_{*} A(V^{[*]})$, or $\gr^{[*]} \gr_{*} A(V^{[*]})$.

\subsection{$A(V^{[*]})$-bimodules}

Let $V^{[*]}$ be a dg vertex operator algebra and $M^{[*]}$ an object in $V^{[*]}\Mod^w$. For $u \in V^{[*]}$ weight homogeneous, $m \in M^{[*]}$, set
\begin{align*}
u \circ m= \on{Res}_z \left( Y_M(u, z)\frac{(1+z)^{\on{wt}(u)}}{z^2}m \right).
\end{align*}
We know from Formula \eqref{eq:d_M(v_nu)} that for $u \in V^{[|u|]}$ weight homogeneous and $m \in M^{[*]}$ we have
\begin{align*}
\begin{array}{rcl}
d_M(u \circ m) & = & \displaystyle d_M \left( \sum_{n \geq 0}\binom{\on{wt}(u)}{n}u_{n-2}m \right), \\[20pt]
                        & = & \displaystyle  \sum_{n \geq 0}\binom{\on{wt}(u)}{n}d_V(u)_{n-2}m+(-1)^{|u|}\sum_{n \geq 0}\binom{\on{wt}(u)}{n}u_{n-2}(d_M(m)).
\end{array}
\end{align*}
By definition, we know that $d_V(\omega)=0$, so $\on{wt}(u)d_V(u)=d_V(L(0)u)=d_V(\omega)_1u+(-1)^4L(0)(d_V(u))=L(0)(d_V(u))$, and thus $\on{wt}(d_V(u))=\on{wt}(u)$. It follows that
\begin{align*}
d_M(u \circ m)=d_V(u) \circ m+(-1)^{|u|}u \circ d_M(m) \in O(M^{[*]}).
\end{align*}
Therefore $O(M^{[*]})$ is stable by $d_M$. However, it is not a subcomplex of $M^{[*]}$, because each dg homogeneous component of $u \circ m$ is not necessarily in $O(M^{[*]})$.

We define the following operations of $M^{[*]}$, for homogeneous $u \in V^{[*]}$ and $m \in M^{[*]}$ with respect to both differential grading and weight grading:
\begin{equation*}
\begin{cases}
\begin{array}{rcl}
u \ast_l m & =  &  \displaystyle \on{Res}_z \left( Y_M(u, z)\frac{(1+z)^{\on{wt}(u)}}{z}m \right), \\[5pt]
m \ast_r u & = & (-1)^{|u||m|} \displaystyle \on{Res}_z \left( Y_M(u, z)\frac{(1+z)^{\on{wt}(u)-1}}{z}m \right).
\end{array}
\end{cases}       
\end{equation*}

We denote the product $\ast$ on $V^{[*]}$ by $\ast_l$, and define another product by
\begin{align*}
u \ast_r v & =  (-1)^{|v||u|}\displaystyle \on{Res}_z \left( Y_M(v, z)\frac{(1+z)^{\on{wt}(v)-1}}{z}u \right).
\end{align*}

By doing a similar reasoning to that of \cite[Lemma 2.1.3]{Zhu}, we obtain the following lemma:

\begin{lemma}\label{lem:rightandleftproducts}
For $u, v \in V^{[*]}$ homogeneous with respect to both weight grading and differential grading, we have
\begin{align*}
u \ast_l v \equiv u \ast_r v \mod O(V^{[*]}).
\end{align*}
\end{lemma}

\delete{\begin{remark}
Notice that in the classical vertex operator algebra setting, the products $\ast_l$ and $\ast_r$ in $A(V^{[*]})$ are the same. But in the dg setting, they need to be considered separately.
\end{remark}}

We will need the following lemmas, the proof of first one being the same as in \cite[Lemma 1.5.3]{Frenkel-Zhu}:

\begin{lemma}\label{lem:Frenkel-Zhu}
Set $u \in V^{[*]}$ weight homogeneous and $m \in M^{[*]}$. For any $k \geq n \geq 0$, we have
\begin{align*}
\on{Res}_z \left( Y_M(u, z)\frac{(1+z)^{\on{wt}(u)+n}}{z^{2+k}}m\right) \in O(M^{[*]}).
\end{align*}
\end{lemma}

\begin{lemma}\label{lem:wt(u)-1}
Set $u \in V^{[*]}$ weight homogeneous and $n \geq 1$. Then
\begin{align*}
\begin{array}{l}
\displaystyle \on{Res}_z \left( Y_M(u, z)\frac{(1+z)^{\on{wt}(u)-1}}{z^{n}}m\right) \\[5pt]
\displaystyle \quad \quad \quad \quad \equiv (-1)^{n-1}\on{Res}_z \left( Y_M(u, z)\frac{(1+z)^{\on{wt}(u)-1}}{z}m\right) \ \on{mod} \ O(M^{[*]}).
\end{array}
\end{align*}
\end{lemma}

\begin{proof}
If $n=1$, there is nothing to prove. Assume that $n \geq 2$. We have $\displaystyle \frac{1}{z^n}+\frac{1}{z^{n-1}}=\frac{1+z}{z^{n}}$ so:
\begin{align*}
\begin{array}{r}
\on{Res}_z \left( Y_M(u, z)\frac{(1+z)^{\on{wt}(u)-1}}{z^{n}}m\right)+\on{Res}_z \left( Y_M(u, z)\frac{(1+z)^{\on{wt}(u)-1}}{z^{n-1}}m\right) \\[5pt]
 =\on{Res}_z \left( Y_M(u, z)\frac{(1+z)^{\on{wt}(u)}}{z^{n}}m\right).
 \end{array}
\end{align*}
Based on Lemma~\ref{lem:Frenkel-Zhu}, the right hand side of the equation is in $O(M^{[*]})$. If $n=2$, then we proved the lemma. Assuming that $n \geq 3$, we can write
\begin{align*}
\begin{array}{r}
\on{Res}_z \left( Y_M(u, z)\frac{(1+z)^{\on{wt}(u)-1}}{z^{n-1}}m\right)+\on{Res}_z \left( Y_M(u, z)\frac{(1+z)^{\on{wt}(u)-1}}{z^{n-2}}m\right)  \\[5pt]
 =\on{Res}_z \left( Y_M(u, z)\frac{(1+z)^{\on{wt}(u)}}{z^{n-1}}m\right),
 \end{array}
\end{align*}
and the right hand side is still in $O(M^{[*]})$. We keep repeating this procedure until the denominator of the right hand side is $z^2$ and the term is again in $O(M^{[*]})$. We then go up the chain of equations we obtained and see that:
\begin{align*}
\begin{array}{rcl}
\on{Res}_z \left( Y_M(u, z)\frac{(1+z)^{\on{wt}(u)-1}}{z}m\right) & \equiv &(-1) \on{Res}_z \left( Y_M(u, z)\frac{(1+z)^{\on{wt}(u)-1}}{z^2}m\right),  \\[5pt]
                                                                                              & \equiv &(-1)^2 \on{Res}_z \left( Y_M(u, z)\frac{(1+z)^{\on{wt}(u)-1}}{z^3}m\right),  \\[5pt]
                                                                                               & & \vdots  \\[5pt]
                                                                                              & \equiv &(-1)^{n-1} \on{Res}_z \left( Y_M(u, z)\frac{(1+z)^{\on{wt}(u)-1}}{z^n}m\right)  \on{mod}  O(M^{[*]}).
\end{array}
\end{align*}
\end{proof}

Let $A(M^{[*]})$ denote the quotient $M^{[*]}/O(M^{[*]})$. As in the case of $A(V^{[*]})$, it does not have a dg structure because $O(M^{[*]})$ is not a subcomplex of $M^{[*]}$. 

\begin{theorem}\label{thm:A(M)bimod}
The space $A(M^{[*]})$ with the operations $\ast_l$ and $\ast_r$ is an $A(V^{[*]})$-bimodule.
\end{theorem}

\begin{proof}
For $u, v \in V^{[*]}$ and $m \in M^{[*]}$, we need to prove the following properties:
\begin{align*}
\begin{array}{c}
{ \setlength{\arraycolsep}{20pt} \begin{array}{ll}
(1) \ O(V^{[*]}) \ast_l m \subset O(M^{[*]}) & (2) \ m \ast_r O(V^{[*]}) \subset O(M^{[*]}) \\[5pt]
(3) \ u \ast_l O(M^{[*]})  \subset O(M^{[*]}) & (4) \ O(M^{[*]}) \ast_r u \subset O(M^{[*]}) \\[5pt]
\end{array} } \\
\begin{array}{l}
\hypertarget{Relation(5)}{(5)} \ u \ast_l (m \ast_r v)-\delete{ (-1)^{|u||v|}}(u \ast_l m) \ast_r v \in O(M^{[*]}) \\[5pt]
(6) \ u \ast_l (v \ast_l m)- (u \ast_l v) \ast_l m \in O(M^{[*]}) \\[5pt]
\hypertarget{Relation(7)}{(7)} \ (m \ast_r u) \ast_r v- m \ast_r (u \ast_r v)  \in O(M^{[*]})
\end{array}
\end{array}
\end{align*}
We only give a proof of $(\hyperlink{Relation(7)}{7})$, because it is the most technical one. The others are done in a similar fashion as the proof of \cite[Theorem 1.5.1]{Frenkel-Zhu}.

Using Lemma~\ref{lem:rightandleftproducts}, we have
\begin{align*}
&(-1)^{(|u|+|v|)|m|}m \ast_r (u \ast_r v) =  (-1)^{(|u|+|v|)|m|}m \ast_r (u \ast_l v), \\[5pt]
&= (-1)^{(|u|+|v|)|m|}\sum_{i \geq 0}\binom{\on{wt}(u)}{i} m \ast_r (u_{i-1}v),\\[5pt]
& =\sum_{i \geq 0}\binom{\on{wt}(u)}{i} \on{Res}_w \left(Y_M(u_{i-1}v, w)\frac{(1+w)^{\on{wt}(u)+\on{wt}(v)-i-1}}{w}m \right), \\[5pt]
&= \sum_{i \geq 0}\binom{\on{wt}(u)}{i} \on{Res}_w \on{Res}_{z-w} \left(Y_M(Y_V(u, z-w)v, w) (z-w)^{i-1}\frac{(1+w)^{\on{wt}(u)+\on{wt}(v)-i-1}}{w}m \right), \\[5pt]
&=\on{Res}_w \on{Res}_{z-w} \left(Y_M(Y_V(u, z-w)v, w)\frac{(1+z)^{\on{wt}(u)}(1+w)^{\on{wt}(v)-1}}{w(z-w)}m \right), \\[5pt]
&=\on{Res}_z \on{Res}_w \left(Y_M(u, z) Y_M(v, w) \frac{(1+z)^{\on{wt}(u)}(1+w)^{\on{wt}(v)-1}}{w(z-w)}m \right)  \\[5pt]
&\quad \quad  -(-1)^{|u||v|}\on{Res}_w \on{Res}_z \left(Y_M(v, w) Y_M(u, z)\frac{(1+z)^{\on{wt}(u)}(1+w)^{\on{wt}(v)-1}}{w(z-w)}m \right), \\[5pt]
&=\sum_{i \geq 0}\on{Res}_z \on{Res}_w \left(Y_M(u, z) Y_M(v, w) z^{-1-i} w^i \frac{(1+z)^{\on{wt}(u)}(1+w)^{\on{wt}(v)-1}}{w}m \right) \\[5pt]
&\quad   +(-1)^{|u||v|}\sum_{i \geq 0} \on{Res}_w \on{Res}_z \left(Y_M(v, w) Y_M(u, z)w^{-1-i} z^i \frac{(1+z)^{\on{wt}(u)}(1+w)^{\on{wt}(v)-1}}{w}m \right), \\[5pt]
&\equiv \on{Res}_z \on{Res}_w \left(Y_M(u, z) Y_M(v, w) \frac{(1+z)^{\on{wt}(u)}(1+w)^{\on{wt}(v)-1}}{zw}m \right) \\[5pt]
&\quad   +(-1)^{|u||v|}\sum_{i \geq 0} \on{Res}_w \on{Res}_z \left(Y_M(v, w) Y_M(u, z) z^i \frac{(1+z)^{\on{wt}(u)}(1+w)^{\on{wt}(v)-1}}{w^{2+i}}m \right), \\[5pt]
&=(-1)^{|v||m|}u \ast_l (m \ast_r v) \\[5pt]
& \quad  +(-1)^{|u||v|}\sum_{i \geq 0} \on{Res}_w \on{Res}_z \left(Y_M(v, w) Y_M(u, z) z^i \frac{(1+z)^{\on{wt}(u)}(1+w)^{\on{wt}(v)-1}}{w^{2+i}}m \right), \\[5pt]
&\equiv (-1)^{|v||m|}(u \ast_l m) \ast_r v \\[5pt]
&\quad   +(-1)^{|u||v|}\sum_{i \geq 0} (-1)^{i-1}\on{Res}_w \on{Res}_z \left(Y_M(v, w) Y_M(u, z) z^i \frac{(1+z)^{\on{wt}(u)}(1+w)^{\on{wt}(v)-1}}{w}m \right),
\end{align*}
where the congruences are modulo $O(M^{[*]})$ and the last one comes from (\hyperlink{Relation(5)}{5}) and Lemma~\ref{lem:wt(u)-1}. We know that
\begin{align*}
& u \ast_l m-(-1)^{|u||m|}m \ast_r u \\[5pt]
&  =  \on{Res}_z \left( Y_M(u, z)\frac{(1+z)^{\on{wt}(u)}}{z}m\right)-\on{Res}_z \left( Y_M(u, z)\frac{(1+z)^{\on{wt}(u)-1}}{z}m\right), \\[5pt]
&  =  \on{Res}_z \left( Y_M(u, z)(1+z)^{\on{wt}(u)-1}m\right).
\end{align*}
So we can replace $u \ast_l m$ in our previous equation to obtain
\begin{align*}
&(-1)^{(|u|+|v|)|m|}m \ast_r (u \ast_r v) \\[5pt]
& \equiv  (-1)^{(|u|+|v|)|m|}(m \ast_r u) \ast_r v +(-1)^{|v||m|}\left( \on{Res}_z(Y_M(u, z)(1+z)^{\on{wt}(u)-1}m)\right) \ast_r v \\[5pt]
& \quad  -(-1)^{|u||v|}\sum_{i \geq 0} \on{Res}_w \on{Res}_z \left(Y_M(v, w) Y_M(u, z) (-1)^{i}z^i \frac{(1+z)^{\on{wt}(u)}(1+w)^{\on{wt}(v)-1}}{w}m \right), \\[5pt]
&  =  (-1)^{(|u|+|v|)|m|}(m \ast_r u) \ast_r v \\[5pt]
& \quad + (-1)^{|u||v|}\on{Res}_w \on{Res}_z\left(Y_M(v, w)Y_M(u, z)(1+z)^{\on{wt}(u)-1} \frac{(1+w)^{\on{wt}(v)-1}}{w} m\right) \\[5pt]
&  \quad  -(-1)^{|u||v|}\on{Res}_w \on{Res}_z \left(Y_M(v, w) Y_M(u, z) (1+z)^{-1} \frac{(1+z)^{\on{wt}(u)}(1+w)^{\on{wt}(v)-1}}{w}m \right), \\[5pt]
&  =  (-1)^{(|u|+|v|)|m|}(m \ast_r u) \ast_r v.
\end{align*}
\end{proof}

\subsection{Modules associated to $\gr^{[*]}A(V^{[*]})$}

For a $V^{[*]}$-module $M^{[*]}$, the $A(V^{[*]})$-module $A(M^{[*]})$ has a (differential) filtration $(F^pA(M^{[*]}))_{p \in \mathbb{Z}}$ where $F^pA(M^{[*]})$ is the image of $\bigoplus_{q  \leq p}M^{[q]}_*$ in $A(M^{[*]})$. Set $u \in V^{[|u|]}$ weight homogeneous and $m \in M^{[|m|]}$. Then $[u] \ast_l [m]$ and $[m] \ast_r [u]$ are in $F^{|u|+|m|}A(M^{[*]})$. But we also know that
\begin{align*}
\begin{array}{rcl}
[u] \ast_l [m]- (-1)^{|u||m|}[m] \ast_r [u] & = & \on{Res}_z (Y_M(u, z)(1+z)^{\on{wt}(u)-1}m) \ \on{mod} \  O(M^{[*]}), \\[5pt]
                                                             & = & \sum_{i \geq 0} \binom{\on{wt}(u)-1}{i}u_{i}m \ \on{mod} \  O(M^{[*]}), \\[5pt]
                                                             & \in & F^{|u|+|m|-2}A(M^{[*]}),
\end{array}
\end{align*}
because $|u_i m|=|u|+|m|-2i-2$. As the (differential) filtration on $A(M^{[*]})$ is preserved by the one on $A(V^{[*]})$ by both the right and left actions, the associated graded module $\gr^{[*]} A(M^{[*]})$ is a right and left module for $\gr^{[*]} A(V^{[*]})$. Let $x \in F^p A(V^{[*]})$ and $w \in F^q A(M^{[*]})$. Then there exist $u \in V^{[p]}$ and $m \in M^{[q]}$ such that $[u]-x \in F^{p-1} A(V^{[*]})$ and $[m]-w \in F^{q-1} A(M^{[*]})$. We write $\widetilde{w}=w \ \on{mod} \  F^{q -1}A(M^{[*]})$ for the class of $w$ in $\gr^{[*]}A(M^{[*]})$. We therefore see that
\begin{align*}
\begin{array}{rcl}
\widetilde{x} \ast_l \widetilde{w}- (-1)^{|\widetilde{x}||\widetilde{w}|} \widetilde{w} \ast_r \widetilde{x} & = & [u] \ast_l [m]- (-1)^{|u||m|}[m] \ast_r [u] \ \on{mod} \ F^{p+q-1}A(M^{[*]}), \\[5pt]
                                                & = & 0,
\end{array}
\end{align*}
based on the computation above. Thus we obtain
\begin{align*}
\widetilde{x} \ast_l \widetilde{w}= (-1)^{|\widetilde{x}||\widetilde{w}|} \widetilde{w} \ast_r \widetilde{x}.
\end{align*}

The proof of Lemma~\ref{lem:degree_mod_2} can be adapted to obtain the following result:

\begin{lemma}\label{lem:Module_degree_mod_2}
Set $m \in M^{[p]}$ such that $[m] \in F^{p-1}A(M^{[*]})$. Then there exist finitely many dg homogeneous $h_i \in M^{[*]}$ with $|h_i| \leq p-2$ such that
\[ m+\sum_i h_i  \in O(M^{[*]}),\]
and
\[  |h_i| \equiv p \ \on{mod} \  2, \text{ for all } i. \]
\end{lemma}

We know that for $x \in F^{|x|}A(V^{[*]})$, $w \in F^{|w|}A(M^{[*]})$, there exist $u \in V^{[|x|]}$, $m \in M^{[|w|]}$ such that $[u]-x \in F^{|x|-1}A(V^{[*]})$ and $[m]-w \in F^{|w|-1}A(M^{[*]})$. We define
\begin{align*}
\{\widetilde{x}, \widetilde{w} \}_{F^\bullet}  =  [u] \ast_l [m]-(-1)^{|x||w|}[m] \ast_r [u] \ \on{mod} \  F^{|x|+|y|-3}A(M^{[*]}).
\end{align*}

\begin{lemma}
The map
\[ \{. ,  . \}_{F^\bullet}: \gr^{[p]}A(V^{[*]}) \otimes \gr^{[q]}A(M^{[*]}) \longrightarrow \gr^{[p+q-2]}A(M^{[*]})\]
is well-defined.
\end{lemma}

\begin{proof}
The proof is the same as the one for Lemma~\ref{lem:Poisson_defined}, except we also use Lemma~\ref{lem:Module_degree_mod_2} fo the element in $M^{[*]}$.
\end{proof}

\begin{proposition}\label{prop:grA(M)isDGPoissonModule}
 The module $\gr^{[*]} A(M^{[*]})$ is a dg Poisson module for the dg Poisson algebra $\gr^{[*]} A(V^{[*]})$. The action is induced by $\ast_l$, the Poisson bracket is $\{.,.\}_{F^\bullet}$ and is of degree $-2$, and the differential is induced by $d_M$. In particular, for $\widetilde{x} \in \gr^{[*]} A(V^{[*]})$,  $\widetilde{w} \in \gr^{[*]} A(M^{[*]})$ dg homogeneous with respective preimages $u \in V^{[|\widetilde{x}|]}$, $m \in M^{[|\widetilde{w}|]}$, we have
 \[ \widetilde{x} \ast_l \widetilde{w}=\widetilde{[u_{-1}m]} \quad   \text{ and }  \quad     \{ \widetilde{x}, \widetilde{w} \}_{F^\bullet}=\widetilde{[u_{0}m]}.
  \]
There is a functor
\[
\begin{array}{ccc}
V^{[*]} \Mod^w & \longrightarrow & \gr^{[*]} A(V^{[*]})\Poiss \\[5pt]
M^{[*]} & \longmapsto & \gr^{[*]} A(M^{[*]}).
\end{array}
\] 
\end{proposition}

\begin{proof}
Similarly to Proposition~\ref{prop:grA(V)isDGPoisson}, we define $d_{F^\bullet}: \gr^{[*]} A(M^{[*]}) \longrightarrow \gr^{[*]} A(M^{[*]})$ by $d_{F^\bullet}(\widetilde{w})=\widetilde{d_Mw}$, where $w$ is an element in $F^{|\widetilde{w}|}A(M^{[*]})$ whose image in $ \gr^{[*]} A(M^{[*]})$ is $\widetilde{w}$. For $\widetilde{x}, \widetilde{y} \in \gr^{[*]}A(V^{[*]})$ dg homogeneous and $\widetilde{w} \in \gr^{[|\widetilde{w} |]} A(M^{[*]})$, we can verify by direct computation that:
\begin{itemize}\setlength\itemsep{5pt}
\item $\{\widetilde{x}, \{ \widetilde{y}, \widetilde{w}\}_{F^\bullet} \}_{F^\bullet}=\{\{\widetilde{x}, \widetilde{y}\}_V, \widetilde{w}\}_{F^\bullet}+(-1)^{|\widetilde{x}||\widetilde{y}|}\{\widetilde{y}, \{ \widetilde{x}, \widetilde{w}\}_{F^\bullet} \}_{F^\bullet}$,
\item $\{\widetilde{x},  \widetilde{y} \ast_l \widetilde{w}\}_{F^\bullet}=\{\widetilde{x} ,\widetilde{y} \}_{F^\bullet} \ast_l \widetilde{w} +(-1)^{|\widetilde{x}||\widetilde{y}|} \widetilde{y} \ast_l \{ \widetilde{x}, \widetilde{w} \}_{F^\bullet}$.
\item $\{\widetilde{x} \ast_l \widetilde{y}, \widetilde{w}\}_{F^\bullet}= \widetilde{x} \ast_l \{\widetilde{y}, \widetilde{w}\}_{F^\bullet} +(-1)^{|\widetilde{x}||\widetilde{y}|} \widetilde{y} \ast_l \{ \widetilde{x}, \widetilde{w} \}_{F^\bullet}$.
\item $d_{F^\bullet} \circ d_{F^\bullet}=0$,
\item $d_{F^\bullet}(\widetilde{x} \ast_l \widetilde{w})=d_{F^\bullet}(\widetilde{x}) \ast_l \widetilde{w}+(-1)^{|\widetilde{x}|} \widetilde{x} \ast_l d_{F^\bullet}(\widetilde{w})$,
\item $d_{F^\bullet}(\{\widetilde{x}, \widetilde{w} \}_{F^\bullet})=\{d_{F^\bullet}(\widetilde{x}), \widetilde{w} \}_{F^\bullet}+(-1)^{|\widetilde{x}|}  \{\widetilde{x},  d_{F^\bullet}(\widetilde{w})\}_{F^\bullet}$.
\end{itemize}
The above relations prove the dg Poisson module statement. The expressions for $\ast_l$ and $\{ \cdot, \cdot\}_{F^\bullet}$, as well as the functorial statement are straightforward computations. 
\end{proof}

\subsection{Modules associated to $\gr_{*}A(V^{[*]})$}
We can do something similar to the previous section but using the filtration by weights, starting from a module $M^{[*]} \in V^{[*]}\Mod^{gr}$.

\begin{proposition}\label{prop:grA(M)isDFModule}
 The module $\gr_{*} A(M^{[*]})$ is a weight graded, left differential filtered module for the differential filtered algebra $\gr_{*} A(V^{[*]})$. 
\end{proposition}

\begin{proof}
The proof is essentially the same as in Proposition~\ref{prop:gr_*(A(V))}, where we verify explicitly each axiom in the definition.
\end{proof}

The left and right $A(V^{[*]})$-module structures on $A(M^{[*]})$ induce two $ \gr_{*} A(M^{[*]})$-module structures on $\gr_{*} A(M^{[*]})$ which satisfy some kind of differential filtered commutative relation:

\begin{proposition}
For any $\widetilde{x} \in  F^p \gr_{n} A(V^{[*]})$, $\widetilde{w} \in F^q \gr_{m} A(M^{[*]}) $, we have
\[
\widetilde{x} \ast_l \widetilde{w} - (-1)^{pq} \widetilde{w} \ast_r \widetilde{x} \in F^{p+q-1}  \gr_{n+m} A(M^{[*]}).
\]
\end{proposition}

\begin{proof}
Take $x \in W_n A(V^{[*]})$ a representative of $\widetilde{x}$, and write $x = u +O(V^{[*]})$ with $u \in \bigoplus_{\substack{ i \leq p \\ j \leq n}} V_j^{[i]}$. We decompose $u=\sum_{\substack{ i \leq p \\ j \leq n}}u^i_j$ with $|u_j^i|=i$ and $\on{wt}(u_j^i)=j$. Likewise, we chose a representative $w  \in W_m A(M^{[*]})$ of $\widetilde{w}$, and decompose $w = \sum_{\substack{ i \leq q \\ j \leq m}}m^i_j +O(M^{[*]})$ with $|m_j^i|=i \leq q$ and $\on{wt}(m_j^i)=j \leq m$. We have
\begin{align*}
\begin{array}{l}
\widetilde{x} \ast_l \widetilde{w} -(-1)^{pq} \widetilde{w} \ast_r \widetilde{x} \\[5pt]
 =  x \ast_l w -(-1)^{pq} w \ast_r x + W_{m+n-1} A(M^{[*]}), \\[5pt]
 = \displaystyle \sum_{\substack{ i,  j \\ s, t}} u^i_j \ast_l m^s_t -(-1)^{pq} m^s_t \ast_r u^i_j+O(M^{[*]})+W_{m+n-1}M^{[*]}, \\[5pt]
= \displaystyle \sum_{\substack{ j \leq n \\ t \leq m}} \left( u^p_j \ast_l m^q_t -(-1)^{pq} m^q_t \ast_r u^p_j \right) \\[5pt]
 \quad + \displaystyle \sum_{\substack{ i < p \text{ or } s  < q}} \left( u^i_j \ast_l m^s_t -(-1)^{pq} m^s_t \ast_r  u^i_j \right) +O(M^{[*]})+W_{m+n-1}M^{[*]}.
\end{array}
\end{align*}
Based on the definitions of the left and right actions, we see that for any $j, t$, we have
\[
u^p_j \ast_l m^q_t -(-1)^{pq} m^q_t \ast_r u^p_j = \on{Res}_z(Y_M(u^p_j, z)(1+z)^{j-1} m^q_t) \in F^{p+q-1}M^{[*]}.
\]
Then, when $i < p$ or $s <q$, we have
\[
u^i_j \ast_l m^s_t -(-1)^{pq} m^s_t \ast_r u^i_j  \in F^{i+s}V^{[*]} \subseteq F^{p+q-1}M^{[*]}.
\]
As $\on{wt}(u_j^i) \leq n$ and $\on{wt}(m^s_t) \leq m$, it follows that
\[
\widetilde{x} \ast_l \widetilde{w} -(-1)^{pq} \widetilde{w} \ast_r \widetilde{x} \in F^{p+q-1}A(M^{[*]}) \cap W_{m+n}A(M^{[*]}) \on{mod} W_{m+n-1}A(M^{[*]}),
\]
which is exactly the desired statement.
\end{proof}

\subsection{Using both module filtrations successively}

In Proposition~\ref{prop:double_graded_A(V)}, we used the two filtrations on $A(V^{[*]})$ to define $\gr_* \gr^{[*]} A(V^{[*]})$. If $M^{[*]}$ is an ordinary dg module for $V^{[*]}$, we can do the same for the natural two filtrations on $A(M^{[*]})$ and construct a new module $\gr_* \gr^{[*]} A(M^{[*]})$. The proof of the next result is similar to the proof of Proposition~\ref{prop:double_graded_A(V)}, but instead of taking $\widetilde{y} \in \gr_{n_2} \gr^{[p_2]}A(V^{[*]})$, we work with $\widetilde{w} \in \gr_{n_2} \gr^{[p_2]}A(M^{[*]})$.

\begin{proposition}
For $M^{[*]} \in V^{[*]}\Mod^{ord}$, the space $\gr_* \gr^{[*]} A(M^{[*]})$ is a weight graded dg Poisson module for the weight graded dg Poisson algebra $\gr_* \gr^{[*]} A(V^{[*]})$. The action is induced by $\ast_l$, the Poisson bracket is of degree $-2$ for the cohomological degree, of degree $-1$ for the weight, and the differential is induced by $d_M$.
\end{proposition}

\section{Relations between $R^{[*]}(M^{[*]})$ and the associated graded modules}\label{sec:8}

\subsection{$R^{[*]}(M^{[*]})$  and $\gr^{[*]}A(M^{[*]})$}

In Proposition~\ref{prop:eta_VsurjectiveDG}, we looked at a surjective dg algebra morphism. We can obtain a similar result for the associated modules. The proof of the next result is similar to that of Proposition~\ref{prop:eta_VsurjectiveDG}.

\begin{proposition}\label{prop:eta_MsurjectiveDGmodules}
The map
\begin{align*}
\begin{array}{rccc}
& R^{[*]}(M^{[*]}) & \longrightarrow & \gr^{[*]} A(M^{[*]}) \\[5pt]
            & m+C_2^{[|m|]}(V^{[*]}) & \longmapsto & \displaystyle m+O(M^{[*]})+\bigoplus_{p<|m|}M^{[p]}
\end{array}
\end{align*}
is a surjective morphism of dg Poisson $R^{[*]}(V^{[*]})$-modules, where $\gr^{[*]} A(M^{[*]})$ is seen as an $R^{[*]}(V^{[*]})$-module through $\eta_{F^\bullet}$.
\end{proposition}

We have seen in Proposition~\ref{prop:eta_VsurjectiveDG} that $\eta_{F^\bullet}: R^{[*]}(V^{[*]}) \longrightarrow \gr^{[*]} A(V^{[*]})$ is a surjective morphism of dg Poisson algebras. Hence for an $R^{[*]}(V^{[*]})$-dg Poisson module $W$, the space $\eta_{F^\bullet}(W):= \gr^{[*]} A(V^{[*]}) \otimes_{R^{[*]}(V^{[*]})} W \cong W/((\mathrm{Ker} \ \eta_{F^\bullet})W)$ is a $\gr^{[*]} A(V^{[*]})$-dg Poisson module. If $(W,W(\bullet))$ is graded, then $\eta_{F^\bullet}(W)$ is also naturally graded with $\eta_{F^\bullet}(W)(n)$ being the image of $W(n)$ for all $n \in \frac{1}{2}\mathbb{Z}$. By looking at the two paths in Diagram~\eqref{fig:losange}, we obtain the following result:

\begin{proposition}\label{prop:nat_F}
There is a natural transformation $\psi : \eta_{F^\bullet} \circ R^{[*]}(-) \longrightarrow \gr^{[*]}A(-)$, illustrated by the diagram below:
\begin{center}
 \begin{tikzpicture}[scale=0.9, transform shape]
\tikzset{>=stealth}
\node (1) at (0,0) []{$V^{[*]}\Mod^{w}$};
\node (2) at (6,0) []{$\gr^{[*]} A(V^{[*]})\Mod^{gr}$};

\draw[->]  (1) edge[out=30, in=150]  node[above] {$\gr^{[*]}A(-)$}(2);
\draw[->]  (1) edge[out=-30, in=-150]  node[below] {$\eta_{F^\bullet} \circ R^{[*]}(-)$}(2);
\draw[double distance=1.5pt,-{Stealth[scale=1.2]}] (3,-0.7) -- node[right] {$\psi$} (3,0.7);
\end{tikzpicture}
\end{center}
such that, for any $M^{[*]} \in V^{[*]}\Mod^{w}$, the morphism $\psi_{M^{[*]}}$ is the following surjective morphism of dg Poisson $\gr^{[*]}A(V^{[*]})$-modules:
\begin{align*}
\begin{array}{ccc}
 \eta_{F^\bullet}(R^{[*]}(M^{[*]})) & \stackrel{\psi_{M^{[*]}}}{\longrightarrow} & \gr^{[*]}A(M^{[*]}) \\[5pt]
 m+C_2^{[|m|]}(M^{[*]})+\left( (\on{Ker}  \eta_{F^\bullet})(R^{[*]}(M^{[*]}) \right)^{[|m|]} & \longmapsto & m+O(M^{[*]})+\displaystyle \bigoplus_{p<|m|}M^{[p]}.
\end{array}
\end{align*}
\end{proposition}

\begin{proof}
Fix $M^{[*]} \in V^{[*]}\Mod^{w}$ and $p \in \mathbb{Z}$. The space $C_2^{[p]}(M^{[*]})=C_2^{[*]}(M^{[*]}) \bigcap M^{[p]}$ is spanned by elements of the form $v_{-2}m$ with dg homogeneous $v \in V^{[*]}$, $m \in M^{[*]}$, such that $|v|+|m|+2=p$. We know that $|v_{i-2}m|=|v|+|m|+2-2i$, so if $v_{-2}m \in C_2^{[p]}(M^{[*]})$, then $|v_{i-2}m|=p-2i$. Therefore if $i \geq 1$, we have $|v_{i-2}m| < p$ and $v_{i-2}m \in \bigoplus_{q < p}M^{[q]}$. It follows that $v \circ m=\sum_{i \geq 0}\binom{\on{wt}(v)}{i}v_{i-2}m \equiv v_{-2}m \ \on{mod} \  \bigoplus_{q < p}M^{[q]}$. Thus $C_2^{[p]}(M^{[*]}) \subseteq O(M^{[*]})+\bigoplus_{q < p }M^{[q]}$.

Set $v \in V^{[*]}$ dg homogeneous and consider $\overline{v} \in R^{[*]}(V^{[*]})$. Then $\overline{v} \in \on{Ker}  \eta_{F^\bullet}$ if and only if $v \in O(V^{[*]}) + \bigoplus_{q < |v|}V^{[q]}$. Thus there exists $x \in O(V^{[*]})$ and some dg homogeneous $v^i \in V^{[*]}$ with $|v^i| < |v|$ such that $v=x+\sum_{|v^i|<|v|}v^i$. \\
As $\on{Ker}  \eta_{F^\bullet}$ is a dg Poisson submodule of $R^{[*]}(V^{[*]})$ and $R^{[*]}(M)$ is a dg Poisson module for $R^{[*]}(V^{[*]})$, it follows that $(\on{Ker}  \eta_{F^\bullet})R^{[*]}(M^{[*]})$ is a dg Poisson submodule of $R^{[*]}(M^{[*]})$. An element of $\left( (\on{Ker}  \eta_{F^\bullet})R^{[*]}(M^{[*]}) \right)^{[p]}$ can be written $\sum_{i}a^i_{-1}m^i+C_2^{[p]}(M^{[*]})$ with $a^i +C_2^{[|a^i|]}(V^{[*]}) \in \on{Ker}  \eta_{F^\bullet}$, $m^i +C_2^{[|m^i|]}(M^{[*]}) \in R^{[*]}(M^{[*]})$ and $|a^i|+|m^i|=p$ for all $i$. However, we just saw that for each $i$, we have $a^i=x^i+\sum_{|a^{i, j}|< |a^i|}a^{i, j}$ for some $x^i \in O(V^{[*]})$ and dg homogeneous $a^{i, j} \in V^{[*]}$. Therefore one can write $a^i_{-1}m^i=x^i_{-1}m^i+\sum_{|a^{i, j}| < |a^i|}a^{i, j}_{-1}m^i$. On the other hand, the definition of the $\ast_l$-product gives: 
\begin{align*}
 \begin{array}{cl}
 x^i \ast_l m^i & =a^i \ast_l m^i-\displaystyle \sum_{|a^{i, j}| < |a^i|}a^{i, j} \ast_l m^i, \\
 & =\displaystyle\sum_{n \geq 0}\binom{\on{wt}(a^i)}{n}a^i_{n-1}m^i -\sum_{|a^{i, j}| < |a^i|}\sum_{n \geq 0}\binom{\on{wt}(a^{i, j})}{n}a^{i, j}_{n-1}m^i, \\
  & =x^i_{-1} m^i + \displaystyle\sum_{n \geq 1} \left( \binom{\on{wt}(a^i)}{n}a^i_{n-1}m^i -\sum_{|a^{i, j}| < |a^i|}\binom{\on{wt}(a^{i, j})}{n}a^{i, j}_{n-1}m^i \right).
 \end{array}
\end{align*}
It follows that $x^i \ast_l m^i \equiv x^i_{-1} m^i \ \on{mod} \  \bigoplus_{q<p}M^{[q]}$. As explained in Theorem~\ref{thm:A(M)bimod}, $O(V^{[*]}) \ast_l m \subseteq O(M^{[*]})$ for all $m \in M^{[*]}$, and so $x^i_{-1} m^i \in O(M^{[*]})+ \bigoplus_{q<p}M^{[q]}$, which in turns implies that $\sum_{i}a^i_{-1}m^i \in O(M^{[*]})+ \bigoplus_{q<p}M^{[q]}$. We saw above that $C_2^{[p]}(M) \subseteq O(M^{[*]})+\sum_{q<p}M^{[q]}$, therefore $C_2^{[p]}(M) +\left((\on{Ker}  \eta_{F^\bullet})R^{[*]}(M^{[*]})\right)^{[p]} \subseteq O(M^{[*]})+ \bigoplus_{q<p}M^{[q]}$ and the map $\psi_{M^{[*]}}$ is well-defined and surjective.

There is a natural differential $\widehat{d}$ on $\eta_{F^\bullet}(R^{[*]}(M^{[*]}))$ given by
\[
\widehat{d}(\widehat{m})=\widehat{d_M(m)} 
\]
where $\widehat{m}=m+C_2^{[|m|]}(M) +\left((\on{Ker}  \eta_{F^\bullet})R^{[*]}(M^{[*]})\right)^{[|m|]}$ for any dg homogeneous $m$ in $M^{[*]}$. It is straightforward to check that $\psi_{M^{[*]}} \circ \widehat{d}=d_{F^\bullet} \circ \psi_{M^{[*]}}$. As the morphism $\psi_{M^{[*]}}$ is clearly of degree $0$, it is a chain map.

Set $v \in V^{[*]}$ dg homogeneous and define $\widehat{v}=\overline{v}+\on{Ker} \eta_{F^\bullet} \in R^{[*]}(V^{[*]})/\on{Ker} \eta_{F^\bullet}$. Likewise, set $m \in M^{[*]}$ dg homogeneous and $\widehat{m}=m+C_2^{[|m|]}(M) +\left((\on{Ker}  \eta_{F^\bullet})R^{[*]}(M^{[*]})\right)^{[|m|]}$ $\in \eta_{F^\bullet}(R^{[*]}(M^{[*]}))$. We have:
\begin{align*}
 \begin{array}{cl}
 \widehat{v}. \widehat{m} & = (\overline{v}+(\on{Ker} \eta_{F^\bullet})^{[|v|]})_{-1}(m+C_2^{[|m|]}(M^{[*]})+((\on{Ker}  \eta_{F^\bullet})R^{[*]}(M^{[*]}))^{[|m|])}), \\[5pt]
                             & = v_{-1}m+C_2^{[|v|+|m|]}(M^{[*]}) +\left((\on{Ker}  \eta_{F^\bullet})R^{[*]}(M^{[*]})\right)^{[|v|+|m|]}.
  \end{array}
\end{align*}
We can then write:
\begin{align*}
\begin{array}{cl}
\psi_{M^{[*]}}(\widehat{v}. \widehat{m}) & =v_{-1}m+O(M^{[*]})+\displaystyle \bigoplus_{q< |v|+|m|}M^{[q]}, \\[5pt]
                                    & =v \ast_l m+O(M^{[*]})+\displaystyle \bigoplus_{q< |v|+|m|}M^{[q]}, \\[5pt]
                                    & =[v] \ast_l [m]+F^{|v|+|m|-1}A(M^{[*]}), \\[5pt]
                                    & =\displaystyle \left(v+O(V^{[*]})+\bigoplus_{q<|v|}V^{[q]} \right) \ast_l \left(m+O(M^{[*]})+\bigoplus_{q<|m|}M^{[q]}\right), \\[5pt]
                                    & =\overline{\eta_{F^\bullet}}(\widehat{v}) \ast_l \psi_{M^{[*]}}(\widehat{m}),
\end{array}
\end{align*}
where $\overline{\eta_{F^\bullet}}$ is the isomorphism $R^{[*]}(V^{[*]})/\on{Ker} \eta_{F^\bullet} \cong \gr^{[*]} A(V^{[*]})$. 

There is a Poisson bracket on $ \eta_{F^\bullet}(R^{[*]}(M^{[*]}))$ induced by the one on $R^{[*]}(M^{[*]})$, and it is given by:
\begin{align*}
 \begin{array}{cl}
\{\widehat{v}, \widehat{m}\} & = (\overline{v}+(\on{Ker} \eta_{F^\bullet})^{[|v|]})_{0}(m+C_2^{[|m|]}(M^{[*]})+\left((\on{Ker}  \eta_{F^\bullet})R^{[*]}(M^{[*]})\right)^{[|m|])}), \\[5pt]
                             & = v_{0}m+C_2^{[|v|+|m|-2]}(M^{[*]}) +\left((\on{Ker}  \eta_{F^\bullet})R^{[*]}(M^{[*]})\right)^{[|v|+|m|-2]},
  \end{array}
\end{align*}
so we can then write:
\begin{align*}
\begin{array}{cl}
\psi_{M^{[*]}}(\{\widehat{v}, \widehat{m}\}) & =v_{0}m+O(M^{[*]})+\displaystyle \bigoplus_{q< |v|+|m|-2}M^{[q]}, \\[5pt]
                                    & =\widetilde{[v_{0}m]}, \\[5pt]
                                    & =\{ \widetilde{[v]}, \widetilde{[m]} \}_{F^\bullet}, \\[5pt]
                                    & =\{\overline{\eta_{F^\bullet}}(\widehat{v}), \psi_{M^{[*]}}(\widehat{m})\}_{F^\bullet},
\end{array}
\end{align*}
by Proposition~\ref{prop:grA(M)isDGPoissonModule}. It follows that $\psi_{M^{[*]}}$ is a surjective morphism of dg Poisson $\gr^{[*]}A(V^{[*]})$-modules.

As a morphism in the category $V^{[*]}\Mod^{w}$ is a chain map, one can check directly that for any $f: M^{[*]} \longrightarrow N^{[*]}$ in $V^{[*]}\Mod^{w}$, we have the following commutative diagram:
\begin{center}
 \begin{tikzpicture}[scale=0.9, transform shape]
\tikzset{>=stealth}
\node (1) at (0,0) []{$ \eta_{F^\bullet}(R^{[*]}(M^{[*]}))$};
\node (2) at (4,0) []{$\gr^{[*]} A(M^{[*]})$};
\node (3) at (0,-3) []{$ \eta_{F^\bullet}(R^{[*]}(N^{[*]}))$};
\node (4) at (4,-3) []{$\gr^{[*]} A(N^{[*]})$};

\draw[->]  (1) -- node[above] {$\psi_{M^{[*]}}$}(2);
\draw[->]  (3) -- node[below] {$\psi_{N^{[*]}}$}(4);
\draw[->]  (1) -- node[left] {$ \eta_{F^\bullet}(R^{[*]}(f))$}(3);
\draw[->]  (2) -- node[right] {$\gr^{[*]}A(f)$}(4);
\end{tikzpicture}.
\end{center}
So $\psi$ is a natural transformation.
\end{proof}

\subsection{$R^{[*]}(M^{[*]})$  and $\gr_{*}A(M^{[*]})$}

The next result is an analogue of Proposition~\ref{prop:eta_VsurjectiveDF} but for the associated modules.

\begin{proposition}\label{prop:eta_MsurjectiveDFmodules}
The map
\begin{align*}
\begin{array}{rccc}
 & R^{[*]}(M^{[*]}) & \longrightarrow & \gr_{*} A(M^{[*]}) \\[5pt]
            & m+C_2^{[*]}(V^{[*]})_{\on{wt}(m)} & \longmapsto & \displaystyle m+O(M^{[*]})+\bigoplus_{n< \on{wt}(m)}M_{n}^{[*]}
\end{array}
\end{align*}
is a surjective morphism of differential filtered $R^{[*]}(V^{[*]})$-modules, where $ \gr_{*} A(M^{[*]})$ is seen as an $R^{[*]}(V^{[*]})$-module through $\eta_{W_\bullet}$.
\end{proposition}

Using similar arguments as in the previous section, for a weight graded differential filtered $R^{[*]}(V^{[*]})$-module $W$, the space $\eta_{W_\bullet}(W)= \gr_{*} A(V^{[*]}) \otimes_{R^{[*]}(V^{[*]})} W \cong W/((\mathrm{Ker} \ \eta_{W_\bullet})W)$ is a weight graded differential filtered  $\gr_{*} A(V^{[*]})$-module. With a proof similar to that of Proposition~\ref{prop:nat_F}, we obtain:

\begin{proposition}
There is a natural transformation $\varphi: \eta_{W_\bullet} \circ R^{[*]}(-) \longrightarrow \gr_{*}A(-)$, illustrated in the diagram below:
\begin{center}
 \begin{tikzpicture}[scale=0.9, transform shape]
\tikzset{>=stealth}
\node (1) at (0,0) []{$V^{[*]}\Mod^{gr}$};
\node (2) at (6,0) []{$\gr_{*} A(V^{[*]})\Mod^{gr}$};

\draw[->]  (1) edge[out=30, in=150]  node[above] {$\gr_{*}A(-)$}(2);
\draw[->]  (1) edge[out=-30, in=-150]  node[below] {$\eta_{W_\bullet} \circ R^{[*]}(-)$}(2);
\draw[double distance=1.5pt,-{Stealth[scale=1.2]}] (3,-0.7) -- node[right] {$\varphi$} (3,0.7);
\end{tikzpicture}
\end{center}
such that, for any $M^{[*]} \in V^{[*]}\Mod^{gr}$, the morphism $\varphi_{M^{[*]}}$ is the following surjective morphism of weight graded differential filtered $\gr_{*}A(V^{[*]})$-modules:
\begin{align*}
\begin{array}{ccc}
 \eta_{W_\bullet}(R^{[*]}(M^{[*]})) & \stackrel{\varphi_{M^{[*]}}}{\longrightarrow} & \gr_{*}A(M^{[*]}) \\[5pt]
 m+C_2(M^{[*]})_{\on{wt}(m)}+\left((\on{Ker}  \eta_{W_\bullet})(R^{[*]}(M^{[*]})\right)_{\on{wt}(m)} & \longmapsto & m+O(M^{[*]})+\displaystyle \bigoplus_{n<\on{wt}(m)}M_n^{[*]}.
\end{array}
\end{align*}
\end{proposition}

\subsection{$R^{[*]}(M^{[*]})$  and $\gr_* \gr^{[*]}A(M^{[*]})$}

A result similar to Proposition~\ref{prop:eta_double_grading} but for the associated modules can be obtained by modifying the proof of Proposition~\ref{prop:eta_VsurjectiveDG}:

\begin{proposition}
The map
\begin{align*}
\begin{array}{rccc}
& R^{[*]}(M^{[*]}) & \longrightarrow &\gr_* \gr^{[*]} A(M^{[*]}) \\[5pt]
            & m+C_2^{[|m|]}(V^{[*]})_{\on{wt}(m)} & \longmapsto & \displaystyle m+O(M^{[*]})+\bigoplus_{\substack{q \leq |m| \\ n<\on{wt}(m)}}M_n^{[q]}+\bigoplus_{\substack{q < |m| \\ n \leq \on{wt}(m)}}M_n^{[q]}
\end{array}
\end{align*}
is a surjective morphism of weight graded dg Poisson $R^{[*]}(V^{[*]})$-modules, where $\gr_* \gr^{[*]} A(M^{[*]})$ is seen as an $R^{[*]}(V^{[*]})$-module through $\eta_{F^\bullet W_\bullet}$.
\end{proposition}

Like in Proposition~\ref{prop:nat_F}, for a weight graded dg module $W$ for $R^{[*]}(V^{[*]})$, we can define $\eta_{F^\bullet W_\bullet }(W) = \gr_{*} \gr^{[*]} A(V^{[*]}) \otimes_{R^{[*]}(V^{[*]})} W \cong W/((\mathrm{Ker} \ \eta_{F^\bullet W_\bullet})W)$ and prove:

\begin{proposition}
There is a natural transformation $\Psi: \eta_{F^\bullet W_\bullet} \circ R^{[*]}(-) \longrightarrow \gr_{*} \gr^{[*]} A(-)$, illustrated in the diagram below:
\begin{center}
 \begin{tikzpicture}[scale=0.9, transform shape]
\tikzset{>=stealth}
\node (1) at (0,0) []{$V^{[*]}\Mod^{gr}$};
\node (2) at (6,0) []{$\gr_{*} \gr^{[*]} A(V^{[*]})\Mod^{gr}$};

\draw[->]  (1) edge[out=30, in=150]  node[above] {$\gr_{*} \gr^{[*]}A(-)$}(2);
\draw[->]  (1) edge[out=-30, in=-150]  node[below] {$\eta_{F^\bullet W_\bullet} \circ R^{[*]}(-)$}(2);
\draw[double distance=1.5pt,-{Stealth[scale=1.2]}] (3,-0.7) -- node[right] {$\Psi$} (3,0.7);
\end{tikzpicture}
\end{center}
such that, for any $M^{[*]} \in V^{[*]}\Mod^{gr}$, the morphism $\Psi_{M^{[*]}}$ 
\begin{align*}
\begin{array}{ccc}
 \eta_{F^\bullet W_\bullet}(R^{[*]}(M^{[*]})) & \stackrel{\Psi_{M^{[*]}}}{\longrightarrow} & \gr_{*}\gr^{[*]}A(M^{[*]})
\end{array}
\end{align*}
sending $m+C_2^{[|m|]}(M^{[*]})_{\on{wt}(m)}+\left((\on{Ker}  \eta_{F^\bullet W_\bullet})(R^{[*]}(M^{[*]})\right)^{[|m|]}_{\on{wt}(m)}$ to $m+O(M^{[*]})+\displaystyle \bigoplus_{\substack{q \leq |m| \\ n<\on{wt}(m)}}M_n^{[q]}+\bigoplus_{\substack{q < |m| \\ n \leq \on{wt}(m)}}M_n^{[q]}$
is a surjective morphism of weight graded dg Poisson $\gr_{*} \gr^{[*]} A(V^{[*]})$-modules.
\end{proposition}

\subsection{More information on the functors}

We have obtained the following diagram:

\begin{center}
 \begin{tikzpicture}[scale=1, transform shape]
\tikzset{>=stealth}
\node (1) at (0,0) []{$V^{[*]}\Mod^{gr}$};
\node (2) at (6,5) []{$\gr^{[*]} A(V^{[*]})\Mod^{gr}$};

\node (3) at (6,-5) []{$\gr_{*} A(V^{[*]})\Mod^{gr}$};
\node (4) at (12,0) []{$\gr_{*} \gr^{[*]}A(V^{[*]})\Mod^{gr}$};

\draw[->]  (1) edge[out=80, in=190]  node[above, left] {$\gr^{[*]}A(-)$}(2);
\draw[->]  (1) edge[out=20, in=-120]  node[right] {$\eta_{F^\bullet} \circ R^{[*]}(-)$}(2);
\draw[double distance=1.5pt,-{Stealth[scale=1.2]}]  (3.6,2) -- node[above] {$\psi$} (1.7,3.3);

\draw[->]  (1) edge[out=-80, in=-190]  node[below, left] {$\gr_{*}A(-)$}(3);
\draw[->]  (1) edge[out=-20, in=120]  node[right] {$\eta_{W_\bullet} \circ R^{[*]}(-)$}(3);
\draw[double distance=1.5pt,-{Stealth[scale=1.2]}] (3.6,-2) -- node[below] {$\varphi$} (1.7,-3.3);

\draw[->]  (2) -- node[above, right] {$\gr_{*}(-)$}(4);
\draw[->]  (3) -- node[below, right] {$\gr^{[*]}(-)$}(4);

\draw[->]  (1) edge[out=10, in=170,  looseness=1] node[above] {$\gr_{*} \circ \gr^{[*]}A(-)$} (4);
\draw[->]  (1) edge[out=-10, in=-170, looseness=1] node[below] {$\eta_{F^\bullet W_\bullet} \circ R^{[*]}(-)$}   (4);
\draw[double distance=1.5pt,-{Stealth[scale=1.2]}] (6,-0.5) -- node[right] {$\Psi$} (6,0.5);
\end{tikzpicture}
\end{center}

We can define two other natural transformations. Set $M^{[*]} \in V^{[*]}\Mod^{gr}$ and consider $\psi_{M^{[*]}}$. There are filtrations on $\eta_{F^\bullet} \circ R^{[*]}(M^{[*]})$ and $\gr^{[*]}A(M^{[*]})$ induced by the weight grading on $M^{[*]}$, and we can verify that $\psi_{M^{[*]}}$ preserves these filtrations. Hence we can construct a morphism $\gr_* (\psi_{M^{[*]}})$ between the associated graded algebras, and we obtain the natural transformation
\[
\gr_* \psi : \gr_* \circ \ \eta_{F^\bullet} \circ R^{[*]}(-) \longrightarrow \gr_* \gr^{[*]}A(-)
\]
given by $(\gr_* \psi)_{M^{[*]}}=\gr_* (\psi_{M^{[*]}})$. As $\psi_M^{[*]}$ is surjective, it follows that $(\gr_* \psi)_{M^{[*]}}$ is also surjective. We can do something similar using $\varphi_{M^{[*]}}$ and the filtration coming from the cohomological grading of $M^{[*]}$, and we can construct a natural transformation
\[
\gr^{[*]} \varphi : \gr^{[*]} \circ \ \eta_{W_\bullet} \circ R^{[*]}(-) \longrightarrow \gr_* \gr^{[*]}A(-)
\]
and $(\gr^{[*]} \varphi)_{M^{[*]}}$ is a surjective morphism.

We will need the following lemma:

\begin{lemma}\label{lemma:lift_filtration}
Let $A$ and $B$  be filtered algebras with $F^nA = \{0\}=F^nB$ for $n \ll 0$, and set $f: A \longrightarrow B$ a morphism of filtered algebras satisfying $f(F^nA) \subset F^nB$ for all $n \in \mathbb{Z}$. Then we have:
\begin{center}
$f$ is an isomorphism $\Longleftrightarrow$ $\gr_F f$ is an isomorphism.
\end{center}
\end{lemma}

\begin{proof}
One direction is obvious. For the other direction, assume that $\gr_F f$ is an isomorphism. Set $i$ such that $F^jA=F^jB=\{0\}$ for $j <i$ and $F^iA$ or $F^iB$ is not zero. We have the following diagram:
\begin{center}
 \begin{tikzpicture}[scale=1, transform shape]
\tikzset{>=stealth}
\node (1) at (0,0) []{$0$};
\node (2) at (2,0) []{$F^{i-1}A$};
\node (3) at (4,0) []{$F^iA$};
\node (4) at (6.5,0) []{$F^iA/F^{i-1}A$};
\node (5) at (9,0) []{$0$};
\draw[->]  (1) -- (2);
\draw[->]  (2) -- (3);
\draw[->]  (3) -- (4);
\draw[->]  (4) -- (5);

\node (6) at (0,-2) []{$0$};
\node (7) at (2,-2) []{$F^{i-1}B$};
\node (8) at (4,-2) []{$F^iB$};
\node (9) at (6.5,-2) []{$F^iB/F^{i-1}B$};
\node (10) at (9,-2) []{$0$};
\draw[->]  (6) -- (7);
\draw[->]  (7) -- (8);
\draw[->]  (8) -- (9);
\draw[->]  (9) -- (10);

\draw[->]  (2) -- node[right] {$f$} (7);
\draw[->]  (3) -- node[right] {$f$} (8);
\draw[->]  (4) -- node[right] {$\gr_F f$} (9);
\end{tikzpicture}
\end{center}
where the rows are exact. The left most down arrow is $\{0\} \longrightarrow \{0\}$ and the right most down arrow is an isomorphism of vector spaces by assumption. By the Five Lemma, $f : F^iA \longrightarrow F^iB$ is an isomorphism of vector spaces. By doing an induction on $i$ and repeating this reasoning, we see that $f :A \longrightarrow B$ is an isomorphism of algebras.
\end{proof}

We now consider $V^{[*]}$ as the adjoint $V^{[*]}$-module.
\begin{theorem}
Let $V^{[*]}$ be a dg vertex operator algebra. Then $\Psi_{V^{[*]}}$, $(\gr_* \psi)_{V^{[*]}}$, $(\gr^{[*]} \varphi)_{V^{[*]}}$ and  $\psi_{V^{[*]}}$ are isomorphisms. If $V^{[p]}=\{0\}$ for $p \ll 0$, then $\varphi_{V^{[*]}}$ is an isomorphism.
\end{theorem}

\begin{proof}
As we consider the adjoint module, all morphisms in the theorem are in fact endomorphisms. From the definition of $V^{[*]}$, we know that $\on{dim} V_n^{[p]}$ is finite for all $n$ and $p$, so each summand of $\gr_* \gr^{[*]}A(V^{[*]})$ is finite dimensional. As $\Psi_{V^{[*]}}$ preserves the gradation and is surjective, it follows that it is an isomorphism. The same is true for $(\gr_* \psi)_{V^{[*]}}$ and $(\gr^{[*]} \varphi)_{V^{[*]}}$. In the definition of a dg vertex operator algebra, we have $V_n^{[*]}=\{0\}$ for $n \ll 0$. We can apply Lemma~\ref{lemma:lift_filtration} to $(\gr_* \psi)_{V^{[*]}}$ and see that $\psi_{V^{[*]}}$ is an isomorphism. If  $V^{[p]}=\{0\}$ for $p \ll 0$, then the same reasoning shows that $\varphi_{V^{[*]}}$ is an isomorphism.
\end{proof}

\delete{
\begin{definition}
A graded dg module $M^{[*]}$ for a dg vertex operator algebra $V^{[*]}$ is called dg compact if for any $p \in \mathbb{Z}$, $n \in \frac{1}{2}\mathbb{Z}$, we have $\on{dim} M_n^{[p]} < \infty$. 
\end{definition}}

\end{document}